\documentclass{amsart}

\usepackage{amsmath}
\usepackage{amssymb}
\usepackage{amsthm}
\usepackage[all]{xypic}
\usepackage{amsfonts}

\usepackage{tikz-cd}

\usepackage{verbatim,color}
\usepackage{url}
\usepackage{enumitem}
\usepackage{hyperref}
\numberwithin{equation}{section}
\theoremstyle{plain}
\newtheorem{theorem}[equation]{Theorem}
\newtheorem{proposition}[equation]{Proposition}
\newtheorem{corollary}[equation]{Corollary}
\newtheorem{lemma}[equation]{Lemma}
\newtheorem{question}[equation]{Question}

\theoremstyle{definition}
\newtheorem{definition}[equation]{Definition}

\newtheorem{example}[equation]{Example}

\newtheorem{remark}[equation]{Remark}

\newcommand{\hdet}{\operatorname{hdet}}

\newcommand{\beq}{\begin{equation}}
\newcommand{\eeq}{\end{equation}}

\DeclareMathOperator{\id}{{id}}
\DeclareMathOperator{\Hom}{{Hom}}
\DeclareMathOperator{\RHom}{RHom}

\DeclareMathOperator{\End}{{End}}
\DeclareMathOperator{\Ext}{{Ext}}

\DeclareMathOperator{\Tor}{Tor}

\DeclareMathOperator{\im}{im}

\DeclareMathOperator{\GK}{GKdim}

\DeclareMathOperator{\soc}{soc}
\DeclareMathOperator{\ann}{ann}

\newcommand{\grHom}{\underline{\Hom}}

\newcommand{\grExt}{\underline{\Ext}}
\DeclareMathOperator{\RgrHom}{R\grHom}
\DeclareMathOperator{\grtimes}{\underline{\otimes}}

\DeclareMathOperator{\pdim}{pdim}
\DeclareMathOperator{\grpdim}{gr.pdim}
\DeclareMathOperator{\gldim}{gldim}
\DeclareMathOperator{\grgldim}{gr.gldim}

\newcommand{\mc}{\mathcal}

\newcommand{\Z}{\mathbb{Z}}

\newcommand{\mb}{\mathbb}

\newcommand{\Vhat}{\widehat{V}}

\DeclareMathOperator{\rgr}{gr-\!}
\DeclareMathOperator{\lgr}{\!-gr}
\DeclareMathOperator{\lGr}{\!-Gr}
\DeclareMathOperator{\rGr}{Gr-\!}

\DeclareMathOperator{\lMod}{\!-Mod}
\DeclareMathOperator{\rMod}{Mod-\!}

\newcommand{\on}{\operatorname}

 \newcommand{\M}{\mathbb{M}}
 \newcommand{\N}{\mathbb{N}}
 \newcommand{\separate}{\bigskip}

\newcommand{\X}{\mathcal{X}}
\newcommand{\Y}{\mathcal{Y}}
 
 \newcommand{\op}{^{\mathrm{op}}}

\setenumerate{label=\textnormal{(\arabic*)}}

\begin{document}
\title[Graded twisted CY algebras are generalized AS regular]{Graded twisted Calabi-Yau algebras are generalized Artin-Schelter regular} 
%{A twisted Calabi-Yau toolkit}
\author{Manuel L. Reyes}
\address{University of California,  Irvine\\ Department of Mathematics\\
340 Rowland Hall\\ Irvine, CA 92697-3875\\ USA}
\email{mreyes57@uci.edu}

\author{Daniel Rogalski}
\address{University of California, San Diego\\ Department of Mathematics\\ 9500 Gilman Dr. \# 0112 \\
La Jolla, CA 92093-0112\\ USA}
\email{drogalski@ucsd.edu}

\date{\today}

\thanks{Reyes was partially supported by the NSF grant DMS-1407152.
Rogalski was partially supported by the NSF grant 
DMS-1201572 and the NSA grant H98230-15-1-0317.}


\begin{abstract}  
This is a general study of twisted Calabi-Yau algebras that are $\N$-graded and locally finite-dimensional, with the
following major results.  We prove that a locally finite graded algebra is twisted Calabi-Yau if and only if 
it is separable modulo its graded radical and satisfies one of several suitable generalizations of the Artin-Schelter regularity property,
adapted from the work of Martinez-Villa as well as Minamoto and Mori.
We characterize twisted Calabi-Yau algebras of dimension~0 as separable $k$-algebras, and we
similarly characterize graded twisted Calabi-Yau algebras of dimension~1 as tensor algebras of certain
invertible bimodules over separable algebras. 
Finally, we prove that a graded twisted Calabi-Yau algebra of dimension~2 is noetherian if and only if it has
finite GK dimension.
\end{abstract}

 \maketitle
 
 \setcounter{tocdepth}{1} %Omit subsections to keep contents short
 \tableofcontents

\section{Introduction}
\label{sec:intro}

Throughout this paper we let $k$ denote a field, on which we place no further assumptions.
By a \emph{graded} ring  we always mean an $\mathbb{N}$-graded ring 
$A = \bigoplus_{n=0}^\infty A_n$. A graded $k$-algebra $A$ is \emph{locally finite} if each 
$A_n$ is a finite-dimensional $k$-vector space, and it is \emph{connected} if $A_0 = k$.

%
%A familiar theme in noncommutative algebraic geometry is to identify a class of noncommutative
%algebras that are well-behaved for geometrically motivated reasons. One such class of much
%recent interest is that of \emph{Calabi-Yau algebras}~\cite{G}. These are algebras $A$ for which
%there is a suitably nice resolution of $A$ by projective bimodules.
%There are now many examples and constructions of these algebras, including preprojective 
%algebras~\cite{Bo, G} and their higher versions~\cite{Keller, HIO, AmiotOppermann},  
%quivers with potentials from dimer models~\cite{Broomhead}, 
%

%One of the most famous
%such classes is that of \emph{Artin-Schelter regular algebras}~\cite{AS}, which are connected graded
%algebras with suitably nice projective resolutions of $k = A/A_{\geq 1}$ as a left and right $A$-module.
%Another class of (not necessarily graded) algebras that has rapidly come to the forefront of such
%study is that of \emph{Calabi-Yau algebras}~\cite{G}, which have a suitably nice resolution of $A$ 
%by projective bimodules.
%
%\todo{talk about significance of CY algebras here: with their DG analogues, play a role in cluster theory; 
%sources of examples and constructions (referee): Auslander's non-singular orders, Van den Berg's NCCRs, quivers with potentials from dimer models, classical and higher preprojective algebras,..
%Important to drop connectedness assumption in these examples.}
%
%

A familiar theme in noncommutative algebraic geometry is to identify a class of noncommutative
algebras that are well-behaved for geometrically motivated reasons. One of the most famous
such classes is that of \emph{Artin-Schelter regular algebras}~\cite{AS}, which are connected graded
algebras with suitably nice projective resolutions of $k = A/A_{\geq 1}$ as a left and right $A$-module.
This class is known to include all connected graded Auslander-regular algebras~\cite{Levasseur}, and thus includes 
many quantum groups and related quantized algebras~\cite[I.15]{BG}.
Another class of (not necessarily graded) algebras that has rapidly come to the forefront of such
study is that of \emph{Calabi-Yau algebras}~\cite{G}, which have a suitably nice resolution of $A$ 
by projective bimodules.
These include preprojective algebras~\cite{Bo, G} and their higher versions~\cite{Keller:deformed, HIO, AmiotOppermann},
as well as quivers with potentials arising from dimer models~\cite{Broomhead}. 

These two classes of algebras are unified by the class of \emph{twisted} (or ``\emph{skew}'') Calabi-Yau
algebras; indeed, it was shown in~\cite[Lemma~1.2]{RRZ1} that a connected graded algebra $A$ is 
twisted Calabi-Yau if and only if it is Artin-Schelter regular.  However, many interesting examples of
(graded) Calabi-Yau algebras are not connected. In this paper we undertake a careful study of graded 
twisted Calabi-Yau algebras that are not necessarily connected, with the goal of explaining their
relationship to a suitable generalization of Artin-Schelter regular algebras.

We now recall the definitions of these
classes of algebras, beginning with Artin-Schelter regular algebras. We emphasize that in contrast to
various other authors, we do not require regular algebras to have finite Gefland-Kirillov (GK)-dimension.  
\begin{definition}
\label{def:connected AS regular}
Let $A$ be a connected graded algebra and let $k = A/A_{\geq 1}$ be the trivial module.  Then $A$ is \emph{Artin-Schelter (AS) regular of dimension~$d$} if it has
graded global dimension~$d$ and satisfies
\[
\Ext^i_A(k,A) \cong 
\begin{cases}
0, & i \neq d \\
k, & i = d
\end{cases} \quad \mbox{and} \quad
\Ext^i_{A\op}(k,A) \cong 
\begin{cases}
0, & i \neq d \\
k, & i = d
\end{cases}
\]
in $\rMod A$ and $A \lMod$, respectively.
\end{definition}

Next we recall the definition of twisted Calabi-Yau algebras. In some sources, this term refers to a twist by
an automorphism as in condition~(iii) below (for instance, see~\cite{BSW}); we will consider the more general
twist by an invertible bimodule defined, for instance, in~\cite[Definition~3.7.9]{Schedler}. 
For a $k$-algebra $A$ with opposite $A\op$, we write $A^e$ for its \emph{enveloping algebra} 
$A \otimes_k A\op$.  Then a $k$-central $(A, A)$-bimodule $M$ is also a left $A^e$-module or a right 
$A^e$-module, where $(a \otimes b\op) \cdot m = amb = m \cdot (b \otimes a\op)$.  In this way we 
can identify the category of $(A, A)$-bimodules with either $A^e \lMod$ or $\rMod A^e$.  Given an 
automorphism $\mu: A \to A$, we write ${}^1 A^\mu$ for the $(A, A)$-bimodule structure on $A$ 
where $b \cdot a \cdot c = ba\mu(c)$.  Recall in addition 
that a left $A$-module $M$ is called \emph{perfect} if it has a finite length projective resolution 
consisting of finitely generated projective modules.

\begin{definition}\label{def:twisted CY}
Let $A$ be a $k$-algebra. We say that:
\begin{enumerate}[label=(\roman*)]
\item $A$ is \emph{homologically smooth over~$k$} (or a \emph{homologically smooth $k$-algebra}) 
if $A$ has a resolution of finite length by finitely generated projective left $A^e$-modules 
(that is, if $A$ is perfect as a left $A^e$-module);
\item $A$ is \emph{twisted Calabi-Yau (of dimension $d$)} if it is homologically smooth over~$k$ and 
if there is an invertible $k$-central $(A,A)$-bimodule $U$ such that
\[
\Ext^i_{A^e}(A,A^e) \cong 
\begin{cases}
0, & i \neq d \\
U, & i = d
\end{cases}
\]
as $(A,A)$-bimodules, where each $\Ext^i_{A^e}(A,A^e)$ is considered as a right $A^e$-module
via the right $A^e$-structure of $A^e$;
\item $A$ has \emph{Nakayama automorphism} $\mu$ if the isomorphism in~(ii) holds with $U = {}^1 A^\mu$.
\end{enumerate}
\end{definition}

The use of the term ``dimension'' is justified in Lemma~\ref{lem:derived form} below, where
it is shown that the dimension $d$ of a twisted Calabi-Yau algebra is equal to the projective dimension of $A$ as a left $A^e$-module.
We call the module $U = \Ext^d_{A^e}(A,A^e)$ the \emph{Nakayama bimodule} for $A$.
It is well-known that the Nakayama automorphism of a twisted Calabi-Yau algebra, if it exists,
is unique up to multiplication by an inner automorphism of $A$. 
A \emph{Calabi-Yau algebra} is a twisted Calabi-Yau algebra as above for which the Nakayama
bimodule is $U \cong A$; equivalently, it is a twisted Calabi-Yau algebra with an inner Nakayama 
automorphism.

The main motivation leading to the present paper is the problem of classifying twisted Calabi-Yau 
algebras $A$ of dimension~$d$ (for small values of~$d$) that are graded homomorphic images of a 
path algebra $kQ$ of a quiver $Q$, so that they are not necessarily connected.
In our work with twisted Calabi-Yau algebras of dimension~2 and~3, in order to compute potential 
Hilbert series, it became useful to work with projective resolutions of the \emph{left} module
$S = A/J(A)$, rather than resolutions of the \emph{bimodule} $A$. 
In the case of connected graded algebras, the aforementioned equivalence between the AS regular and 
twisted Calabi-Yau properties means that in order to check the twisted Calabi-Yau property, it is 
sufficient to compute a particular left module resolution. 
Thus for non-connected algebras of the form $A = kQ/I$, we were led to wonder whether the twisted 
Calabi-Yau property might similarly be equivalent to a suitable generalization of the AS~regular 
property, defined in terms of projective resolutions 
of $S$ as a left or right $A$-module.   Our main theorem, which we present as Theorem~\ref{thm:twisted CY equivalence} below, 
will give a precise result along these lines.

Another goal of this paper is to give full details of the proofs of some basic properties of homologically smooth algebras 
and twisted Calabi-Yau algebras, in particular the stability of these properties under common constructions, especially in the graded case.  
Some of these results may be folklore, but it seems useful to have them written down in one place.   We begin in Section~\ref{sec:preliminary} 
with a review of preliminary results on locally finite graded algebras 
and their modules, including information about minimal projective resolutions, idempotent decompositions, and some isomorphisms in their derived categories of modules.  Then in Sections~\ref{sec:smooth} and~\ref{sec:twisted CY} we study homological smoothness and the twisted Calabi-Yau property, respectively.  In both cases, we show that these properties are stable 
under finite direct sum of algebras, tensor product of algebras, base field extension, and Morita equivalence.  We also review several of the different ``Serre duality" formulas that hold for modules over twisted Calabi-Yau algebras.

In case $A$ is graded with $\dim_k A_0 < \infty$, we let $J(A)$ be its graded Jacobson radical; then 
$A_{\geq 1} \subseteq J(A)$ and $S = A/J(A) = A_0/J(A_0)$ is a finite-dimensional semisimple algebra.   
In our study of homological smoothness in Section~\ref{sec:smooth}, we obtain the following novel 
characterization of homological smoothness for graded $k$-algebras.  Recall that a $k$-algebra 
$B$ is \emph{separable} (over $k$) if $B \otimes_k K$ is semisimple for all field extensions 
$k \subseteq K$. 
\begin{theorem}
\label{thm:main1} 
(Theorem~\ref{thm:homologically smooth})
Let $A$ be graded with $\dim_k A_0 < \infty$ as above.  Then the following are equivalent:
\begin{enumerate}
\item $A$ is homologically smooth over $k$;
\item $S$ is separable as a $k$-algebra and perfect as a left (or equivalently, right) $A$-module.
\end{enumerate}
\end{theorem}
\noindent
The proof that homological smoothness implies separability of $S$ is due to Jeremy Rickard, and we thank him for allowing
us to include it here.

An important precedent for the study of generalized AS~regular properties for locally finite graded
algebras that are not necessarily connected was set in the work of Martinez-Villa and
Solberg in~\cite{MV, MVS} and of Minamoto and Mori in~\cite{MM}; we recall these definitions in 
Section~\ref{sec:AS regular}.   Another possibility we would like to highlight is the following definition, 
which is formally similar to the original definition of AS~regular algebras and tends to be one of the 
easier ones to work with technically.
\begin{definition}
\label{def:Sreg}
Let $A$ be a locally finite graded $k$-algebra, with $S = A/J(A)$.  We say that $A$ is 
\emph{generalized AS~regular of dimension $d$} if $A$ has graded global dimension $d$ and 
there is a $k$-central invertible $(S,S)$-bimodule $V$ such that
\[
\Ext_A^i(S, A) \cong \begin{cases} 0, & i \neq d \\ V, & i = d \end{cases}
\]
as $(S, S)$-bimodules.
\end{definition} 
In Section~\ref{sec:AS regular}, we show that this definition, the definition of Martinez-Villa 
and Solberg, and a slightly modified version of the definition of Minamoto and Mori, all give equivalent 
notions.   Indeed, in Theorem~\ref{thm:reg char} we show that these and even several additional 
slight variations are equivalent.  For example, it is equivalent to require the isomorphism in 
Definition~\ref{def:Sreg} to hold as right $S$-modules only.

Consequently we arrive at the main theorem of the paper, which shows that the twisted Calabi-Yau condition is 
related to generalized AS~regularity in a precise way (and often they are equivalent).
Our hope is that this result will allow for an interplay between results in noncommutative
Calabi-Yau geometry and the theory of AS regular algebras for algebras that are not necessarily
connected. 
\begin{theorem} (Theorem~\ref{thm:twisted CY equivalence})
Let $A$ be a locally finite graded algebra, and denote $S = A_0/J(A_0)$.
Then the following are equivalent:
\begin{enumerate}[label=\textnormal{(\alph*)}]
\item $A$ is twisted Calabi-Yau of dimension~$d$;
\item $A$ is generalized AS regular of dimension~$d$, and $S$ is a separable $k$-algebra;
\item For every field extension $K$ of $k$, the $K$-algebra $A \otimes K$ is generalized AS regular
of dimension~$d$.
\end{enumerate}
\end{theorem}
\noindent In particular, the result above shows that if one is working over a perfect base field $k$ (for 
example, an algebraically closed field or a field of characteristic~$0$) then the twisted Calabi-Yau and 
generalized AS~regular conditions are equivalent for locally finite graded algebras. This also recovers
the equivalence~\cite[Lemma~1.2]{RRZ1} of the two properties in the case where $A$ is connected
graded, because if $A$ is connected then $S \cong k$ is separable over $k$. 

The following diagram illustrates the relationships between several of the properties
studied in this paper for locally finite graded $k$-algebras $A$ with $S = A/J(A)$. The implications all follow from the
theorems above, with the exception of the far right arrow which follows from the proof of Theorem~\ref{thm:reg char}.
\begin{center}
\begin{tikzcd}
\textnormal{twisted CY} \ar[r, Leftrightarrow]  \ar[d, Rightarrow] 
 & \textnormal{\begin{tabular}{c} generalized AS~regular \\ and $S$ separable\end{tabular}} \ar[r, Rightarrow] \ar[d, Rightarrow]
 & \textnormal{\begin{tabular}{c} generalized \\ AS~regular \end{tabular}} \ar[d, Rightarrow] \\
 
\textnormal{\begin{tabular}{c} homologically \\ smooth \end{tabular}} \ar[r, Leftrightarrow] 
 & \textnormal{\begin{tabular}{@{}c@{}} ${}_A S$ perfect \\ and $S$ separable\end{tabular}} \ar[r, Rightarrow] 
 & \textnormal{${}_A S$ perfect}
\end{tikzcd}
\end{center}

In the final section of the paper, Section~\ref{sec:applications}, we apply our earlier results to obtain 
information about twisted Calabi-Yau algebras of dimension at most two. 
In dimension zero, we have a simple characterization of (not necessarily graded) twisted Calabi-Yau algebras as follows.
\begin{theorem} (Theorem~\ref{thm:CY 0})
\label{thm:CY0 intro} 
For an algebra $A$, the following are equivalent:
\begin{enumerate}[label=\textnormal{(\alph*)}]
\item $A$ is twisted Calabi-Yau of dimension~$0$;
\item $A$ is twisted Calabi-Yau and a finite-dimensional algebra;
\item $A$ is Calabi-Yau of dimension~$0$;
\item $A$ is a separable $k$-algebra.
\end{enumerate} 
\end{theorem}

In dimension one, we are able to characterize graded twisted Calabi-Yau algebras as certain tensor 
algebras and to prove that they are noetherian. 
\begin{theorem} (Theorem~\ref{thm:CY1}, Corollary~\ref{cor:noetherian dim 1})
\label{thm:CY1 intro}
Let $A$ be a locally finite graded algebra, and suppose that the algebra $S = A_0/J(A_0)$ is separable.
Then the following are equivalent:
\begin{enumerate}[label=\textnormal{(\alph*)}]
\item $A$ is twisted Calabi-Yau of dimension~$1$;
\item There is an invertible, nonnegatively graded, finite-dimensional $k$-central 
$(S,S)$-bimodule $V$ such that $T_S(V) \cong A$.
\end{enumerate}
Furthermore,  every algebra $A$ satisfying the conditions above is noetherian.
\end{theorem}

In future work, we hope to present a fuller picture of the structure of locally finite graded twisted 
Calabi-Yau algebras of dimension~$2$.  Here we will focus solely on the question of when such algebras have
the noetherian property. In noncommutative algebraic geometry, it is generally expected that 
noncommutative graded algebras of a geometric nature will exhibit the best ring-theoretic properties 
when they have finite GK-dimension.  In keeping with this theme, we prove the following. 
\begin{theorem} (Theorem~\ref{thm:noetherian dim 2})
\label{thm:CY2 intro}
Let $A$ be a locally finite graded twisted Calabi-Yau algebra of dimension~2.  Then $A$ is noetherian 
if and only if $A$ has finite GK-dimension.
\end{theorem}
\noindent 
We wish to emphasize that the proofs of the results above stating that an algebra is noetherian are all
of a ``structural'' nature, and do not rely upon a  classification of such algebras to determine that they 
are noetherian.  

We note that while we have restricted our present study to graded algebras, it seems likely that a number of results
and techniques employed here should generalize to the setting of semilocal twisted Calabi-Yau algebras
$A$ that are ``locally finite'' in the sense that, for the Jacobson radical $J$, the $k$-algebras $A/J^n$ 
are finite-dimensional for all $n \geq 0$.

The companion paper~\cite{RR2} focuses on the GK-dimension of locally finite graded twisted 
Calabi-Yau algebras.  It includes fundamental results about the matrix Hilbert series of such 
algebras, the basic structure of the generators and relations for those of dimension 2 and 3, and 
techniques for distinguishing which of these algebras have finite GK-dimension.  We note that the 
proof of Theorem~\ref{thm:CY2 intro} depends on a technical result from \cite{RR2}.

Several routine proofs in this paper are sketched or omitted. Readers who wish to see full proofs may consult an earlier draft of this paper at \href{https://arxiv.org/abs/1807.10249}{\texttt{arXiv:1807.10249v1}}.

\textbf{Acknowledgements.}
We would like to thank James Zhang for helpful conversations and Stephan Weispfenning for helpful 
comments about earlier drafts of this paper.  We are particularly grateful to Jeremy Rickard 
for allowing us to use his ideas in the proof of Theorem~\ref{thm:main1}.
Finally, we thank the referees for their careful reading of this paper, providing a number of suggestions that have improved the readability of this paper as well as interesting suggestions for future problems.

\section{Preliminaries on graded algebras and their modules}
\label{sec:preliminary}

This section collects a number of preparatory results for our treatment of graded twisted Calabi-Yau
algebras with the goal of improving readability of proofs in later sections.  
Readers who feel so inclined are encouraged to browse through the next few paragraphs on notations and conventions and then proceed directly to Section~\ref{sec:smooth}, referring back to these preparatory results only as needed. 

In some parts we collect known results, and in such cases we point to proofs in the literature. 
In others, we provide generalizations of common arguments for connected graded algebras to locally finite algebras. While these generalizations are relatively routine, they involve some subtleties that are easily overlooked if one is accustomed to working only with connected graded algebras. 
We also require extensions of certain facts about modules to objects in the derived category. In such cases where proofs are not available in the literature, we include some of the shorter proofs but omit longer or more tedious ones.

\separate

We begin by discussing some conventions on rings, modules, and complexes.
All rings, homomorphisms, and modules are assumed to be unital.
For a ring $A$, we write $A \lMod$ for its category of left modules and $\rMod A$ for its category of right $A$-modules. 
Letting $A\op$ denote the opposite ring of $A$, we have $\rMod A = A\op \lMod$.
We use the notation $\Hom_A(-,-)$ for Hom-sets of \emph{left} $A$-modules,
and $\Hom_{A\op}(-,-)$ for \emph{right} modules. This will occasionally clarify
any (bi)module structures induced on Hom groups in which one or both arguments are
$(A,A)$-bimodules. (Recall that if ${}_A M_B$ and ${}_A N_C$ are bimodules for rings $B$ 
and $C$, then $\Hom_A(M_B,N_C)$ carries an induced structure of a $(B,C)$-bimodule,
which is most easily seen if one allows this $\Hom$ group to act on the right, opposite from
the scalars.  Similarly, if ${}_B P_A$ and ${}_C Q_A$ are bimodules, then $\Hom_{A\op}({}_B P, {}_C Q)$ 
is a $(C,B)$-bimodule.)   If $A$ and $B$ are $k$-algebras, then all $(A,B)$-bimodules $M$ considered 
below are assumed to be $k$-central (that is to say, the satisfy $\lambda m = m \lambda$ for every 
module element $m \in M$ and scalar $\lambda \in k$) unless explicitly stated otherwise.

It will be useful to us to phrase the proofs of some of our main results in terms of derived categories, 
so we recall the relevant notation.  Our convention is that all complexes will be cohomological 
complexes, as is standard when working with the derived category.  Let $A$ be a $k$-algebra. To 
indicate a  complex of $A$-modules
\[
\dots \to P^{i-1} \to P^i \to P^{i+1} \to \dots
\] 
we will use the notation $P^{\bullet}$ or the shorter form $P$.
As a consequence of this cohomological notation, for modules $M_A, L_A,$ and ${}_A N$, 
we have $\Ext^i_A(M,L) = H^i(\RHom_A(M,L))$ and $\Tor_i^A(M,N) = H^{-i}(M \otimes_A^L N)$.
For an abelian category $\mc{C}$, we write $D(\mc{C}), D^+(\mc{C}), D^-(\mc{C})$, and $D^b(\mc{C})$ 
to respectively denote the derived categories of complexes, bounded below 
complexes, bounded above complexes, and bounded complexes. We also write 
$D(A) = D(A \lMod)$ and $D(A\op) = D(\rMod A)$, and similarly for the other derived categories.
For references on derived categories and the use of $\RHom$ functors, we refer readers
to~\cite[Chapter~10]{We} and~\cite[Tag~0A5W]{StacksProject}.
Given two complexes $P^\bullet$ and $Q^\bullet$, we will regularly write $P^\bullet \cong Q^\bullet$
to denote an isomorphism in the derived category $D(A)$, which is to say a quasi-isomorphism;
we trust that this will not cause undue confusion, as we rarely care if two complexes are isomorphic
in the (homotopy) category of complexes $K(A)$.

Given a ring $R$, we say that a complex of left $R$-modules is \emph{perfect}
if it is quasi-isomorphic to a bounded complex of finitely generated projective
left $R$-modules~\cite[Tag~0656]{StacksProject}. A left $R$-module $M$ is called 
\emph{perfect} if it is perfect when considered as a complex with a single term in degree zero; 
this is equivalent to the existence of a projective resolution of finite length
\begin{equation}\label{eq:perfect}
0 \to P^{-n} \to \cdots \to P^{-1} \to P^0 \to M \to 0
\end{equation}
where all $P^i$ are finitely generated. 

%\separate

\subsection{Graded algebras, modules, and resolutions}

We now turn to conventions and basic results on graded algebras, modules, and resolutions.
Unless explicitly indicated otherwise, by a \emph{graded ring (or algebra)} we mean an
$\N$-graded ring (or algebra). 
Suppose that $A = \bigoplus_{n=0}^\infty A_n$ is a graded algebra.   We write $\rGr A$ for the category of 
graded right $A$-modules, with morphisms the degree-preserving homomorphisms, which we refer to as \emph{graded} homomorphisms. 
 Similarly, we write $A \lGr$ for the category of graded left $A$-modules.  Given a graded module $M$ and $n \in \mb{Z}$, $M(n)$ will indicate the same module $M$ with the grading shifted so that $M(n)_m = M_{m+n}$ for all $m$.
Given graded left $A$-modules 
$M$ and $N$, let $\Hom_A^i(M,N)$ denote the vector space of left $A$-module homomorphisms
$\phi \colon M \to N$ that are homogeneous of degree~$i$ (that is, $\phi(M_n) \subseteq N_{n+i}$).
Continue to write $\Hom_A(M,N)$ for the usual Hom in the category $A\lMod$. Then we have 
the graded Hom-groups 
\[
\grHom_A(M,N) = \bigoplus_{i=-\infty}^\infty \Hom_A^i(M,N) \subseteq \Hom_A(M,N).
\]
It is well-known that if $M$ and $N$ are
graded left $A$-modules with $M$ finitely generated, then in fact $\grHom_A(M,N) = \Hom_A(M,N)$; 
see~\cite[Corollary~2.4.4]{NV2}, for instance.  We will extend this result to the derived category in Lemma~\ref{lem:graded equivalence} below.  We say that $M$ is \emph{graded perfect} if there is a graded projective resolution as in~\eqref{eq:perfect} with all $P^i$ finitely generated. We denote the right derived functors
of the graded Hom functor by $\grExt^i_A = R^i\grHom_A$, and similarly we denote the total
derived functor by $\RgrHom_A$.

Given graded modules $M_A$ and ${}_A N$ over the graded algebra $A$, we recall 
from~\cite[p.~12]{NV1} that the \emph{graded tensor product} $M \grtimes_A N$ is the graded 
$k$-vector space that has the same underlying vector space as the usual tensor product 
$M \otimes_A N$, with the grading induced by declaring each pure tensor of the form $m \otimes n$, 
where $m$ and $n$ are homogeneous, to be itself homogeneous of degree 
$\deg(m \otimes n) = \deg(m) + \deg(n)$.

For a graded ring $A$, we let $J(A)$ denote its \emph{graded Jacobson radical}, the
intersection of its maximal homogeneous left (equivalently, right) ideals.  If $I$ is a maximal 
homogeneous left ideal, then $1 \not \in I$; since $1 \in A_0$, we also have 
$1 \not \in I_0 + A_{\geq 1} = I + A_{\geq 1}$.  By maximality, $A_{\geq 1} \subseteq I$.  
It follows that
\[
J(A) = J(A_0) \oplus A_{\geq 1},
\] 
where $A_0$ is considered as a graded ring concentrated in degree zero, so that $J(A_0)$ is 
the usual Jacobson radical (see also ~\cite[Corollary~II.6.5]{NV1}).   
Thus we have $A/J(A) \cong A_0/J(A_0)$ as rings.
Note that if $A_0$ is semisimple, then $J(A) = A_{\geq 1}$ and $A/J(A) \cong A_0$.  We always 
write $S$ for the semisimple algebra $A/J(A)$.

We say that a graded $k$-algebra $A = \bigoplus_{n = 0}^\infty A_n$ is \emph{locally finite}
(elsewhere called \emph{finitely graded}) if each $A_n$ is finite-dimensional as a $k$-vector space.   
While this will eventually be our 
case of interest, we are able to develop many of our results assuming only that $A_0$ is
finite-dimensional.
If $A$ is a graded algebra with $A_0$ finite-dimensional, then every left $A_0$-module
has a projective cover~\cite[Proposition~24.12]{FC}. From this it follows that every 
graded left $A$-module $M$ that is bounded below (i.e., $M_n = 0$ for $n \ll 0$)
has a \emph{graded projective cover} 
$f \colon P \to M$, in the sense that $P$ is graded projective, $f$ is a graded surjective morphism,
and $\ker(f) \subseteq J(A)P$ (see also~\cite[p.~4064]{MM}).  We define a \emph{minimal graded projective resolution} of a graded left $A$-module to be a graded projective resolution $P^{\bullet} \to M$ such that each map $d^n: P^n \to P^{n-1}$
in the complex, as well as the augmentation map $\epsilon: P^0 \to M$, is a projective cover of its image.
This is equivalent to the conditions $\ker(d^n) \subseteq J(A) P^n$ for all $n \geq 1$ together with $\ker(\epsilon) \subseteq J(A) P^0$.
Alternatively, this can be expressed as the conditions that $\im(d^n) \subseteq J(A) P^{n-1}$ for all $n \geq 1$.

Continuing to assume that $\dim_k A_0 < \infty$, then as in~\cite[Lemma~2.6]{MM}, every graded projective left $A$-module
$P$ that is bounded below is of the form $P \cong A \otimes_{A_0} Q$ for some
graded projective left bounded left $A_0$-module $Q$. Furthermore, as every projective left module
over the finite-dimensional algebra $A_0$ is a direct sum of indecomposable projective modules,
each isomorphic to $A_0 e$ where $e$ is a primitive idempotent \cite[6.3]{P}, we have in fact
$P \cong \bigoplus Ae_i(l_i)$ for some primitive idempotents $e_i \in A_0$ and integers 
$l_i$.

We say that a graded module $M$ is \emph{graded-indecomposable} if one
cannot write $M = M_1 \oplus M_2$ for nonzero graded submodules $M_i$.
Note that if $e \in A_0$ is a primitive idempotent, then $Ae$ is a graded-indecomposable
projective. Thus every bounded below graded projective left $A$-module is (uniquely)
a direct sum of graded-indecomposable projectives.
It follows as in~\cite[Section~2]{MM} that any graded $A$-module that is bounded
below has a minimal graded projective resolution whose terms are bounded below. 
One may check using the graded Nakayama lemma below 
that such a minimal resolution is unique up to (a generally non-unique) isomorphism, 
as expected.

The following graded version of Nakayama's lemma is well-known.
A proof of~(1) can be found in~\cite[Lemma 2.1]{MM}, and the derivation of~(2) from~(1) is 
standard.  
\begin{lemma}\label{lem:Nakayama}
Let $A$ be a graded $k$-algebra whose degree zero part $A_0$ has finite $k$-dimension, 
and let $M$ be a bounded-below graded left $A$-module.
Let $I \subseteq A$ be a graded ideal contained in $J(A)$.
\begin{enumerate}
\item If $M = IM$, then $M = 0$.
\item Any lift of a generating set of $M/IM$ is a generating set of $M$.
\end{enumerate}
\end{lemma}

%\begin{proof}
%(1) Assume for contradiction that $M \neq 0$, and fix $n$ minimal such that $M_n \neq 0$. Then
%from $M = JM$ and the fact that $A_{\geq 1} M_n \subseteq \bigoplus_{i > n} M_i$, we see that
%the $A_0$-module $M_n$ must satisfy $M_n = (J \cap A_0) \cdot M_n$. 
%From $J \subseteq J(A)$ we find that $(J \cap A_0) \subseteq J(A_0)$ is nilpotent. Thus we obtain 
%the contradiction $M_n = 0$.
%
%(2) Fix a generating set $\overline{X} \subseteq M/JM$, and let $X \subseteq M$ be any set of lifts
%of $\overline{X}$ under the projection $M \twoheadrightarrow M/JM$. The submodule $M' = AX$ of
%$M$ generated by $X$ then satisfies $M = M' + JM$, so that $M/M' = J \cdot (M/M')$. 
%Because $M/M'$ remains bounded below, we deduce from~(1) that $M/M' = 0$, or $M = M'$.
%\end{proof}

The following generalizes a well-known characterization of finitely generated algebras from the
case where $A$ is connected (i.e., $A_0 = k$).

\begin{lemma}\label{lem:affine}
Let $A$ be a graded algebra with $A_0$ finite-dimensional, and set $S = A/J(A)$. Then the
following are equivalent:
\begin{enumerate}[label=\textnormal{(\alph*)}]
\item $A$ is finitely generated as a $k$-algebra;
\item $S$ is finitely presented as a left (equivalently, right) $A$-module; 
\item $\dim_k J(A)/J(A)^2 < \infty$.
\end{enumerate}
In particular, if $S$ is perfect as a left or right $A$-module, then $A$ is a finitely generated 
algebra.  In case the above conditions hold, $A$ is locally finite. 
\end{lemma}

\begin{proof} 
For a proof in the case where $A$ is connected, see~~\cite[Lemma~2.1.3]{Rogalski}. Because of the subtleties involved in extending to the case of a general graded algebra, we sketch a proof below. 

Set $J = J(A)$. The equivalence of~(b) and~(c) follows by considering the left (respectively, right) module presentation $0 \to J \to A \to S \to 0$ and applying Schanuel's Lemma along with Lemma~\ref{lem:Nakayama}(2) to see that $S$ is finitely presented if and only if $J/J^2$ is a finitely generated module over $A/J = S$. Since $S$ is finite-dimensional, this occurs if and only if $\dim_k J/J^2 < \infty$. 

To see that (c)$\implies$(a), it suffices to fix finite-dimensional subspaces $X \subseteq A$ and $V \subseteq J$ such that $A = X + J$ and $J = V + J^2$, and to verify that $A = X + V + V^2 + V^3 + \cdots$ so that $A$ is finite generated. One may also show that for each graded component $A_m$, there exists $r \gg 0$ such that $A_m \subseteq X + V + \cdots + V^r$, proving that $A$ is locally finite. 

To verify that (a)$\implies$(c), begin with a finite-dimensional graded subspace $W$ of $A$ that generates $A$ as a $k$-algebra. We may assume without loss of generality that $W$ is an $(A_0, A_0)$-bimodule. Then writing $W = A_0 + V$ for some $(A_0)^e$-submodule $V \subseteq A_{\geq 1}$ and verifying that $A = A_0 + V + J^2$, one concludes that $\dim_k(J/J^2) < \infty$.
\end{proof}

As mentioned earlier, if $A$ is a graded ring and $M$ and $N$ are graded left $A$-modules with $M$ finitely generated, 
then $\grHom_A(M,N) = \Hom_A(M,N)$.  We require the following derived version of this fact. 
\begin{lemma}\label{lem:graded equivalence}
Let $A$ be a graded ring and let $M \in A \lGr$. 
\begin{enumerate}
\item $M$ is perfect if and only if $M$ is graded perfect.
\item If $Q \in D^-(A \lGr)$ is a complex of finitely generated graded projectives, and $N \in D(A \lGr)$, then the natural map $\RgrHom_A(Q,N) \to \RHom_A(Q,N)$ is a quasi-isomorphism.  In particular, if $M, N \in A \lGr$ with $M$ perfect, then the natural inclusions $\grExt^i_A(M,N) \subseteq \Ext^i_A(M,N)$ are equalities, for all $i \geq 0$.
\end{enumerate}
\end{lemma}

\begin{proof}
(1) If $M$ is graded perfect, then it is clearly perfect. Conversely, assume that $M$ 
has a projective resolution of finite type, say 
\[
0 \to P^{-n} \to \cdots \to P^{-2} \overset{d^{-2}}{\longrightarrow} P^{-1} \overset{d^{-1}}{\longrightarrow} P^0 \to M \to 0,
\]
so that $M$ is quasi-isomorphic to the perfect complex $P^{\bullet}$.
Then one may construct a graded resolution $Q^{\bullet}$ of $M$ consisting of finitely generated
graded projective modules.
Because $M$ is finitely generated, it has a finite set of homogeneous generators,
and thus there is a finitely generated graded projective module $Q^0$ with a
graded surjection $Q^0 \twoheadrightarrow M$. 
The rest of the graded resolution can now be constructed by an inductive argument, with the assistance of the generalized Schanuel's
lemma~\cite[Corollary~5.5]{LMR}.

%The rest of the graded resolution is constructed by induction. Suppose that we have
%constructed $Q^0, \dots, Q^{-n+1}$. Let $K$ be the kernel of $P^{-n+1} \to P^{-n+2}$, and
%let $L$ be the kernel of the map $Q^{-n+1} \to Q^{-n+2}$. By the generalized Schanuel's
%lemma~\cite[Corollary~5.5]{LMR}, we have
%\[
%K \oplus Q^{-n+1} \oplus P^{-n+2} \oplus Q^{-n+3} \oplus \cdots 
%\cong L \oplus P^{-n+1} \oplus Q^{-n+2} \oplus P^{-n+3} \oplus \cdots.
%\]
%Since the $P^i$ and $Q^i$ above are all finitely generated, and since $K$ is finitely generated
%(as $P^{\bullet}$ is a finite type resolution), we conclude that $L$ is finitely generated.
%Furthermore, $L$ is graded because it is the kernel of a graded module homomorphism.
%Thus there exists a finitely generated projective graded module $Q^{-n}$ with a surjection
%$Q^{-n} \twoheadrightarrow L$, which we compose with the embedding $L \hookrightarrow Q^{-n+1}$
%to obtain the next step $Q^{-n} \to Q^{-n+1}$  in the resolution.

(2) Since $Q$ is a bounded above complex of graded projectives, $\RHom_A(Q, N) = \Hom_A(Q,N)$ and $\RgrHom_A(Q,N) = \grHom_A(Q,N)$.  Since each $Q^i$ is finitely generated, the inclusion $\grHom_A(Q^i, N^j) \to \Hom_A(Q^i, N^j)$ is an equality for each $i, j$; it follows that the natural inclusion of complexes $\grHom_A(Q, N) \to \Hom_A(Q, N)$ is an equality as required.  

In particular, if $M$ is a perfect module, it is quasi-isomorphic to a bounded complex $Q$ of finitely generated graded projectives
and we have $\RHom_A(M, N) = \Hom_A(Q,N)$ and $\RgrHom_A(M,N) = \grHom_A(Q,N)$; thus the second statement arises from 
the first by taking cohomology.
\end{proof}

\begin{remark}
We will often invoke the derived $\Hom$ construction in the form $\RHom_A(M,N)$ 
where $M$ and $N$ are complexes of graded modules and $M$ is perfect. 
Thanks to 
Lemma~\ref{lem:graded equivalence} and the above remarks about graded tensor products, the 
graded and ungraded Hom and tensor functors, as well as their corresponding derived functors, 
are equal in almost all cases where we are concerned. Thus unless otherwise noted, we will 
use the notation of the usual ungraded functors (such as $\Hom$, $\Ext$, $-\otimes_R -$, 
$\Tor$, etc.) with the understanding that these objects carry a canonical grading when the 
arguments are graded. We trust that this slight abuse of notation will not cause confusion,
and it carries the added benefit that some proofs can be written for ungraded and graded
twisted Calabi-Yau algebras simultaneously; we will carefully indicate when this is not the case.
\end{remark}

The following generalizes a well-known argument from the setting of connected graded algebras
to detect finiteness in the minimal projective resolution of a graded module.
\begin{lemma}\label{lem:fg term}
Let $A$ be a graded $k$-algebra whose degree zero part $A_0$ has finite $k$-dimension, 
and set $S = A/J(A)$. 
Let $M$ be a bounded-below graded left $A$-module with minimal graded projective resolution
$P^\bullet \to M \to 0$.  Then for every integer $i \geq 0$ we have isomorphisms of right $S$-modules
\[
\Hom_S(\Tor_i^A(S,M),S) \cong \Ext^i_A(M,S),
\]
and the following are equivalent:
\begin{enumerate}[label=\textnormal{(\alph*)}]
\item $P^{-i}$ is finitely generated (respectively, zero);
\item the left $S$-module $\Tor_i^A(S, M)$ has finite $k$-dimension (respectively, is zero);
\item The right $S$-module $\Ext^i_A(M,S)$ has finite $k$-dimension (respectively, is zero).
\end{enumerate}
\end{lemma}

\begin{proof}
For notational convenience, we set $j = -i$ throughout this proof.
Minimality of $P^\bullet$ implies that the boundary operators of the tensored complex 
$S \otimes_A P^\bullet$ are all zero. 
Setting $J = J(A)$, the homology modules of this complex are equal to 
\[
\Tor^A_{-j}(S, M) = A/J \otimes_A P^j \cong P^j/JP^j.
\]
Because $M$ is bounded below, the same is true of each term $P^j$. 
Similarly, because $P^\bullet$ is minimal, the boundary operators of $\Hom_A(P^\bullet,S)$ are
all zero. Combining this observation with tensor-Hom adjointness yields
\begin{align*}
\Ext^{-j}_A(M,S) &\cong \Hom_A(P^j, S) \cong \Hom_A(P^j, \Hom_S(S,S)) \\
&\cong \Hom_S(S \otimes_A P^j, S) \cong \Hom_S(P^j/JP^j, S).
\end{align*}
A direct comparison of these expressions gives $\Hom_S(\Tor_i(S,M),S) \cong \Ext^i_A(M,S)$.

By Lemma~\ref{lem:Nakayama}(2), any generating set for $P^j/JP^j$ lifts to a generating set for $P^j$. 
Note that the action of $A$ on $P^j/JP^j$ factors through the finite-dimensional algebra $S = A/J$. 
Thus $P^j$ is finitely generated (respectively, zero) if and only if $P^j / JP^j \cong \Tor^A_{-j}(S,M)$ 
is finite-dimensional (respectively, zero), establishing (a)$\iff$(b). Because the left and right $S$-dual
functors $\Hom_S(-,S)$ and $\Hom_{S\op}(-,S)$ restrict to inverse equivalences between finite-dimensional graded left and right $S$-modules, we obtain (b)$\iff$(c). 
\end{proof}

%\separate

\subsection{Idempotents and decomposition}

The next few results concern idempotents and their associated decompositions.
Recall that a ring is \emph{indecomposable} if it is not isomorphic to a direct product of two
nonzero rings, or equivalently, if it has no nontrivial central idempotents. Similarly, we say that
a graded ring is \emph{graded-indecomposable} if it is not isomorphic to a direct product
of two graded rings; this is equivalent to the property that $A$ has no nontrivial \emph{homogeneous}
central idempotents.   It is useful to know that in the case of interest to us we do not need to distinguish 
between these.

\begin{lemma}
\label{lem:central idem}
Let $A$ be a graded ring.  Then $A$ is indecomposable if and only if it is graded-indecomposable.
\end{lemma}
\begin{proof}
Let $e = e_0 + e_1 + \dots + e_n$ be a idempotent, where $e_i \in A_i$, and note that $e_0$ is a
homogeneous idempotent. Suppose that $e \neq 0,1$ is nontrivial; we will prove that $e_0$ is 
nontrivial. 
First suppose that $e_0 = 0$. Fix $d > 0$ minimal such that $e_d \neq 0$.
Then $0 \neq e_d  = (e^2)_d = 0$, a contradiction.  
Now suppose that $e_0 = 1$. As $e \neq 1$, let $d > 0$ be the minimal positive index such that 
$e_d \neq 0$.   Then $e_d = (e^2)_d = 2e_d$, yielding the contradiction $e_d = 0$.  
Thus $e_0$ is a homogeneous nontrivial idempotent. It is also easy to see that if $e$ is central, then 
so is $e_0$.  Thus if $A$ has a nontrivial central idempotent, then it has a nontrivial homogeneous central 
idempotent.  The converse is trivially true.
\end{proof}

We will need the following easy lemma on the interaction between idempotents and $\RHom$,
similar to~\cite[Proposition~21.6]{FC}.  
Note that for an idempotent $e$ in a ring $R$ and a complex $P \in D(R)$, we write $eP$ for the 
complex $eR \otimes_R^L P$, which may be obtained by multiplying all terms in the complex by $e$ on the left.
Right multiplication of a complex in $D(R\op)$ by an idempotent $e \in R$ is defined similarly.
\begin{lemma}\label{lem:central idempotent}
Let $R$, $S$, and $T$ be  rings.  Let $P$ be a bounded above complex of $(R, S)$-bimodules and 
$Q$  a complex of $(R, T)$-bimodules.
\begin{enumerate}
\item For any idempotents $e \in S$, $f \in T$ we have 
\[
\RHom_R(Pe, Qf) = e \RHom_R(P,Q) f
\]
as complexes of $(eSe, fTf)$-bimodules.
\item If $z \in Z(R)$ is a  central idempotent, then 
\begin{gather*}
z \RHom_R(P, Q) = \RHom_R(P, Q) z = \RHom_R(zP, Q) = \RHom_R(P, zQ) \\
= \RHom_R(z P, zQ) = \RHom_{zR}(zP, zQ).
\end{gather*}
as complexes of $(S, T)$-bimodules.  
\end{enumerate}
Further, if $R, S, T$ are graded rings, $P, Q$ are complexes of graded bimodules, and $e, f, z$ are homogeneous idempotents,
the same results hold with $\RHom$ replaced with $\RgrHom$, as equalities of complexes of graded bimodules.
\end{lemma}
 
\begin{proof}
(1) We prove the ungraded case; the proof in the graded case is similar.
For any $(R, S)$-bimodule $M$ and $(R, T)$-bimodule $N$, the proof that 
$\Hom_R(Me, Nf) = e \Hom_R(M,N) f$ as $(eSe, fTf)$-bimodules is routine.   Then for any complexes $P, Q$ as 
in the statement one immediately obtains $\Hom_R(Pe, Qf) = e \Hom_R(P,Q) f$ from the definition of 
the total Hom complex.  Considering $P$ as a complex of $R \otimes S\op$-modules, it may 
be replaced by a quasi-isomorphic bounded above complex of projective $R \otimes S\op$-modules, which are also projective as left $R$-modules.  The modules in $Pe$ are then also projective on the left, so we get
\[
\RHom_R(Pe, Qf) = \Hom_R(Pe, Qf) =  e \Hom_R(P,Q) f = e \RHom_R(P, Q)f
\]
as required.

(2)  Note that $P$ is a complex of $(R, S \otimes Z(R))$-bimodules and 
$Q$ a complex of $(R, T \otimes Z(R))$-bimodules.  Since $Pz = zP$ and $Qz = zQ$, 
most of the equalities follow immediately from part (1) and the others are similarly easy.
\end{proof}

\begin{lemma}\label{lem:product RHom}
If $A = \bigoplus_{i=1}^n A_{(i)}$ is a direct sum of algebras, then the following hold in $D(\rMod A^e)$:
\begin{enumerate}
\item $\RHom_{A^e}(A_{(i)},A^e) \cong \RHom_{A_{(i)}^e}(A_{(i)}, A_{(i)}^e)$;
\item $\RHom_{A^e}(A,A^e) \cong \bigoplus_{i=1}^n \RHom_{A_{(i)}^e}(A_{(i)}, A_{(i)}^e)$.
\end{enumerate}
Further, if $A = \bigoplus_{i=1}^n A_{(i)}$ is a direct sum of graded algebras, then
the corresponding quasi-isomorphisms above hold replacing $\RHom$ with $\RgrHom$.
\end{lemma}

\begin{proof}
Let $z_i \in \bigoplus_{j=1}^n A_{(j)}$ denote the element whose $i$th entry is
the identity of $A_{(i)}$ and whose other entries are zero, so that $1 = \sum z_i$ is a 
sum of orthogonal central idempotents.

For~(1), consider the central idempotent $z = z_i \otimes z_i\op \in A^e$. Note
that as a left $A^e$-module we have $A_{(i)} = zA_{(i)}$, and also that $zA^e \cong A_{(i)}^e$
as algebras and as left $A^e$-modules. Then using 
Lemma~\ref{lem:central idempotent} we have  that 
\begin{align*}
\RHom_{A^e}(A_{(i)}, A^e) &= \RHom_{A^e}(zA_{(i)}, A^e) \\
&\cong \RHom_{zA^e}(zA_{(i)}, zA^e) \\
&\cong \RHom_{A_{(i)}^e}(A_{(i)}, A_{(i)}^e).
\end{align*}
Since $A^e$ is a $(A^e, A^e)$-bimodule, the isomorphisms above hold as complexes of right $A^e$-modules.  

We deduce~(2) as follows, where~(1) is invoked in the final isomorphism:
\begin{align*}
\RHom_{A^e}(A,A^e) &= \RHom_{A^e}\left( \bigoplus A_{(i)}, A^e \right)  \\
&\cong \bigoplus \RHom_{A^e}(A_{(i)}, A^e) \cong \bigoplus \RHom_{A_{(i)}^e}(A_{(i)}, A_{(i)}^e).
\end{align*}

In case the $A_i$ are graded algebras, analogous proofs to those above, using the $\RgrHom$ version of 
Lemma~\ref{lem:central idempotent}, yield the corresponding graded isomorphisms. 
\end{proof}

\separate

\subsection{Invertible bimodules}

Next we turn our attention to invertible bimodules over algebras and graded algebras.
Note that a $k$-central invertible $(A,A)$-bimodule $U$ induces a $k$-linear Morita equivalence
$U \otimes_A - \colon A \lMod \to A \lMod$. For this reason, a number of standard results
from Morita theory (such as those in~\cite[\S 18]{LMR}, for instance) apply to invertible bimodules,
if we impose the extra condition that the equivalences of categories are $k$-linear.

Suppose that $A$ is a graded algebra.
We say that a graded $(A, A)$-bimodule $U$ is \emph{graded-invertible} if there exists a graded bimodule
$V$ such that there are isomorphisms $U \grtimes_A V \cong A \cong V \grtimes_A U$ as graded bimodules.  
It is useful to note that a graded bimodule that is invertible in the ungraded sense is actually graded-invertible, as 
follows.

\begin{lemma}\label{lem:graded invertible}
Let $A$ be a graded $k$-algebra, and suppose that $U$ is an $A^e$-module that is both graded and invertible.  
Then $U$ is graded-invertible.
\end{lemma}

\begin{proof}
Because the bimodule $U$ induces a Morita self-equivalence of $A$, it is finitely generated projective
(and a generator) on each side.  In particular, the natural map $\grHom_A(U,A) \to \Hom_A(U,A)$ is an
isomorphism, and similarly for $\grHom_{A\op}(U,A)$. 

Now invertibility of $U$ implies that the natural evaluation map 
\[
\Hom_{A\op}(U,A) \otimes_A U \to A.
\]
is an isomorphism of $(A, A)$-bimodules. But this map preserves grading, so we in fact have a graded isomorphism
\[
\grHom_{A\op}(U,A) \grtimes_A U \overset{\sim}{\longrightarrow} A.
\]
Allowing homomorphisms of left modules to act from the right (opposite of the scalars), the evaluation 
map $U \grtimes_A \grHom_A(U,A) \to A$ is also a graded isomorphism by symmetry. From 
this it is straightforward to deduce that $U$ is graded-invertible with inverse bimodule 
$\grHom_{A\op}(U,A) \cong \grHom_A(U,A)$.
\end{proof}

We will make use of the following manner in which a graded-invertible bimodule interacts well with
the graded Jacobson radical and semisimple quotient of a locally finite graded algebra. 

\begin{lemma}\label{lem:invertible commutes}
Let $A$ be a graded algebra such that $A_0$ is finite-dimensional. Denote $J = J(A)$ and $S = A/J$.
If $U$ is a graded-invertible $(A,A)$-bimodule, then the following hold:
\begin{enumerate}
\item $S \otimes_A U \cong U \otimes_A S$ as $(A,A)$-bimodules, and $UJ = JU$.
\item The bimodule $S \otimes_A U \cong U \otimes_A S$, considered as a graded $(S,S)$-bimodule
is graded-invertible.
\end{enumerate}
\end{lemma}

\begin{proof}
(1) Note that $- \grtimes_A U$ induces an autoequivalence of the category $\rGr A$, with quasi-inverse
given by $- \grtimes_A U^{-1}$. Because $S_A$ is a semisimple graded right $A$-module, the same is 
true of $S \otimes_A U$. Thus $(S \otimes_A U) J = 0$. Because $S \otimes_A U \cong U/JU$, this means
that $(U/JU) J = 0$ and thus $UJ \subseteq JU$. By a symmetric argument, we have $JU \subseteq UJ$,
giving $UJ = JU$. This implies that 
\[
S \otimes_A U \cong U/JU = U/UJ \cong U \otimes_A S.
\]

(2) Because $V = S \otimes_A U \cong U \otimes_A S$ satisfies $JV = VJ = 0$, we may consider it as
an $(S,S)$-bimodule. It is easy to deduce from~(1) that $U^{-1} \otimes_A S \cong S \otimes_A U^{-1}$ as well,
so that this bimodule can also be considered as an $(S,S)$-bimodule. We claim that
$V' = S \otimes_A U^{-1}$ is a graded inverse to $V$. Indeed,
\begin{align*}
V \otimes_S V' &\cong (U \otimes_A S) \otimes_S (S \otimes_A U^{-1}) \\
&\cong U \otimes_A S \otimes_A U^{-1} \\
&\cong S \otimes_A U \otimes_A U^{-1} \\
&\cong S,
\end{align*}
and similarly one may compute $V' \otimes_S V \cong S$. Thus $V$ is graded-invertible as an $(S,S)$-bimodule.
\end{proof}

Given a graded automorphism $\sigma$ of $A$ and an integer $l \in \Z$, it is clear that the twisted
bimodule ${}^1 A^\sigma(l)$ is graded-invertible, with inverse 
${}^1 A^{\sigma^{-1}}(-l) \cong {}^\sigma A^1(-l)$.   Suppose that $A$ is connected graded.  In this case 
all graded projective modules are free, so if $U$ is a graded invertible bimodule, $U$ is free and clearly of rank $1$ on both sides; 
it follows that $U$ is of the form ${}^1 A^\sigma(l)$ \cite[Lemma 2.9]{MM}.  However, for not necessarily connected 
graded algebras, an invertible bimodule need not have this form.  See Example~\ref{ex:weirdbimod} below for one simple example.
We defer more detailed results about the structure of graded invertible bimodules over graded algebras to the companion paper \cite{RR2}, since such information is not needed for our results in this paper.

%\separate

\subsection{Manipulating objects in the derived category}

The final results of this section concern manipulation of objects in derived categories of modules.
Part~(2) of the next lemma is a slightly more general version of the lemma observed, for instance, 
in~\cite[Lemma~2.2]{YZ}.
\begin{lemma}\label{lem:moving tensor}
Fix (graded) algebras $A$ and $B$. Suppose that $L$ is a perfect complex of (graded) left 
$A$-modules, and that $M$ is a complex of (graded) left 
$A \otimes B\op$-modules.
\begin{enumerate}
\item There is a natural isomorphism in the derived category of (graded)  right $B$-modules
\[
\RHom_A(L,A) \otimes^L_A M \cong \RHom_A(L,M).
\]
\item If $N$ is bounded above complex of left $B \otimes C\op$-modules for a (graded) algebra $C$, then there is an isomorphism in the derived category of (graded) right $C$-modules
\[
\RHom_A(L,M) \otimes^L_B N \cong \RHom_A(L, M \otimes^L_B N).
\]
\end{enumerate}
\end{lemma}

\begin{proof}
Part~(1) is proved, for instance, in~\cite[Tag~07VI]{StacksProject},
and the same proof carries over  to the graded case.
Then~(2) follows directly from~(1) because
\[
\RHom_A(L,M) \otimes^L_B N \cong \RHom_A(L,A) \otimes^L_A M \otimes^L_B N
\cong \RHom_A(L,M \otimes^L_B N). \qedhere
\]
\end{proof}

It is well-known that one can use Hochschild cohomology groups of certain bimodules to recover 
the ``one-sided'' $\Tor$ and $\Ext$ groups over $A$, for $K_A$, ${}_A M$, and ${}_A N$ as follows 
(see~\cite[Corollary~IX.4.4]{CE}):
\begin{align*}%\label{eq:computing Tor and Ext}
\Tor_n^A(K,N) &\cong \operatorname{HH}_n(A, N \otimes K) = \Tor_n^{A^e}(A, N \otimes K), \\
\quad \Ext^n_A(M,N) &\cong \operatorname{HH}^n(A, \Hom_k(M,N)) = \Ext^n_{A^e}(A,\Hom_k(M,N)).
\end{align*}
We require the following version of these identities, stated in the derived context.

\begin{lemma}\label{lem:derived Tor and Ext}
Let $A, B, C$ be  (graded) algebras.  Let $P$ be a bounded above complex of (graded) left $A^e$-modules, 
$K$ be a complex of (graded) $B \otimes A\op$-modules, $N$ a complex of (graded) 
left $A \otimes C\op$-modules, and $M$ a complex of (graded) $A \otimes B\op$-modules.  
\begin{enumerate}
\item If $K$ and $N$ are bounded above, then 
\[
K \otimes^L_A P \otimes^L_A N \cong P \otimes_{A^e}^L (N \otimes K) \cong (K \otimes N) \otimes^L_{A^e} P
\]
in the derived category of (graded) $B \otimes C\op$-modules.  
\item If $M$ is bounded above and $N$ is bounded below, then 
\[ \RHom_A(P \otimes^L_A M,N) \cong \RHom_{A^e}(P, \Hom_k(M,N)) 
\]
in the derived category of (graded) $B \otimes C\op$-modules.
\end{enumerate}
As usual, when $A, B, C$ are graded algebras, and $P, K, N, M$ are graded modules, the same results 
hold as isomorphisms of graded modules, with $\otimes$ replaced by $\grtimes$ and $\RHom$ by $\RgrHom$.
\end{lemma}

%\noindent 
%
%Before presenting the proof, 
The notation in part (1) needs a bit of explanation.  Note that if ${}_A N$
and $K_A$ are modules, then $N \otimes K$ is naturally a left $A^e$-module via the ``outer'' action 
$(a \otimes b\op) \cdot (n \otimes k) = an \otimes kb$.  
Any left $A^e$-module is naturally also a right $A^e$-module via the anti-automorphism 
$A^e \to A^e$ given by $a \otimes b\op \mapsto b \otimes a\op$.
When we think of $N \otimes K$ as a right $A^e$-module in this way, we may also write it as 
$K \otimes N$ where the right action is the ``inner'' one given by
$(k \otimes n) \cdot (b \otimes a\op) = kb \otimes an$, so that the left $B \otimes C\op$-action 
is an outer action. In case $N$ and $K$ are complexes, we simply extend the ``inner'' and ``outer''
actions above to each term in the tensor product of complexes.

We omit the proof for the sake of brevity. 
The argument is a straightforward computation in the derived category, carried out after replacing $P$ by a quasi-isomorhpic bounded above complex of projective $A^e$-modules.

\section{Homologically smooth algebras and global dimension}
\label{sec:smooth}

Recall from Definition~\ref{def:twisted CY}(i) that a $k$-algebra $A$ is \emph{homologically smooth over~$k$}
if $A$ is perfect as a left $A^e$-module. 
This section is devoted to fundamental results on locally finite graded homologically smooth $k$-algebras.

\subsection{A characterization of homological smoothness}

Our first goal will be to find an alternative characterization of homological smoothness for locally finite graded algebras.

\begin{remark}
\label{rem:swap}
We note (as in~\cite[Remark~3.2]{Lu}) that the ``swap'' algebra  isomorphism
$A^e = A \otimes A\op \to A\op \otimes A \cong (A\op)^e$ induces an equivalence
between left $A^e$-modules and left $(A\op)^e$-modules that interchanges
$A$ and $A\op$ and preserves perfect modules. This makes it easy to see
that $A$ is homologically smooth if and only if $A\op$ is.   Thus the homologically smooth 
property is ``left-right symmetric.''   
\end{remark}

A projective $A^e$-module resolution of $A$ can be used to obtain projective resolutions
of \emph{every} left $A$-module as follows.

\begin{lemma}\label{lem:smooth resolution}
Let $A$ be a (graded) algebra.
\begin{enumerate}
\item Let $P^{\bullet} \to A \to 0$ be a (graded) $A^e$-projective resolution.
For any bounded above complex $N^{\bullet}$ of (graded) left $A$-modules, $P \otimes_A^L N \cong N$ 
as complexes of (graded) left $A$-modules.  
In particular, for any (graded) left $A$-module $M$, the complex $P^\bullet \otimes_A M$
is a (graded) projective resolution of $M$.
\item If $A$ is homologically smooth and $N$ is a bounded complex of (graded) finite-dimensional left $A$-modules, 
then $N$ is perfect.
\item If $A$ is graded and homologically smooth with $A_0$ finite-dimensional, then $A$ is a finitely
generated (hence locally finite) algebra.
\end{enumerate}
\end{lemma}

\begin{proof}
(1) This is well known. For a proof in the case where $N$ consists of a single nonzero module, see~\cite[p.~68]{G}.
%The complex $P^\bullet$ consists of bimodules that are projective as both left and right $A$-modules.  
%Since $P^\bullet$ and $A$ are quasi-isomorphic and both are complexes of projective right $A$-modules, 
%we have quasi-isomorphisms $P \otimes_A N = P \otimes_A^L N \cong A \otimes_A^L N = A \otimes_A N \cong N$ 
%as complexes of left $A$-modules.  In particular, for a left $A$-module $M$, $P^{\bullet} \otimes_A M$ is a complex of 
%projective left $A$-modules which is quasi-isomorphic to $M$.

(2)  Since $A$ is homologically smooth, we may fix a perfect $A^e$-projective resolution $P^{\bullet}$ 
of $A$. By part~(1), $P \otimes_A N$ is quasi-isomorphic to $N$.   Because each $N^j$ is 
finite-dimensional and each $P^i$ is a finitely generated projective $A^e$-module, $P^i$ is a summand
of some $(A^e)^n$, and then $P^i \otimes_A N^j$ is a summand of 
$(A^e)^n \otimes_A N^j \cong A \otimes_k N^j$, which is finitely generated free on the left.
So $P^i \otimes_A N^j$ is finitely generated projective.  Thus $P \otimes_A N$ is a complex of finitely generated projective left $A$-modules and so $N$ is perfect.

(3)  Let $J = J(A)$.  In this case, $S = A/J$ is perfect as a left $A$-module by part~(2), so that $A$ is finitely generated as
an algebra and locally finite according to Lemma~\ref{lem:affine}.
\end{proof}

When manipulating $k$-central bimodules over locally finite algebras that are not necessarily connected, it 
will become crucial to assume that $A$ is separable modulo its graded Jacobson radical.
Recall that a $k$-algebra $S$ is said to be \emph{separable} (over $k$) if $S \otimes K$ is semisimple for 
every field extension $K$ of $k$.  As is well-known, there are many equivalent formulations of this property, 
as follows.

\begin{lemma} 
\label{lem:sepdef}
The following conditions are all equivalent for 
a $k$-algebra $S$:
\begin{enumerate}
\item $S \otimes K$ is semisimple for every field extension $K$ of $k$, that is, $S$ is separable;
\item $S \otimes T$ is semisimple for every semisimple $k$-algebra $T$;
\item $S$ is projective as a left $S^e$-module;
\item $S^e$ is semisimple;
\item $S \cong \bigoplus_{i=1}^n \M_{r_i}(D_i)$ for finite-dimensional division $k$-algebras $D_i$ whose
centers $Z(D_i)$ are separable field extensions of $k$.
\end{enumerate}
\end{lemma}
\begin{proof}
See~\cite[\S 13]{BourbakiVIII}, \cite[Section~71]{CR}, \cite[Chapter~II]{DI}, or~\cite[Section~9.2.1]{We}.
\end{proof}

\noindent From the characterization of separability in part (5) above, we see that a separable algebra is necessarily finite-dimensional 
over $k$; also, if $k$ is a perfect field (such as a field of characteristic~$0$, an algebraically closed 
field, or a finite field), then \emph{every} finite-dimensional semisimple $k$-algebra is separable.  

For finite-dimensional algebras, the property of being homologically smooth can be seen as
a higher-dimensional generalization of separability. 

\begin{lemma}\label{lem:fd smooth}
For a finite-dimensional $k$-algebra $A$, the following are equivalent:
\begin{enumerate}[label=\textnormal{(\alph*)}]
\item $A$ is homologically smooth;
\item $\pdim({}_{A^e} A) < \infty$;
\item $A^e$ has finite global dimension;
\item $A \otimes  K$ has finite global dimension for every field extension $K$ of $k$.
\end{enumerate}
\end{lemma}

\begin{proof}
Clearly (a)$\implies$(b); to prove the converse, assume $\pdim({}_{A^e} A) = d < \infty$.
Since $A^e$ is finite-dimensional and consequently noetherian, all terms in the minimal projective
resolution of $A$ in $A^e \lMod$ are finitely generated. Thus (b)$\implies$(a).
The equivalence of (b)--(d) is demonstrated in~\cite[p.~807]{BGMS}.
\end{proof}

In fact, it is known that conditions~(b) and~(c) above are equivalent for \emph{every} $k$-algebra
$A$; see~\cite[Corollary~2.5]{Kra}.

The question of whether a locally finite graded algebra is separable modulo its graded radical 
immediately reduces to the corresponding problem for finite-dimensional algebras, by restricting to the
degree~$0$ part of the algebra. The proofs of Theorem~\ref{thm:Rickard} 
and its supporting Lemma~\ref{lem:Rickard} were communicated to us by Jeremy Rickard. We thank him for his 
permission to include them here; we also thank MathOverflow\footnote{\url{http://mathoverflow.net/q/245764/778}} 
for providing a forum to pose the question. The following is likely a folk result, but we could not
locate a suitable reference.

\begin{lemma}\label{lem:Rickard}
If $S$ is a finite-dimensional semisimple $k$-algebra and $K = \overline{k}$ is an algebraic closure, 
then $S \otimes K$ is a finite direct sum of matrix rings over commutative local $K$-algebras.
\end{lemma}
\begin{proof}
Let $k^s \subseteq K$ denote the separable closure of $k$.  By 
\cite[Section~12.7, Corollaire~1 to Proposition~12.8]{BourbakiVIII}, 
the extension $S^s = S \otimes k^s$ is semisimple. Thus 
$S^s \cong \bigoplus_{i=1}^n \M_{r_i}(D_i)$ for some finite-dimensional division $k^s$-algebras $D_i$.
The centers $L_i = Z(D_i)$ form intermediary fields $k^s \subseteq L_i \subseteq K$.
Since $K/k^s$ is purely inseparable, the same is true for each $K/L_i$. Thus each $L_i$ is separably 
closed and consequently has trivial Brauer group \cite[Corollary 4.6]{FD}.
It follows that in fact $S^s \cong \bigoplus \M_{r_i}(L_i)$ is a sum of matrix algebras over fields.
Now $S^K = S^s \otimes_{k^s} K \cong \bigoplus \M_{r_i}(L_i \otimes_{k^s} K)$.  Each ring $R_i = L_i \otimes_{k^s} K$ 
is a finite-dimensional commutative $K$-algebra and hence is Artinian.  Then $R_i$ is a direct sum 
of finitely many commutative local $K$-algebras, say $R_i = \bigoplus_{j=1}^{n_i} R_{ij}$.  
Thus $S^K = \bigoplus_i \bigoplus_{j=1}^{n_i} \M_{r_i}( R_{ij})$ is a finite direct sum of matrix rings over commutative local $K$-algebras.
\end{proof}

\begin{theorem}[Rickard]\label{thm:Rickard}
Let $A$ be a finite-dimensional $k$-algebra. If $A$ is homologically smooth, then $S = A/J(A)$ is separable over $k$. 
\end{theorem}
\begin{proof}
Set $K = \overline{k}$. It suffices to show that $S^K = S \otimes K$ is semisimple; see~\cite[Theorem~71.2]{CR}
or~\cite[Proposition~II.1.8]{DI}. Thanks to Lemma~\ref{lem:Rickard} we have $\Ext^1_{S^K}(U,V) = 0$ 
for any non-isomorphic simple left $S^K$-modules $U$ and $V$. 
Note that $A^K$ has finite global dimension by Lemma~\ref{lem:fd smooth}.
If there were a non-split extension of a simple left $S^K$-module $U$
by itself, then we would have $\Ext^1_{A^K}(U,U) \neq 0$, contradicting the 
no-loops conjecture~\cite{Igusa}. So all extensions of simple left $S^K$-modules split,
making $S^K$ semisimple.
\end{proof}

\begin{lemma}\label{lem:enveloping radical}
Let $A$ and $B$ be graded algebras such that $A_0$ and $B_0$ are finite-dimensional, and denote
$S = A/J(A) = A_0/J(A_0)$ and $T = B/J(B) = B_0/J(B_0)$. Suppose that the semisimple $k$-algebra
$S$ is separable. Then $J(A \otimes B) = A \otimes J(B) + J(A) \otimes B$, and 
the natural map $(A \otimes B)/J(A \otimes B) \twoheadrightarrow S \otimes_k T$ is an isomorphism. 
\end{lemma}

\begin{proof}
Because $J(A \otimes B) \supseteq (A \otimes B)_{\geq 1} = A \otimes B_{\geq 1} + A_{\geq 1} \otimes B$,
we have $(A \otimes B)/J(A \otimes B) \cong (A \otimes B)_0/J((A \otimes B)_0) = (A_0 \otimes B_0)/J(A_0 \otimes B_0)$. So without loss of
generality, we may replace $A$ and $B$ by their degree-zero parts to assume that $A = A_0$ and $B = B_0$ 
are finite-dimensional algebras.

Let $J' = A \otimes J(B) + J(A) \otimes B$. Because
$A$ and $B$ are finite-dimensional, their respective Jacobson radicals $J(A)$ and $J(B)$ are nilpotent
ideals. It follows that $J'$ is a nilpotent ideal of $A \otimes B$, so that $J' \subseteq J(A \otimes B)$. 

On the other hand, if $V$ and $W$ are finite dimensional vector spaces with subspaces $V'$ and $W'$ respectively, 
then it is a standard fact that $V' \otimes W + V \otimes W'$ is the kernel of the natural map $V \otimes W \to V/V' \otimes W/W'$.
Thus the kernel of the natural map $A \otimes B \twoheadrightarrow S \otimes T$, given by the tensor product of the 
natural surjections $A \twoheadrightarrow A/J(A)$ and $B \twoheadrightarrow B/J(B)$, is equal 
to $J'$.
Because $S$ is separable and $T$ is semisimple, the algebra $S \otimes T \cong (A \otimes B)/J'$
is also semisimple by Lemma~\ref{lem:sepdef}(2).
We deduce (as in~\cite[Exercise~4.11]{FC}) that $J(A \otimes B) \subseteq J'$. So in fact 
$J' = J(A \otimes B)$, and the isomorphism 
$S \otimes T \cong (A \otimes B)/J' = (A \otimes B)/J(A \otimes B)$ follows.
\end{proof}

The next lemma facilitates the passage from modules over a graded algebra $A$ to modules
over its degree zero part.

\begin{lemma}\label{lem:degree zero resolution}
Let $A$ be a graded algebra with $A_0$ finite-dimensional, and let $M$ be a nonnegatively 
graded left $A$-module. 
If $M$ is perfect (resp., has $\pdim(M) \leq d$) over $A$, then $M_0$ is perfect 
(resp., has $\pdim(M_0) \leq d$) as an $A_0$-module.
\end{lemma}

\begin{proof}
Let $P^{\bullet} \to M \to 0$ be a minimal projective left $A$-module resolution of $M$.
Recall that every left bounded graded left projective $A$-module is of the form 
$A \otimes_{A_0} Q$ for some graded left projective $A_0$-module $Q$.
Writing $P^i = A \otimes_{A_0} Q^i$ for some graded projective $A_0$-modules 
$Q^i$, because $M$ is nonnegatively graded and the resolution $P^{\bullet}$ is minimal, it 
easily follows that each $Q^i$ is nonnegatively graded.   Take the degree-zero part of the resolution of 
$M$ to obtain $P^{\bullet}_0 \to M_0 \to 0$, which is an exact sequence of $A_0$-modules.
Then $P^i_0 = Q^i_0$, which is projective over $A_0$ for each $i$ as it is a summand
of $Q^i$.
So we have a projective resolution $P_0^\bullet$ of $M$ over $A_0$. 
If $P^\bullet$ is perfect or has length at most $d$, then the same is true of the
$A_0$-resolution $P_0^\bullet$.
%
%
%Let $P^{\bullet} \to M \to 0$ be a minimal projective left $A$-module resolution of $M$.
%Recall that every left bounded graded left projective $A$-module is of the form 
%$A \otimes_{A_0} Q$ for some graded left projective $A_0$-module $Q$.
%Writing $P^i = A \otimes_{A_0} Q^i$ for some graded projective $A_0$-modules 
%$Q^i$, because $M$ is nonnegatively graded and the resolution $P^{\bullet}$ is minimal, it 
%easily follows that each $Q^i$ is nonnegatively graded.   Take the degree-zero part of the resolution of 
%$M$ to obtain $P^{\bullet}_0 \to M_0 \to 0$, which is an exact sequence of $A_0$-modules.
%Then $P^i_0 = Q^i_0$, which is projective over $A_0$ for each $i$ as it is a summand
%of $Q^i$.
%So we have a projective resolution $P_0^\bullet$ of $M$ over $A_0$. 
%If $P^\bullet$ is perfect or has length at most $d$, then the same is true of the
%$A_0$-resolution $P_0^\bullet$.
\end{proof}

\begin{lemma}\label{lem:homologically smooth}
Let $A$ be a graded algebra with $S = A/J(A)$. 
\begin{enumerate}
\item $\Tor^A_i(S,S) \cong \Tor^{A^e}_i(S^e,A)$ as graded left $S^e$-modules. 
\item If $A$ is homologically smooth, then so is $A_0$.
\end{enumerate}
\end{lemma}

\begin{proof}
(1) We adapt the method of proof from the case of a connected algebra given in~\cite[Lemma~4.3(a)]{YZ}. Using Lemma~\ref{lem:derived Tor and Ext}(1) in the second quasi-isomorphism below, we have 
\[
S \otimes^L_A S \cong S \otimes^L_A (A \otimes^L_A S)  \cong S^e \otimes^L_{A^e} A,
\]
as complexes of left $S^e$-modules.  Taking cohomology of the above yields an isomorphism 
$\Tor^A_i(S,S) \cong \Tor^{A^e}_i(S^e, A)$ as desired.

(2) Because $A$ is perfect as a left $A^e$-module, it follows from 
Lemma~\ref{lem:degree zero resolution} that $A_0$ is perfect as a module over
$(A^e)_0 = (A_0)^e$. Thus $A_0$ is homologically smooth.
\end{proof}

We arrive at the following alternative characterization of homological smoothness for
locally finite graded algebras.

\begin{theorem}\label{thm:homologically smooth}
Let $A$ be a graded $k$-algebra with $A_0$ finite-dimensional, and set $S = A/J(A) = A_0/J(A_0)$.
Then the following conditions are equivalent:
\begin{enumerate}[label=\textnormal{(\alph*)}]
\item $A$ is (graded) homologically smooth over~$k$;
\item $S$ is separable over~$k$ and $S$ is perfect as a left (respectively, right) $A$-module.
\end{enumerate}
If the conditions hold, then $A$ is a finitely generated (hence locally finite) algebra, 
and the canonical surjection $A \twoheadrightarrow S$ is split by an algebra homomorphism 
$S \hookrightarrow A$.
\end{theorem}

\begin{proof}
Assume~(a) holds. Then $S$ is perfect as both a left and a right $A$-module by 
Lemma~\ref{lem:smooth resolution}(2).
Next, $A_0$ is homologically smooth by Lemma~\ref{lem:homologically smooth}(2).
Now Rickard's Theorem~\ref{thm:Rickard} implies that $S = A_0/J(A_0)$ is separable over $k$.
Thus (a)$\implies$(b).

Next, assume that~(b) holds; we will deduce~(a). Because $S$ is separable we have 
$A^e/J(A^e) \cong S^e$ by Lemma~\ref{lem:enveloping radical}. Then 
Lemma~\ref{lem:homologically smooth}~(1) yields
\[
\Tor^A_i(A/J(A),A/J(A)) \cong \Tor^{A^e}_i(A^e/J(A^e),A).
\]
Applying Lemma~\ref{lem:fg term} (and its right sided variant) for both $A$ and $A^e$, we have that 
$S$ is perfect as a left (respectively, right) $A$-module if and only if the $\Tor$ vector spaces above 
are all finite-dimensional and eventually zero for $i \gg 0$, if and only if $A$ is perfect as a left 
$A^e$-module. Thus (b)$\implies$(a). 

If condition (b) (or equivalently condition (a)) holds, then Lemma~\ref{lem:affine}
implies that $A$ is a finitely generated and locally finite algebra. 
To establish the final claim about the splitting of $A \twoheadrightarrow S$, simply note that
this factors through the split surjection $A \twoheadrightarrow A_0$ via the surjection 
$A_0 \twoheadrightarrow S$, and that the latter is split when $S$ is separable by a classical result 
of Wedderburn and Malcev~\cite[Theorem~72.19]{CR}.
\end{proof}

\subsection{Basic operations preserving homological smoothness}

We next turn our attention to operations that preserve the property of homological smoothness.
We begin with direct sums.

\begin{proposition}\label{prop:smooth product}
For $k$-algebras $A$ and $B$, denote $R = A \oplus B$. Then $R$ is 
homologically smooth over $k$ if and only if both $A$ and $B$ are homologically smooth over $k$.
\end{proposition}

\begin{proof}
First suppose that $A$ and $B$ are homologically smooth. For covenience, denote $A_{(1)} = A$ and $A_{(2)} = B$.
Then $R^e = (A_{(1)} \oplus A_{(2)}) \otimes (A_{(1)} \oplus A_{(2)})\op = \bigoplus_{i,j\in \{1,2\}} A_{(i)} \otimes A_{(j)}\op$ as $k$-algebras. 
Let $P_{(i)}$ denote a perfect left $A_{(i)}^e$-resolution of each $A_{(i)}$; then it is also a perfect
left $R^e$-resolution via the projection $R^e \twoheadrightarrow A_{(i)} \otimes A_{(i)}\op = A_{(i)}^e$.
Thus $P_{(1)} \oplus P_{(2)}$ is a perfect left $A^e$-resolution of $A = A_{(1)} \oplus A_{(2)}$, 
and so $A$ is homologically smooth as desired.

Conversely, suppose that $R$ is homologically smooth.
Consider the central idempotent $z = (1,0) \in A \oplus B = R$.
If $P^{\bullet}$ is a finite 
resolution of $R$ by finitely generated $R^e$-projectives, then $(z \otimes z) P^{\bullet}$ is an 
$A^e$-projective resolution of $A$ with the same properties. Thus $A$ is homologically
smooth, and the same is true of $B$ by symmetry.
\end{proof}

Next we consider extension of scalars. Given a $k$-vector space $V$ and a field extension $K$ of $k$, we
denote the scalar extension of $V$ to a $K$-vector space by $V^K = V \otimes K$. In case
$V$ is a $k$-algebra, $V^K$ becomes a $K$-algebra; if $V$ is a module over a $k$-algebra $A$,
then $V^K$ becomes a module over $A^K$ in the natural way.  The following basic results give   
the interaction between base field extension and Hom and tensor as well as their derived analogues.

\begin{lemma}\label{lem:extension}
Let $A$ be a $k$-algebra.
\begin{enumerate}
\item Let $P^{\bullet}, Q^{\bullet}$ be complexes of left $A$-modules, with $P$ bounded above.  There is a natural 
map 
\[
\Psi: \RHom_A(P,Q)^K \to \RHom_{A^K}(P^K, Q^K),
\]
respecting any module structure obtained if $P$ or $Q$ is a complex of bimodules. The map $\Psi$ is an isomorphism if $P$ is perfect.
\item Let $P^{\bullet}$ be a bounded above complex of right $A$-modules, and $Q^{\bullet}$ a complex 
of left $A$-modules.  There is a natural isomorphism 
\[
\Phi: (P \otimes^L_A Q)^K \to P^K \otimes^L_{A^K} Q^K
\]
which respects any module structure obtained if $P$ or $Q$ is a complex of bimodules.
\end{enumerate}
When $A$ is graded, the same results hold for graded modules and complexes, with $\RHom$ replaced by $\RgrHom$ and $\otimes^L$ replaced by $\grtimes^L$. 
\end{lemma}

\begin{proof}
(1) For any $M, N \in A \lMod$ there is a natural map $\psi: \Hom_A(M,N)^K \to \Hom_{A^K}(M^K, N^K)$ 
which respects any bimodule structures involved.  The map $\psi$ is clearly an isomorphism when $M = A$, and one can readily deduce that the same is true when $M$ is finitely generated projective.   

Since $P$ is bounded above, we can replace it with a quasi-isomorphic bounded above complex of projective modules.  Then $P^K$ is a  bounded above complex of projective $A^K$-modules, so $\RHom_A(P, Q)^K = \Hom_A(P, Q)^K$ and $\RHom_{A^K}(P^K, Q^K) = \Hom_{A^K}(P^K, Q^K)$.  Thus one defines $\Psi$ to be the map on complexes whose $n$th component is the natural map 
\[
\Psi^n: [\prod_i \Hom_A(P^i, Q^{i+n})] \otimes K \to \prod_i \Hom_{A^K}(P^i \otimes K, Q^{i+n} \otimes K)
\]
induced by $\psi$.  
When $P$ is perfect, we can assume that $P$ is a bounded complex of finitely generated projectives.  In this case
in $\Psi^n$ above the product $\prod_i$ is actually finite and each $P^i$ is finitely generated.  Since tensor products commute with finite 
products (that is, direct sums), we see that $\Psi^n$ is an isomorphism from the first paragraph.

(2)  This is a similar argument as part (1), but no additional hypotheses are needed since tensor products commute with 
arbitrary direct sums.  We leave the details to the reader.

The proofs in the graded case are the same.
%
%(1) \todo{(How to shorten?)} For any $M, N \in A \lGr$ there is a natural map $\psi: \Hom_A(M,N)^K \to \Hom_{A^K}(M^K, N^K)$ 
%which respects any bimodule structures involved.  The map $\psi$ is clearly an isomorphism when $M = A$, and therefore when $M = A^{\oplus n}$ for any integer $n$, since finite direct sums pull out of the first coordinate of $\Hom$.  Then $\psi$ is an isomorphism when $M$ is finitely generated projective.   
%
%Since $P$ is bounded above, we can replace it with a quasi-isomorphic bounded above complex of projective modules.  Then $P^K$ is a  bounded above complex of projective $A^K$-modules, so $\RHom_A(P, Q)^K = \Hom_A(P, Q)^K$ and $\RHom_{A^K}(P^K, Q^K) = \Hom_{A^K}(P^K, Q^K)$.  Thus one defines $\Psi$ to be the map on complexes whose $n$th component is the natural map 
%\[
%\Psi^n: [\prod_i \Hom_A(P^i, Q^{i+n})] \otimes K \to \prod_i \Hom_{A^K}(P^i \otimes K, Q^{i+n} \otimes K)
%\]
%induced by $\psi$.  
%
%When $P$ is perfect, we can assume that $P$ is a bounded complex of finitely generated projectives.  In this case
%in $\Psi^n$ above the product $\prod_i$ is actually finite and each $P^i$ is finitely generated.  Since tensor products commute with finite 
%products (that is, direct sums), we see that $\Psi^n$ is an isomorphism from the first paragraph.
%
%(2)  This is a similar argument as part (1), but no additional hypotheses are needed since tensor products commute with 
%arbitrary direct sums.  We leave the details to the reader.
%
%The proofs in the graded case are the same.
\end{proof}

We now study the behavior of homological smoothness under base field extension.  In fact 
we have the following more general result about how perfect modules behave under extension of the base field.  
We do not know if part (2) of the 
following result also holds without the graded hypothesis.

\begin{proposition}\label{prop:smooth extension}
Let $A$ be a $k$-algebra, and let $K$ be a field extension of $k$.
\begin{enumerate}
\item If $M$ is a perfect left $A$-module, then $M^K$ is a perfect $A^K$-module.
\item Suppose that $A$ is graded with $\dim_k A_0 < \infty$, and let $M \in A \lGr$ be left bounded.  Then $M^K$ is a perfect left $A^K$-module if and only if $M$ is perfect over $A$.  In this case, the length of the minimal graded  projective resolutions of $M$ over $A$ and $M^K$ over $A^K$ are the same.
\end{enumerate}
\end{proposition}

\begin{proof}
(1) Let $P^{\bullet} \to M$ be a projective resolution of $M$ as a left $A$-module.  If $M$ is perfect, 
then we can take $P^i$ to be finitely generated for all $i$ and zero for $i \ll 0$.  The functor $- \otimes_k K$ is exact, and is 
easily seen to preserve the properties of being projective and being finitely generated as a module.  
Then $(P^{\bullet})^K$ is an $A^K$-projective resolution of $A^K$, for which $(P^i)^K$ is finitely generated for all $i$ and $0$ 
for $i \ll 0$.  Thus $M^K$ is perfect.

(2) By part (1), it suffices to assume that $M^K$ is perfect and prove that $M$ is.  
The conditions on $A$ and $M$ ensure that we can take a minimal graded projective resolution $P^{\bullet} \to M$ of $M$ as a left $A$-module.   Similarly as above, $(P^{\bullet})^K$ is a graded $A^K$-projective resolution of $M^K$.  We claim that this is also a minimal projective resolution.
Suppose that $N, P \in A \lGr$ are left bounded graded projective modules.  Let $f: N \to P$ be a graded $A$-module homomorphism with $\ker f \in J(A) N$.  Then $f \otimes 1: N^K \to P^K$ is a graded $(A^K)$-module homomorphism 
and clearly $\ker (f \otimes 1) = (\ker f) \otimes K$ since $- \otimes_k K$ is exact.   Recall that  $J(A) = A_{\geq 1} \oplus J(A_0)$.
Clearly $A_{\geq 1}^K = (A^K)_{\geq 1} \subseteq J(A^K)$.  Since $J(A_0)$ is nilpotent, so is $J(A_0)^K$ 
and hence $J(A_0)^K \subseteq J(A^K)$ as well.  Thus $J(A)^K \subseteq J(A^K)$.  
This easily extends to projective modules to prove that for any left bounded graded projective $A$-module $Q$, 
$(J(A) Q)^K \subseteq J(A^K)(Q^K)$.
Thus 
\[
\ker (f \otimes 1)  = (\ker f) \otimes K \subseteq (J(A) N)^K \subseteq J(A^K)N^K.
\]
Since the condition that $P^{\bullet}$ be minimal is that for each $d^i: P^i \to P^{i-1}$ we have 
$\ker d^i \subseteq J(A) P^i$, we see that $(P^K)^{\bullet}$ is also minimal, as claimed.
Now suppose that $M^K$ is perfect.  Then since $M^K$ has a finite $(A^K)$-resolution by finitely generated 
projectives, the minimal graded projective resolution of $M^K$ must have this property.  
So each $(P^i)^K$ is a finitely generated $A^K$-module, 
with $(P^i)^K = 0$ for $i \ll 0$.  It follows that $P^i = 0$ for $i \ll 0$.  It is also easy to see that $(P^i)^K$ finitely generated 
implies that $P^i$ is a finitely generated $A$-module, as follows.   Let $Q = P^i$ and suppose that
$Q = \bigcup Q_j$ is a directed union of submodules. Then $Q^K = \bigcup Q_j^K$ is a directed union 
of the extended submodules. By finite generation of $Q^K$, we have some $Q_j^K = Q^K$.
Thus the quotient module $Q^K/Q_j^K \cong (Q/Q_j)^K$ is zero. As $K$ is faithfully flat over $k$
we have $Q/Q_j = 0$, whence $Q = Q_j$. Thus $Q = P^i$ is finitely generated.

We see now that $M$ is perfect.  The final statement on the lengths of the minimal projective resolutions is also clear from the proof above.
\end{proof}

\begin{corollary}\label{cor:smooth extension}
Let $A$ be a $k$-algebra, and let $K$ be a field extension of $k$.
If $A$ is a homologically smooth $k$-algebra, then $A^K$ is a homologically smooth $K$-algebra. 
The converse holds if $A$ is graded with $\dim_k A_0 < \infty$.
\end{corollary}
\begin{proof}
Note that $(A \otimes K)^e \cong A^e \otimes K$ as (graded) $K$-algebras.   Thus the result 
follows immediately from Proposition~\ref{prop:smooth extension}, applied to the algebra $A^e$ and the module $A$.
\end{proof}

Next we wish to show that homological smoothness is preserved under tensor products of algebras.
For this we require the following observation. 
\begin{remark}\label{rem:Kunneth}
Let $R$ and $S$ be $k$-algebras.   For $M \in R \lMod$ and $N \in S \lMod$, suppose that 
$P^\bullet$ is a projective $R$-module resolution of $M$ and $Q^{\bullet}$ is a projective 
$S$-module resolution of $N$.  As we are tensoring over a field~$k$, the K\"{u}nneth 
formula~\cite[Theorem~3.6.3]{We} yields 
$H^i(X \otimes Y) \cong  \bigoplus_{p+q = i} H^p(X) \otimes H^q(Y)$ for any bounded above 
complexes $X$ and $Y$.  In particular,
the tensor complex $P \otimes Q$ forms a projective resolution of $M \otimes N$ as a left 
$R \otimes S$-module.  
\end{remark}

\begin{proposition}\label{prop:smooth tensor}
If $A$ and $B$ are (graded) homologically smooth $k$-algebras, then so is $A \otimes B$.
\end{proposition}

\begin{proof}
Denote $R = A \otimes B$, so that $R^e = A^e \otimes B^e$.  If $A$ and $B$ 
are homologically smooth, we may fix finite length resolutions 
$P^{\bullet} \to A \to 0$ and $Q^\bullet \to B \to 0$, whose terms are finitely generated projective left 
modules over $A^e$ and $B^e$, respectively.
By Remark~\ref{rem:Kunneth}, the complex $P \otimes Q$ forms a projective
$R^e$-resolution of $R$, from which we deduce that $R$ is homologically smooth.
\end{proof}

The final operation preserving homological smoothness that we will investigate is 
$k$-linear Morita equivalence. We say that two $k$-algebras are \emph{$k$-linearly
Morita equivalent} if their left module categories are equivalent via a $k$-linear equivalence
of categories; thanks to standard Morita theory, this is clearly induced by a $k$-central
invertible bimodule.

\begin{proposition}\label{prop:smooth Morita}
If $A$ and $B$ are $k$-linearly Morita equivalent $k$-algebras, then
$A$ is a homologically smooth $k$-algebra if and only if $B$ is.
\end{proposition}

\begin{proof}
Fix a $k$-central invertible $(A,B)$-bimodule $U$, so that $U \otimes_B - \colon B \lMod \to A \lMod$
and $- \otimes_A U \colon \rMod A \to \rMod B$ are $k$-linear equivalences.
Then the functor $U \otimes_B - \otimes_B U^{-1}$ gives a $k$-linear equivalence of categories from $k$-central 
$(B,B)$-bimodules to $k$-central $(A,A)$-bimodules; that is, we have a linear equivalence
from $B^e \lMod$ to $A^e \lMod$. This is readily verified to preserve the tensor product of bimodules
up to natural isomorphism, yielding a monoidal equivalence. 
In particular, the tensor units $B$ and $A$ correspond under the equivalence.

It follows that $A$ has a perfect resolution in $A^e \lMod$ if and only if $B$ has a perfect
resolution in $B^e \lMod$, establishing the claim.
\end{proof}

%\separate

\subsection{Global dimension of homologically smooth algebras}

We now turn our attention to the global dimension of locally finite graded algebras.
Recall that the \emph{graded projective dimension} of a left module $M$ over a graded ring $R$, denoted
$\grpdim(M)$,  is the minimal (possibly infinite) length of all graded projective resolutions of $M$,
and that the \emph{graded left global dimension} of $R$, denoted $\grgldim_l(M)$, is the supremum 
of $\grpdim(M)$ with $M$ ranging over all graded left $R$-modules. 
%It is clear that $\pdim(M) \leq \grpdim(M)$ for all $M$.

The next result shows that for the graded algebras which we consider in this paper, one need not 
distinguish between left and right graded global dimensions, nor between graded and ungraded 
global dimensions.  
%Special cases of all three parts of the following have appeared in the literature. For~(1) and~(3) in the 
%case of connected graded  algebras, see~\cite{Be}. Part~(2) has a more general formulation
%in~\cite[\S5--6]{Eilenberg}, and similar results appear in~\cite{Li} and~\cite[Proposition~2.7]{MM}.
%We adapt our respective arguments from these references.

\begin{proposition}\label{prop:global dimension}
Let $A$ be a graded algebra with $A_0$ finite-dimensional, and set $S = A/J(A)$. 
\begin{enumerate}
\item For any graded left $A$-module $M$ that is bounded below, one has
\[
\pdim(M) = \grpdim(M) = \sup\{n \in \N \mid \Tor^A_n(S,M) \neq 0\}.
\]
Thus $\grgldim_l(A) \leq \gldim_l(A)$.
\item The left and right graded global dimensions of $A$ are equal; in fact,
\[
\grgldim_l(A) = \pdim({}_A S) = \pdim(S_A) = \grgldim_r(A).
\]
\item If $S$ is a separable $k$-algebra, then 
\begin{align*}
\gldim_l(A)  &= \grgldim_l(A) = \pdim({}_A S) = \pdim({}_{A^e} A)\\
&= \gldim_r(A) = \grgldim_r(A)  = \pdim(S_A).
\end{align*}
\end{enumerate}
\end{proposition}

\begin{proof}
For~(1), note that $\sup\{n \mid \Tor^A_n(S,M) \neq 0\} \leq \pdim(M) \leq \grpdim(M)$. 
Now let $P^\bullet \to M \to 0$ be a minimal graded projective resolution of $M$. Then by Lemma~\ref{lem:fg term}, 
we see that $\Tor^A_n(S,M) = 0$ if and only if the term 
$P^{-n}$ in the complex is zero, giving $\sup\{n \mid \Tor^A_n(S,M) \neq 0\} =\grpdim(M)$.
The inequality $\grgldim_l(A) \leq \gldim_l(A)$ readily follows. 

Part~(2) follows from the results of~\cite[\S5--6]{Eilenberg}. Note that similar results appear in~\cite{Li} and~\cite[Proposition~2.7]{MM}.

%To establish~(2), first note that by part (1) and its right sided analog, we have 
%$\pdim({}_A S) = \sup\{n \in \N \mid \Tor^A_n(S,S) \neq 0\} = \pdim(S_A)$.  
%Next note that $\grgldim_l(A) = \sup\{\pdim(A/I)\}$ where $I$ ranges over the graded 
%left ideals of $A$; the argument is the obvious graded analogue of the classical one given 
%in~\cite[Theorem~1]{A} or in \cite[Theorem 8.16]{Rotman}. For such modules $M = A/I$, part~(1) gives 
%\[
%\pdim(M) \leq \pdim(S_A) = \pdim({}_A S) \leq \grgldim_l(A).
%\]
%Taking the supremum over such $\pdim(M)$ gives $\grgldim_l(A) = \pdim(S_A) = \pdim({}_A S)$.  
%By symmetry, these are also equal to $\grgldim_r(A)$.  

To prove~(3), we adapt the argument from the connected graded case given in~\cite{Be}. Note that 
\[
\pdim({}_A S) = \grgldim_l(A) \leq \gldim_l(A) \leq \pdim({}_{A^e} A),
\]
where the first equality follows from part~(2), the middle inequality from part~(1), and the 
final inequality from Lemma~\ref{lem:smooth resolution}. 
Applying~(1) to the left $A^e$-module $A$, along with Lemmas~\ref{lem:enveloping radical}
and~\ref{lem:homologically smooth}(1), we see that
\[
\pdim({}_{A^e} A) = \sup\{n \mid \Tor^{A^e}_n(A^e/J(A^e), A) \neq 0\} 
= \sup\{n \mid \Tor^A_n(S,S) \neq 0\}.
\]
As we saw in the proof of part (2), the quantity on the right-hand-side above is $\pdim({}_A S) = \pdim(S_A)$. 
Thus we obtain
\[
\pdim({}_A S) = \grgldim_l(A) = \gldim_l(A) = \pdim({}_{A^e} A) = \pdim(S_A),
\]
and the rest of the equalities follow by symmetry.
\end{proof}

It is well-known that any finite-dimensional algebra $A$ with $S = A/J(A)$ has 
$\gldim_l(A) = \pdim({}_A S) = \pdim(S_A) = \gldim_r(A)$; for instance, 
see~\cite[Corollary~5.60, Theorem~5.72]{LMR}.
If $S$ is separable, then we can deduce directly from the previous proposition that
this global dimension is further equal to $\pdim({}_{A^e} A)$.

\begin{corollary}\label{cor:separable dimension}
If $A$ is a finite-dimensional $k$-algebra such that $S = A/J(A)$ is separable, then
$\gldim(A) = \pdim({}_{A^e} A)$.
\end{corollary}

\begin{proof}
Considering $A$ as a graded algebra concentrated in degree 
zero, this follows immediately from Proposition~\ref{prop:global dimension}(3).
\end{proof}

The following example shows that the hypothesis that $S = A/J(A)$ is separable cannot simply
be omitted in some of our previous results. 
It makes use of the fact that a module $M$ over a Frobenius algebra $A$ must
have $\pdim(M) \in \{0,\infty\}$. Indeed, because projective and injective $A$-modules coincide,
any finite projective resolution of $M$ of length greater than zero must split at the end,
and therefore can be replaced with a projective resolution of shorter length.
 
\begin{example}
\label{ex:inseparable}
Let $S$ be a finite-dimensional, semisimple $k$-algebra that is not separable. (For instance, $S$
could be a non-separable field extension of a non-perfect field $k$.) Then $S^e = S \otimes S\op$
is a tensor product of Frobenius algebras and thus is Frobenius~\cite[Exercise~3.12]{EMR}.
Because $S$ is not separable, it is not projective as a left $S^e$-module. 
Because $S^e$ is Frobenius, it follows from the remark above that $\pdim({}_{S^e} S) = \infty$.
So $S$ is not perfect as a left $S^e$-module. On the other hand, $\gldim(S) = 0$ and $S$ is 
perfect as a left $S$-module. Taking $A = S$ we see that both the implication (b)$\implies$(a) in 
Theorem~\ref{thm:homologically smooth} as well as part~(3) of 
Proposition~\ref{prop:global dimension} can fail if we omit the assumption that $S$ is separable. 
\end{example}

\begin{example}
For $\mb{Z}$-graded rings, not the $\mb{N}$-graded rings we are restricting to in this paper, 
it is easy to find examples whose graded global dimension and ungraded global dimension are different.
For example, a graded division ring such as the Laurent polynomial ring $k[t, t^{-1}]$ has graded global 
dimension~$0$, but global dimension~$1$.  We are unaware of an $\mb{N}$-graded, locally finite 
algebra $A$ whose global dimension and graded global dimension are different.
\end{example}

The next result provides one further characterization of finite-dimensional homologically smooth 
algebras, in addition to those already given in Theorem~\ref{thm:homologically smooth}.

\begin{corollary}
Let $A$ be a locally finite graded algebra that is left noetherian; this holds, in particular, if $A$ is a
finite-dimensional algebra with trivial grading.
Then $A$ is homologically smooth if and only if $S = A/J(A)$ is separable and $A$ has finite global
dimension.
\end{corollary}

\begin{proof}
By Theorem~\ref{thm:homologically smooth} we know that $A$ is homologically smooth if and 
only if $S$ is a separable algebra and $S$ is perfect as a left $A$-module. 
Thus it is enough to show, under the assumption that $S$ is separable, that $A$ has finite global dimension 
if and only if $S$ is perfect.  Because $A$ is left noetherian, the finite-dimensional module ${}_A S$ is perfect if and only if it has 
finite projective dimension.  The latter holds (since $S$ is separable) if and only if $\gldim(A) = \grgldim(A)$
is finite, thanks to Proposition~\ref{prop:global dimension}, establishing the claim.
\end{proof}

It is also useful to note that the global dimension of a graded algebra 
bounds the global dimension of its degree $0$ part.

\begin{lemma}\label{lem:degree zero dimension}
Let $A$ be a graded algebra with $A_0$ finite-dimensional. 
If $A$ has finite graded global dimension~$d$, then the global dimension of $A_0$ is at most $d$.
\end{lemma}

\begin{proof}
As in Proposition~\ref{prop:global dimension}(2), the left and right graded global dimensions of 
$A$ coincide, and similarly for $A_0$. Thus it suffices to prove that $\gldim_l(A_0) \leq d$.  
Given a left $A_0$-module $M$, consider it as a graded left $A$-module concentrated in degree zero.
Because the projective dimension of $M$ as a left $A$-module is at most~$d$, it follows from
Lemma~\ref{lem:degree zero resolution} that it also has projective dimension at most~$d$ as
a left $A_0$-module, establishing $\gldim_l(A_0) \leq d$.
\end{proof}

\separate

Homologically smooth graded algebras need not have good ring-theoretic properties.  Over a perfect 
field, we have seen that a locally finite graded algebra is homologically smooth if and only if ${}_A S$
is a perfect module, and there are many examples of finite global dimension algebras with finite 
GK-dimension but bad properties.  
One simple such example is $A = k \langle x, y \rangle/(yx)$. This is a connected graded algebra of
GK-dimension~$2$ over which ${}_A S = A/J(A) = k$ is perfect; 
however, it is neither left nor right noetherian, nor is it semiprime.
See~\cite[Example~2.1.7, Exercise~2.4.3]{Rogalski} for details.
On the other hand, twisted Calabi-Yau algebras typically have good ring-theoretic properties when 
they also have finite GK-dimension. The reasons for this are not yet well-understood in general, but
we address the noetherian property in particular for algebras of dimension $d \leq 2$ in subsequent 
sections of this paper.

\section{Twisted Calabi-Yau algebras}
\label{sec:twisted CY}

In this section we provide some basic tools for the study of twisted Calabi-Yau algebras, as we 
defined in Definition~\ref{def:twisted CY}.  

\subsection{Graded twisted Calabi-Yau algebras}
Our main interest is in \emph{graded} twisted Calabi-Yau algebras. The following is the natural
graded analogue of the twisted Calabi-Yau condition.

\begin{definition}\label{def:graded twisted CY}
Let $A$ be an $\N$-graded $k$-algebra. We say that:
\begin{enumerate}[label=\textnormal{(\roman*)}]
\item $A$ is \emph{graded homologically smooth} if it is graded perfect as a left $A^e$-module;
\item $A$ is \emph{graded twisted Calabi-Yau (of dimension $d$)} if it is graded homologically 
smooth and
\[
\grExt^i_{A^e}(A,A^e) \cong 
\begin{cases}
0, & i \neq d \\
U, & i = d
\end{cases}
\]
as graded right $A^e$-modules, for some graded-invertible bimodule $U$.
%\item If the isomorphism of~(ii) holds with $U \cong {}^1 A^\mu(l)$ as graded bimodules for a graded 
%$k$-algebra automorphism $\mu$ and some integer $l$, then $\mu$ is called  a \emph{Nakayama 
%automorphism} of $A$, and $l$ is called the \emph{AS index} or \emph{Gorenstein parameter} of $A$.
\end{enumerate}
\end{definition}

Similarly as in Remark~\ref{rem:swap}, since there is an anti-isomorphism of $A^e$ it is 
easy to see that $A$ is  twisted Calabi-Yau of dimension $d$ if and only if $A\op$ is.

As is turns out, the graded and ungraded versions of the twisted Calabi-Yau property are equivalent.

\begin{theorem}\label{thm:graded versus ungraded}
Let $A$ be a graded algebra.   
Then $A$ is twisted Calabi-Yau of dimension~$d$ if and only if $A$ is graded twisted Calabi-Yau 
of dimension~$d$.
\end{theorem}

\begin{proof}
First note that by Lemma~\ref{lem:graded equivalence}, $A$ is homologically smooth if and 
only if it is graded homologically smooth, and in case either condition is satisfied the inclusions 
$\grExt^i_{A^e}(A, A^e) \subseteq \Ext^i_{A^e}(A,A^e)$ of right $A^e$-modules are equalities for 
all $i$ by Lemma~\ref{lem:graded equivalence}. For homologically smooth $A$, the graded bimodule $U = \Ext^d_{A^e}(A, A^e)$ is invertible 
if and only if it is graded-invertible, by Lemma~\ref{lem:graded invertible}. It follows immediately
that $A$ is graded twisted Calabi-Yau algebra of dimension~$d$ if and only if it is (ungraded) twisted 
Calabi-Yau of dimension~$d$.
\end{proof}

The following general lemma will be useful in translating from conditions involving $\Ext$ to conditions 
in the derived category.
\begin{lemma}
\label{lem:extderived}
Let $A$ be a $k$-algebra.
\begin{enumerate}
\item Let $0 \neq M \in A \lMod$ be perfect.  Then $\pdim(M) = \max \{ i | \Ext^i_A(M, A) \neq 0 \}$.
\item If $M \in A \lMod$ is perfect and $\Ext^i_A(M, A) = 0$ for $i \neq d$, then setting $N = \Ext^d_A(M, A)$
we have $\RHom_A(M, A) \cong N[-d]$ in $D^b(\rMod A)$.
\item Suppose that $A$ is graded with $\dim_k A_0 < \infty$.  If $M \in A \lGr$ is a bounded below module with 
$\pdim(M) \leq d$, $\Ext^i_A(M, A) = 0$ for $i \neq d$, and such that $N = \Ext^d_A(M, A)$ satisfies $\dim_k N < \infty$, then $M$ is perfect 
and part (2) applies to show $\RHom_A(M,A) = \RgrHom_A(M, A) \cong N[-d]$ in $D^b(\rGr A)$.
\end{enumerate}
\end{lemma}
\begin{proof}
(1)  This is a standard result.  Assume that $M \neq 0$.   By definition we have 
a finite projective resolution $P^{\bullet} \to M$, with all $P^i$ finitely generated.  Let $c = \pdim_{A}(M)$.  We have 
\[
c = \max \{ i | \Ext^i_{A}(M, L) \neq 0\ \text{for some}\ A\text{-module}\ L \}.
\]
If $F$ is a free $A$-module surjecting onto $L$, from the long exact sequence we obtain $\Ext^c_{A}(M, F) \neq 0$.
Since the terms $P^i$ in the resolution of $M$ are finitely generated, it follows from computing $\Ext$ with the resolution $P$ of $M$ 
that $\Ext^c_A$ commutes with direct sums in the second coordinate, so that $\Ext^c_{A}(M, A) \neq 0$.  Since $\Ext^i_A(M, A) = 0$ for $i \geq \pdim_A(M)$, the result follows.

(2)  The result is trivial if $M = 0$, so assume that $M \neq 0$.  By part (1), we have $\pdim(M) = d$.  Thus we can choose a projective resolution of $M$ of the form $P^{\bullet} \to M$, where 
\[
P^{\bullet} =  P^{-d} \to \dots \to P^0.
\]
Then $\RHom_A(M, A)= \Hom_A(P, A)$ is a complex which lives in complex degrees $0$ to $d$.  By assumption, $\RHom_A(M, A)$ 
has cohomology only in degree $d$, where the cohomology is $N$.  It is now easy to see that 
there is a quasi-isomorphism $\RHom_A(M, A) \cong N[-d]$.

(3) We may again assume that $M \neq 0$. Then $\pdim(M) = d$, and it follows from~\cite[Proposition 3.4]{MM} that $M$ is (graded) perfect. 
Now $\RHom_A(M,A) = \RgrHom_A(M,A)$ by Lemma~\ref{lem:graded equivalence} so that the quasi-isomorphism in part~(2) automatically preserves the grading.  
%
%Note that $M$ has a minimal graded projective resolution of length $c = \pdim(M) \leq d$.  \todo{This should be easy to shrink.} Then it follows from 
%the conditions that $\Ext^i_A(M, A) = 0$ for $i \neq c$ and $\dim_k \Ext^c_A(M, A) < \infty$ that the projectives 
%in the minimal graded resolution of $M$ are actually finitely generated \cite[Proposition 3.4]{MM}.  In other words, 
%$M$ is (graded) perfect (and in fact, by part (1), either $M = 0$ or else we must have $c = d$).  Now 
%$\RHom_A(M,A) = \RgrHom_A(M,A)$ by Lemma~\ref{lem:graded equivalence} and the quasi-isomorphism $\RHom_A(M,A) \cong N[-d]$ given 
%by part (2) automatically preserves the grading.  
\end{proof}

Using the preceding lemma, it is convenient to reformulate the second condition in the definition 
of (graded) twisted Calabi-Yau in terms of the derived category.

\begin{lemma}
\label{lem:derived form}
A (graded) $k$-algebra $A$ is twisted Calabi-Yau if and only if $A$ is homologically smooth and 
$\RHom_{A^e}(A,A^e) \cong U[-d]$ in $D^b(A^e \lMod)$ (respectively, in $D^b(A^e \lGr)$) for some (graded) invertible bimodule $U$.  Moreover, in this case $d =  \pdim_{A^e}(A)$.
\end{lemma}

\begin{proof}
We prove the ungraded version; the proof in the graded case is the same.  It is enough to assume that $A$ is homologically smooth 
and prove that under this assumption, for an invertible bimodule $U$, the following conditions are equivalent:
\begin{itemize}
\item[(ii)] $\Ext^i_{A^e}(A,A^e) \cong \begin{cases} 0, & i \neq d \\ U, & i = d \end{cases}$ 
\item[(ii$'$)] $\RHom_{A^e}(A,A^e) \cong U[-d]$ in $D^b(A^e \lMod)$.  
\end{itemize}
If (ii$'$) holds, 
then by taking cohomology we obtain (ii).   If (ii) holds, then since $A$ is a perfect $A^e$-module we obtain (ii$'$) from Lemma~\ref{lem:extderived}(2).

The last statement follows from Lemma~\ref{lem:extderived}(1).
\end{proof}
\noindent From now on we will use the derived category formulation of the definition of twisted Calabi-Yau without further comment.

\subsection{Basic operations preserving the twisted Calabi-Yau property}

Next we show that the twisted Calabi-Yau property is stable under several basic operations.

Let $Z(A)$ denote the center of an algebra $A$. An $(A,A)$-bimodule $U$ is said to be 
\emph{$Z(A)$-central} if, for all $z \in Z(A)$ and all $u \in U$, we have $zu = uz$. 

\begin{lemma}
\label{lem:Ucentral}
Let $A$ be a twisted Calabi-Yau algebra with Nakayama bimodule $U$.
Then $U$ is a $Z(A)$-central $(A,A)$-bimodule. 
In particular, if $A$ has a Nakayama automorphism $\mu$, then the restriction of $\mu$ to the center 
$Z(A)$ is the identity.
\end{lemma}

\begin{proof}
Fix $z \in Z(A)$ and denote $r = z \otimes 1 - 1 \otimes z\op \in Z(A^e)$. Let $\rho_r \colon A \to A$
denote the $A^e$-module homomorphism given by (left) multiplication by $r$. Because $A$ is
a central $(A,A)$-bimodule, we have $\rho_r = 0$.
It follows from~\cite[Proposition~6.18]{Rotman} that the $Z(A^e)$-module endomorphism of 
$\Ext^i_{A^e}(A,A^e)$ given by multiplication by $r$ is equal to
\[
\Ext^i_{A^e}(\rho_r, A^e) \colon \Ext^i_{A^e}(A,A^e) \to \Ext^i_{A^e}(A,A^e).
\]
Because $\rho_r = 0$, this is the zero morphism. Therefore $U = \Ext^d_{A^e}(A,A^e)$ is 
a central $(A,A)$-bimodule.
\end{proof}

We will discuss a number of operations that preserve the twisted Calabi-Yau property. The first
such operation is the direct sum. 

\begin{proposition}\label{prop:twisted CY product}
For algebras $A$ and $B$, denote $R = A \oplus B$. Then $R$ is twisted Calabi-Yau
of dimension~$d$ if and only if both $A$ and $B$ are twisted Calabi-Yau of dimension~$d$,
in which case the Nakayama bimodule of $R$ is the direct sum of the Nakayama bimodules of 
$A$ and $B$ (and if $A$ and $B$ have Nakayama automorphisms $\mu_A$ and $\mu_B$, 
then $R$ has Nakayama automorphism $\mu_R = \mu_A \oplus \mu_B$).
\end{proposition}

\begin{proof}
First, suppose that both $A$ and $B$ are twisted Calabi-Yau of dimension~$d$ with respective Nakayama bimodules
$U_1$ and $U_2$. Then $R$ is homologically smooth by Proposition~\ref{prop:smooth product}.
Note that $U_1 \oplus U_2$ is an invertible $(R,R)$-bimodule with inverse
$U_1^{-1} \oplus U_2^{-1}$.
Applying Lemma~\ref{lem:product RHom} we have 
\begin{align*}
\RHom_R(R, R^e) &\cong \RHom_{A^e}(A, A^e) \oplus \RHom_{B^e}(B, B^e) \\
&\cong U_1[-d] \oplus U_2[-d] \\
&\cong (U_1 \oplus U_2)[-d],
\end{align*}
so that $R$ is twisted Calabi-Yau of dimension~$d$ with Nakayama bimodule $U_1 \oplus U_2$.
In case $U_1 = {}^1 A^{\mu_A}$ and $U_2 = {}^1 B^{\mu_B}$ for Nakayama automorphisms $\mu_A$ and $\mu_B$, then
we obtain $U_1 \oplus U_2 = {}^1 R^{\mu_R}$ for $\mu_R = \mu_A \oplus \mu_B$.

Conversely, suppose that $R$ is twisted Calabi-Yau of dimension~$d$ with $U = \Ext^d_{R^e}(R,R^e)$;
it follows from Proposition~\ref{prop:smooth product} that both $A$ and $B$ are homologically smooth. 
Lemma~\ref{lem:Ucentral} implies that the central idempotent $z_1 = (1,0) \in A \oplus B = R$ centralizes
$U$, so that also $U \cong U_1 \oplus U_2$ for $U_1 = z_1U$ and $U_2 = (1-z_1)U$, which are 
respectively invertible bimodules over $A$ and $B$. 
Now, using Lemma~\ref{lem:central idempotent}, we have for the central idempotent $z = z_1 \otimes z_1\op \in R^e$ 
that  \[
\RHom_{A^e}(A,A^e) \cong \RHom_{R^e}(zR,R^e) \cong z\RHom_{R^e}(R,R^e) = zU[-d] = U_1[-d].
\]
Thus $A$ is twisted Calabi-Yau of dimension~$d$, and the same is true of $B$ by symmetry.
\end{proof}

We next show that the twisted Calabi-Yau property is a ``geometric'' property, being preserved under 
extension of scalars. (This terminology is as in~\cite[Exercise~II.3.15]{Hartshorne}, for instance.)
Recall the notation $M^K =  M \otimes K$ for an $A$-module $M$.

\begin{proposition}\label{prop:extension of scalars}
Let $A$ be a $k$-algebra and $K/k$ a field extension. 
\begin{enumerate}
\item If $A$ is twisted Calabi-Yau of dimension $d$, then so is $A^K$.
\item Suppose that  $A$ is graded with $\dim_k A_0 < \infty$.  Then $A^K$ is (graded) twisted Calabi-Yau of dimension $d$ if and only if $A$ is.
\end{enumerate}
\end{proposition}

\begin{proof}
(1) Suppose that $A$ is twisted Calabi-Yau of dimension~$d$. From Corollary~\ref{cor:smooth extension} we see
that $A^K$ is homologically smooth.
Set $U = \Ext^d_{A^e}(A,A^e)$.  Then using Lemma~\ref{lem:extension}(2), since $A$ is perfect as an $A^e$ module we have
\begin{align*}
\RHom_{(A^K)^e}(A^K,(A^K)^e) &\cong \RHom_{(A^e) \otimes K}(A \otimes K, (A^e) \otimes K) \\
&\cong \RHom_{A^e}(A,A^e) \otimes K \quad \\
&\cong U \otimes K[-d].
\end{align*}
It follows from taking $0$th cohomology in Lemma~\ref{lem:extension}(2) that $(M \otimes_A N)^K \cong M^K \otimes_{A^K} N^K$ 
for any modules $M \in \rMod A$ and $N \in A \lMod$.  In particular, $U^K$ is an invertible $A^K$-bimodule (with inverse $(U^{-1})^K$), so that $A^K$ is twisted Calabi-Yau of dimension~$d$. 

(2) One direction is a special case of part (1).  To prove the converse, suppose that $A^K$ is twisted Calabi-Yau of dimension~$d$. 
Again, Proposition~\ref{prop:smooth extension} tells us that $A$ is homologically smooth.
Now the argument above shows that 
\[
\RHom_{(A^K)^e}(A^K,(A^K)^e) \cong \RHom_{A^e}(A,A^e) \otimes K, 
\]  
which by assumption is isomorphic to $V[-d]$ for a graded invertible $A^K$-bimodule $V$.  In 
particular, taking cohomology we see that $\RHom_{A^e}(A, A^e) \otimes K$ has cohomology only in 
degree $d$, so the same is true of $\RHom_{A^e}(A, A^e)$ since the functor $- \otimes K$ is exact.  
Thus we have $\Ext_{A^e}^i(A, A^e) = 0$ for $i \neq d$ and since $A$ is perfect as an $A^e$-module, 
Lemma~\ref{lem:extderived}(2) implies that $\RHom_{A^e}(A, A^e) \cong U[-d]$ for some 
$(A, A)$-bimodule $U$.  Clearly $U \otimes K \cong V$.  

Now there are natural evaluation maps of bimodules 
$\phi_l: U \otimes_A \Hom_A(U, A) \to A$ and $\phi_r: \Hom_{A\op}(U, A) \otimes_A U \to A$, and 
$U$ is an invertible bimodule if and only if both $\phi_l$ and $\phi_r$ are bijections.   Since $U^K = V$ is an invertible 
$A^K$-module, it is finitely generated and projective on both sides, and so is certainly perfect as a left and right $A^K$-module.
Then $U$ is perfect as a left and right $A$-module, by Proposition~\ref{prop:smooth extension}(2).  By taking $0$th cohomology 
in Lemma~\ref{lem:extension}(1), we get that the natural map $(\Hom_A(U, A))^K \to \Hom_{A^K}(U^K, A^K)$ is an isomorphism, and similarly on the right.
Using Lemma~\ref{lem:extension}(1)(2) again, extending the base field in $\phi_l$ and $\phi_r$ now yields 
the morphisms of $(A^K, A^K)$-bimodules 
$U^K \otimes_{A^K} \Hom_{A^K}(U^K, A^K) \to A^K$ and 
$\Hom_{(A\op)^K}(U^K, A^K) \otimes_{A^K} U^K \to A^K$, which are bijections since $U^K = V$ is invertible.
By the exactness of base field extension, $\phi_l$ and $\phi_r$ are also bijections, 
so that $U$ is invertible as required.  Thus $A$ is also (graded) twisted Calabi-Yau. 
\end{proof}

Note in the proof above that if $A$ has (graded) Nakayama automorphism $\mu$, then $A \otimes K$
has Nakayama automorphism $\mu \otimes \id_K$.

Another operation that preserves the twisted Calabi-Yau property is the tensor product of algebras.  
Before proving this we need the following technical lemma.
\begin{lemma}
\label{lem:tensor}
Let $R$ and $S$ be $k$-algebras.   For complexes $P \in  D^-(R \lMod), N \in D(R \lMod), Q \in D^-(S \lMod), J \in D(S \lMod)$, 
there is a natural map of complexes of abelian groups
\[
\Phi: \RHom_R(P, N) \otimes \RHom_S(Q, J) \to \RHom_{R \otimes S}(P \otimes Q, N \otimes J), 
\]
respecting any extra module structure obtained if $M, N, H$, or $J$ is a complex of bimodules.  The map $\Phi$ is an 
isomorphism if $P$ and $Q$ are perfect.
\end{lemma}

\begin{proof}
For left $R$-modules $M$ and $N$ and $S$-modules $H$ and $J$, it is easy to see that there is a natural map 
\[
\phi: \Hom_R(M, N) \otimes \Hom_S(H, J) \to \Hom_{R \otimes S}(M \otimes H, N \otimes J), 
\]
respecting any bimodule structures that arise.  A similar argument as in the first paragraph of the proof of Lemma~\ref{lem:extension}
shows that $\phi$ is an isomorphism in case $M$ and $H$ are finitely generated projective modules. 

One may now replace $P$ and $Q$ by bounded above complexes of projectives and proceed in the same spirit as the proof Lemma~\ref{lem:extension}, obtaining the desired isomorphism if $P$ and $Q$ are perfect. We omit the details.
\end{proof}

\begin{proposition}\label{prop:twisted CY tensor}
If $A$ and $B$ are (graded) twisted Calabi-Yau algebras of respective dimensions~$d_1$ 
and~$d_2$, then $A \otimes B$ is a (graded) twisted Calabi-Yau algebra of 
dimension~$d_1 + d_2$ whose Nakayama bimodule is the tensor product of the 
Nakayama bimodules of $A$ and $B$.
If $A$ and $B$ have Nakayama automorphisms $\mu_1, \mu_2$, respectively, then 
$A \otimes B$ has Nakayama automorphism $\mu_1 \otimes \mu_2$.
\end{proposition}
\begin{proof}
Proposition~\ref{prop:smooth tensor} implies that $R = A \otimes B$ is homologically smooth.
%Set $U_1 = \Ext^{d_1}_{A^e}(A, A^e)$ for $i = 1,2$, 
Let $U_1$ and $U_2$ respectively denote the Nakayama bimodules of $A$ and $B$,
which are invertible bimodules and therefore projective on each side. Now applying Lemma~\ref{lem:tensor} in the second isomorphism below, 
which gives an isomorphism since $A$ is perfect over $A^e$ and $B$ is perfect over $B^e$, 
we obtain
\begin{align*}
\RHom_{R^e}(R,R^e) &\cong \RHom_{A^e \otimes B^e}(A \otimes B, A^e \otimes B^e) \\
&\cong \RHom_{A^e}(A, A^e) \otimes \RHom_{B^e}(B, B^e) \\
&\cong U_1[-d_1] \otimes U_2[-d_2] = (U_1 \otimes U_2)[-(d_1 + d_2)], 
\end{align*}
where these isomorphisms hold as complexes of right $A^e$-modules.  
Thus we see that $A$ is twisted Calabi-Yau of the desired dimension, and it is easy to verify 
the behavior of Nakayama automorphisms from the above.
\end{proof}

A fundamental tool for the study of twisted Calabi-Yau algebras, and a key motivation behind their
definition, is Van den Bergh duality~\cite{V, V:erratum}. Readers may find an excellent survey of this 
topic in~\cite{Kra}. We record the derived version of this duality below.

\begin{lemma}\label{lem:derived VdB}
Suppose $A$ is (graded) twisted Calabi-Yau of dimension~$d$, with $U = \Ext^d_{A^e}(A,A^e)$. 
Then for any complex of (graded) left $(A^e, C)$-bimodules $M$, one has  a (graded) isomorphism
\[
\RHom_{A^e}(A,M) \cong (U \otimes^L_{A^e} M)[-d].
\]
as complexes of right $C$-modules.  
\end{lemma}

\begin{proof}
This is the same as in~\cite[Theorem~1]{V}, but without passing to cohomology. 
Beginning with Lemma~\ref{lem:moving tensor}, we have:
\begin{align*}
\RHom_{A^e}(A,M) &\cong \RHom_{A^e}(A,A^e) \otimes^L_{A^e} M \\
&\cong U[-d] \otimes^L_{A^e} M \\
&\cong (U \otimes^L_{A^e} M)[-d]. \qedhere
\end{align*}
\end{proof}

Adapting Ginzburg's observation in~\cite[Remark~3.4.2]{G}, we may apply Van den Bergh duality to
show that the property of being twisted Calabi-Yau of dimension~$d$ is Morita invariant, by providing
an alternate characterization of such algebras.

\begin{proposition}
\label{prop:CY Morita}
Let $A$ be an algebra and $U$ an invertible $(A,A)$-bimodule. Then $A$ is twisted Calabi-Yau of 
dimension~$d$ with Nakayama bimodule $U$ if and only if $A$ is homologically smooth and, for 
every algebra $C$ and for $0 \leq i \leq d$, one has natural isomorphisms of functors
\[
\Tor_i^{A^e}(U,-) \cong \Ext^{d-i}_{A^e}(A,-) \colon (A^e \otimes C\op) \lMod \to \rMod C.
\]
Consequently, the property of being twisted Calabi-Yau of dimension~$d$ is preserved under
$k$-linear Morita equivalence.
\end{proposition}

\begin{proof}
If $A$ is twisted Calabi-Yau of dimension~$d$, then $A$ is homologically smooth by definition, 
and satisfies the above isomorphisms by passing to cohomology in Lemma~\ref{lem:derived VdB}.  
Conversely, suppose that $A$ is perfect as an $A^e$-module and satisfies the isomorphisms above. 
Evaluating these functors at the $(A^e,A^e)$-bimodule
$A^e$ yields isomorphisms $\Ext^{i}_{A^e}(A,A^e) \cong \Tor_{d-i}^{A^e}(U,A^e)$ as right
$A^e$-modules for $0 \leq i \leq d$.
For $i < d$ one has 
\[
\Ext^i_{A^e}(A,A^e) \cong \Tor_{d-i}^{A^e}(U,A^e) = 0
\]
by flatness of $A^e$. For $i = d$ we have
\[
\Ext^d_{A^e}(A,A^e) \cong \Tor_0^{A^e}(U,A^e) = U \otimes_{A^e} A^e \cong U, 
\]
as right $A^e$-modules.
Thus $A$ is twisted Calabi-Yau of dimension~$d$ with Nakayama bimodule $U$.

Now suppose that $A$ is twisted Calabi-Yau of dimension~$d$ with Nakayama bimodule $U$,
and suppose that an algebra $B$ is $k$-linearly Morita equivalent to $A$. Then $B$ is homologically
smooth by Proposition~\ref{prop:smooth Morita}. As discussed in the proof of that proposition,
a $k$-linear equivalence of categories $A \lMod \to B \lMod$ induces a monoidal $k$-linear 
equivalence $F \colon A^e \lMod \to B^e \lMod$. 
Because $A$ and $B$ are the respective tensor units, we have $F(A) \cong B$; also, 
$V = F(U)$ must be a $k$-central invertible $(B,B)$-bimodule.
Given an algebra $C$, $F$ also induces an equivalence of categories
$(A^e \otimes C\op) \lMod \to (B^e \otimes C\op) \lMod$ by simply transporting the right 
$C$-action on a left $A^e$-module $X$ via $F$ as an algebra morphism 
$C\op \to \End_{A^e}(X) \to \End_{B^e}(F(X))$; with a slight abuse of notation we use $F$
to denote this equivalence as well.
Now the isomorphism in the statement induces a natural isomorphism
\[
\Tor_i^{A^e}(U,F^{-1}(-)) \cong \Ext^{d-i}_{A^e}(A,F^{-1}(-)) \colon (B^e \otimes C\op) \lMod \to \rMod C.
\]
Because $F$ is an equivalence, it also preserves the construction of $\Ext$ and $\Tor$ spaces, so that
applying $F$ to the previous isomorphism yields a natural isomorphism of functors
\[
\Tor_i^{B^e}(V,-) \cong \Ext^{d-i}_{B^e}(B,-) \colon (B^e \otimes C\op) \lMod \to \rMod C.
\]
Thus $B$ is twisted Calabi-Yau of dimension~$d$, as desired.
\end{proof}

\subsection{Homological tools for finite-dimensional modules}
This final subsection records some results that are of use for handling finite-dimensional modules
over (graded) twisted Calabi-Yau algebras. It concludes with a characterization of twisted Calabi-Yau
algebras of dimension~0.

The following result is a type of Serre duality formula. The idea of the next result and its 
corollary were extracted from the proof of~\cite[Lemma~4.1]{Keller}.
We will use the notation $X^* = \RHom_k(X,k) = \Hom_k(X,k)$ below to denote the extension
of the $k$-dual to the derived category of $k$-vector spaces.  

\begin{proposition}\label{prop:duality}
Let $A$ be a (graded) twisted Calabi-Yau algebra of dimension~$d$, and set $U = \Ext^d_{A^e}(A,A^e)$.
Let $M$ be a (graded) $(A, B)$-bimodule and $N$ a (graded) $(A, C)$-bimodule, with $M$ finite-dimensional. Then there are
(graded) quasi-isomorphisms
\[
\RHom_A(M,N) \cong M^* \otimes^L_A (U \otimes_A N)[-d] \cong (M^* \otimes_A U) \otimes^L_A N[-d],
\]
and, for each integer $i$, we have
\[
\Ext^i_A(M,N) \cong \Tor_{d-i}^A(M^*, U \otimes_A N) \cong \Tor_{d-i}^A(M^* \otimes_A U, N)
\]
as (graded) $(B, C)$-bimodules.  
\end{proposition}

\begin{proof} 
Because $M$ is finite-dimensional over $k$, there is  a natural isomorphism of $(A^e, B\op \otimes C)$-bimodules
$\Hom_k(M,N) \cong N \otimes M^*$. 
Combining this observation with Lemma~\ref{lem:derived Tor and Ext} and 
Lemma~\ref{lem:derived VdB} yields 
\begin{align*}
\RHom_A(M,N) &\cong \RHom_{A^e}(A, \Hom_k(M,N))  \qquad ( \text{by Lemma}~\ref{lem:derived Tor and Ext}(2))  \\
&\cong \RHom_{A^e}(A, N \otimes M^*) \\
& \cong U \otimes^L_{A^e} (N \otimes M^*)[-d] \qquad (\text{by Lemma}~\ref{lem:derived VdB})\\
&\cong M^* \otimes^L_A (U \otimes^L_A N)[-d] \qquad (\text{by Lemma}~\ref{lem:derived Tor and Ext}(1)), 
\end{align*}
and these are quasi-isomorphisms of complexes of right $B\op \otimes C$-modules.
By the definition of twisted Calabi-Yau, $U$ is invertible and hence projective as both a left and a right $A$-module, so we obtain
\[
M^* \otimes^L_A (U \otimes_A N) \cong M^* \otimes^L_A U \otimes^L_A N 
\cong (M^* \otimes_A U) \otimes^L_A N. 
\]
The isomorphisms between $\Ext$ and $\Tor$ are obtained by taking cohomology.

The proof in the graded case is exactly the same.  Note that because $M$ is finite-dimensional and $A$ is homologically smooth, $M$ is perfect by Lemma~\ref{lem:smooth resolution}.  Thus in the graded case there is no difference betweeen $\RHom_A(M,N)$ and $\RgrHom_A(M,N)$, justifying the uniform statement of the result.
\end{proof}

In case both $M$ and $N$ above are finite-dimensional, we obtain the following, more conventional
Serre duality result. 

\begin{corollary}\label{cor:Serre duality}
Keeping the hypotheses of Proposition~\ref{prop:duality}, assume additionally that $N$ is
finite-dimensional. Then we have (graded) quasi-isomorphisms
\[
\RHom_A(M,N)^* \cong \RHom_A(U \otimes_A N, M)[d]
\]
and, for each integer $i$,
\[
\Ext^i_A(M,N)^* \cong \Ext^{d-i}_A(U \otimes_A N, M)
\]
as (graded) $(C, B)$-bimodules.  
\end{corollary}

\begin{proof}
Invoking Proposition~\ref{prop:duality}, we compute as follows:
\begin{align*}
\RHom_A(M,N)^* &\cong \RHom_k(\RHom_A(M,N),k) \\
&\cong \RHom_k(M^* \otimes^L_A (U \otimes_A N), k)[d] \\
&\cong \RHom_A(U \otimes_A N, \RHom_k(M^*,k))[d] \\
&\cong \RHom_A(U \otimes_AN, M)[d].
\end{align*}
As before, the isomorphism relating the $\Ext$ groups is obtained by passing to cohomology.

Again, the same proof applies in the graded case.
\end{proof}

We may apply the above to compute the global dimension of a twisted Calabi-Yau algebra 
that has a nontrivial finite-dimensional representation, including locally finite graded algebras.

\begin{corollary}
\label{cor:twisted CY global dimension}
Let $A$ be a twisted Calabi-Yau algebra of dimension~$d$. If $A$ has a nonzero finite-dimensional 
module, then $\gldim_l(A) = \gldim_r(A) = d$.  
In particular, this holds if $A$ is graded with $\dim_k A_0 < \infty$.
\end{corollary}

\begin{proof}
Recall from Lemma~\ref{lem:derived form} that $\pdim({}_{A^e} A) = d$, so that the left and right
global dimensions of $A$ are both at most~$d$ by Lemma~\ref{lem:smooth resolution}(1).
If $A$ has a finite-dimensional left module $M \neq 0$, then it also has a finite-dimensional right 
module $M^* \neq 0$. So by symmetry, it suffices to show that a finite-dimensional left $A$-module 
$M \neq 0$ has projective dimension~$d$.
But this follows from Proposition~\ref{prop:duality}, since $U = \Ext^d_{A^e}(A,A^e)$ is invertible
and
\[
\Ext^d_A(M,A) \cong \Tor_0^A(M^*, U) = M^* \otimes_A U \neq 0. 
\]
If $A$ is graded with $\dim_k A_0 < \infty$, then $A_0 = A/A_{\geq 1}$ is a nonzero finite-dimensional 
module and the argument above applies.  Alternatively, one may conclude that 
$d = \pdim({}_{A^e} A) = \gldim_l(A) = \gldim_r(A)$ from Proposition~\ref{prop:global dimension}(3), 
since $S = A/J(A)$ must be separable by Theorem~\ref{thm:Rickard}.
\end{proof}

The conclusion of the corollary above may fail if $A$ has no finite-dimensional representations.

\begin{example}
For an integer $n \geq 1$, let $A_n$ denote the $n$th Weyl algebra over a field $k$ of 
characteristic zero. Then $A_n$ is a Calabi-Yau algebra of dimension $d = 2n$ as explained 
in~\cite[Exercise~3.7.11]{Schedler}, 
while the left and right global dimensions of $A_n$ are $n$~\cite[Theorem~5.8]{MR}.
Furthermore, let $D_n$ denote the division algebra of quotients of $A_n$.
Then $D_n$ is a Calabi-Yau algebra of dimension $d = 2n$ as discussed 
in~\cite[Example~1.9(e) and p.~115]{YZ}, while $\gldim(D_n) = 0$.
\end{example}

The following standard result will be useful in the applications of the 
duality results above.
\begin{lemma}
\label{lem:Sstar}
Let $S$ be a finite-dimensional semisimple $k$-algebra.  Then $S$ is a symmetric Frobenius algebra; that is, 
$S^* = \Hom_k(S, k) \cong S$ as $(S, S)$-bimodules.
\end{lemma}
\begin{proof}
See~\cite[\S 16F]{LMR}.
\end{proof}

We may use the duality results above to obtain information about the socle of a graded module.
For a graded left $A$-module $M$ we let $\soc(M)$ denote the \emph{graded socle} of $M$, 
the largest graded semisimple submodule of $M$. Note that $\soc(M) = \{m \in M \mid J(A)m = 0\}$
is the annihilator in $M$ of the graded Jacobson radical $J(A)$.

\begin{proposition}
\label{prop:socle}
Let $A$ be a graded algebra with $\dim_k A_0 < \infty$, which is a twisted Calabi-Yau $k$-algebra of dimension $d$ with Nakayama bimodule $U$. Set $S = A/J(A)$, and let 
${}_A M$ be a graded left $A$-module. Then there is an isomorphism of graded left $S$-modules 
(hence of left $A$-modules)
\[
\Tor_d^A(S, M) \cong U^{-1} \otimes_A \soc(M).
\]
\end{proposition}

\begin{proof}
The finite-dimensional semisimple algebra $S$ satisfies $S^* \cong S$ as an $(S,S)$-bimodule by 
Lemma~\ref{lem:Sstar}, hence also as a graded left $A$-module. 
Because ${}_A S$ is annihilated by $J(A)$, any graded module homomorphism $S \to M$ has
image annihilated by $J(A)$, and therefore has image in $\soc(M)$. Thus we
have $\Hom_A(S, M) = \Hom_S(S, \soc(M)) = \soc(M)$. 
Now applying Proposition~\ref{prop:duality}, we have
\begin{align*}
\Tor^A_d(S_A , {}_A M) &\cong \Ext^0_A(S^*, U^{-1} \otimes_A M) \\ 
&\cong \Ext^0_A(S, U^{-1} \otimes_A M) \\
&= \soc(U^{-1} \otimes_A M) \\
&= U^{-1} \otimes_A \soc(M)
\end{align*}
as left $S$-modules, where we use in the last step that the graded autoequivalence $U^{-1} \otimes_A -$ must 
preserve the graded socle of a module.  
\end{proof}

In particular, it is rare for a twisted Calabi-Yau algebra to have a socle.
\begin{corollary}\label{cor:fd CY}
Let $A$ be a twisted Calabi-Yau algebra of dimension~$d$.
\begin{enumerate}
\item Suppose that $A$ is graded with $\dim_k A_0 < \infty$. If $d > 0$, then $\soc(A) = 0$.
\item If $A$ is a finite-dimensional algebra, then $d = 0$.
\end{enumerate}
\end{corollary}

\begin{proof}
(1) Invoking Proposition~\ref{prop:socle} in the case where $d > 0$ yields
$U^{-1} \otimes_A \soc(A) \cong \Tor^A_d(S,A) = 0$. Because $U$ is invertible, we obtain $\soc(A) = 0$.

(2) If $A$ is finite-dimensional, then we may consider $A$ as a graded algebra with $A = A_0$. 
Because $A$ is artinian, its (graded) socle is nonzero. It follows from part~(1) above that $d = 0$.
\end{proof}

We conclude this section by characterizing twisted Calabi-Yau algebras of dimension~0.

\begin{theorem}\label{thm:CY 0}
For a $k$-algebra $A$, the following are equivalent:
\begin{enumerate}[label=\textnormal{(\alph*)}]
\item $A$ is twisted Calabi-Yau of dimension~0;
\item $A$ is twisted Calabi-Yau and has finite $k$-dimension;
\item $A$ is Calabi-Yau of dimension~0;
\item $A$ is a separable $k$-algebra.
\end{enumerate}
\end{theorem}

\begin{proof}
Clearly (c)$\implies$(a). For (a)$\implies$(d), note that if $A$ is twisted Calabi-Yau
of dimension~0, then $\pdim({}_{A^e} A) = 0$ by Lemma~\ref{lem:derived form},
making $A$ separable by Lemma~\ref{lem:sepdef}(3).

To see that (d)$\implies$(c), suppose that $A$ is separable.  Then ${}_{A^e} A$ is projective, and so $A$ is certainly 
homologically smooth.   As noted earlier, Lemma~\ref{lem:sepdef}(5) shows that a separable algebra must be 
finite-dimensional.  Then we also know that $A^e$ is a finite-dimensional semisimple $k$-algebra, by 
Lemma~\ref{lem:sepdef}(4).  Since a separable algebra must be semisimple by definition, we have $A^* \cong A$ as $(A, A)$-bimodules in Lemma~\ref{lem:Sstar}.
We now calculate that 
\begin{align*}
\Hom_{A^e}(A,A^e) &\cong \Hom_{A^e}(A, \Hom_k(A^e,k)) \\
&\cong \Hom_k(A^e \otimes_{A^e} A, k) \\
&\cong \Hom_k(A,k) \cong A
\end{align*}
as right $A^e$-modules.
Furthermore, $\Ext_{A^e}^i(A,A^e) = 0$ for $i > 0$ as ${}_{A^e} A$ is projective.
Thus $A$ is Calabi-Yau of dimension~0, as desired.  Since we already recalled that a separable 
algebra is finite-dimensional, (d)$\implies$(b) as well.
Finally, (b)$\implies$(a) follows from Corollary~\ref{cor:fd CY}(2). 
\end{proof}

Graded twisted Calabi-Yau algebras of dimension~0 are essentially the same as
(ungraded) Calabi-Yau algebras of dimension~0, since the former must be trivially graded.

\begin{corollary}\label{cor:graded CY 0}
A graded algebra $A$ is (twisted) Calabi-Yau of dimension~0 if and only if $A = A_0$
and $A$ is separable.
\end{corollary}

\begin{proof}
If $A = A_0$ is separable, then it is Calabi-Yau of dimension~0 by Theorem~\ref{thm:CY 0};
because the grading of $A$ is trivial, it is clear that $A$ is graded twisted Calabi-Yau of dimension~0.

Conversely, suppose that $A$ is graded twisted Calabi-Yau of dimension~0.
Theorem~\ref{thm:graded versus ungraded}(1)  shows that $A$ is twisted Calabi-Yau
of dimension~0.
It follows from Theorem~\ref{thm:CY 0} that $A$ is separable, hence a finite-dimensional semisimple 
$k$-algebra. Thus $J(A) = 0$; because $A_{\geq 1} \subseteq J(A)$ for an $\N$-graded algebra, we 
obtain $A = A_0$. 
\end{proof}

\section{Artin-Schelter regularity for locally finite algebras}
\label{sec:AS regular}

In this section, we study the relationship between the twisted Calabi-Yau property and certain
generalizations of the Artin-Schelter regular property for locally finite graded algebras.
It was shown in~\cite[Lemma~1.2]{RRZ1} that a connected graded algebra $A$ is graded
twisted Calabi-Yau if and only if it is Artin-Schelter regular (not necessarily of finite GK dimension).  
Several possible generalizations of Artin-Schelter regularity to the context of locally finite algebras 
have been proposed in the literature.  
In this section we recall some of these and show that several of the most natural generalizations are 
in fact equivalent.
We then show that for any locally finite graded $k$-algebra $A$, under a mild technical condition 
(that $A$ is separable modulo its graded Jacobson radical), these notions of regularity are further
equivalent to the twisted Calabi-Yau condition.

\subsection{Defining generalized Artin-Schelter regular algebras}

For a graded algebra $B$, we write $B \lgr$ for the category of finitely generated graded left $B$-modules; similarly, $\rgr B$ is 
the corresponding category of finitely generated graded right $B$-modules.  
%When $A$ is a locally finite graded $k$-algebra, it is easy to see that $A_0 \lgr$ and $\rgr A_0$ have the
%same objects as  the categories of finite-dimensional graded left and right $A$-modules, respectively.  

We begin with a technical lemma that will help us to relate the different notions of regularity.  
\begin{lemma}
\label{lem:contraequiv}
Let $A$ be a locally finite graded $k$-algebra with $\grgldim(A) = d$. 
Let  $\X$ be the full subcategory of $A_0 \lgr$ consisting of modules $M$ such that 
$\Ext_A^i(M, A) = 0$ for $i \neq d$ and $\Ext^d_A(M, A) \cong N$ for some $N \in \rgr A_0$. 
Similarly, let $\Y$ be the full subcategory of $\rgr A_0$ consisting of  modules $N$ such 
that $\Ext_{A\op}^i(N, A) = 0$ for $i \neq d$ and $\Ext_{A\op}^d(N, A) = M$ for some $M \in A_0 \lgr$.  
\begin{enumerate}
\item  $\X$  can also be described as the full subcategory of $A_0 \lgr$ consisting of modules $M$ such that 
$\RHom_A(M, A)[d] \cong N$ for some $N \in \rgr A_0$.  Similarly, $\Y$ can be described as the 
full subcategory of $\rgr A_0$ consisting of modules $N$ such that   $\RHom_{A\op}(N, A)[d] \cong M$ for some $M \in A_0 \lgr$.
The objects in $\X$ and $\Y$ are perfect. 
\item $\X$ and $\Y$ are closed under extensions and direct summands.
\item The functors 
\begin{align*}
\RHom_A(-, A)[d] &\colon \X\op \to \Y \quad \mbox{and} \\
\RHom_{A\op}(-, A)[d] &\colon \Y\op \to \X
\end{align*}
are mutually inverse, yielding a contravariant equivalence of categories between $\X$ and $\Y$. 
\item If $M \in \X$ is an $(A_0, C)$-bimodule for some $k$-algebra $C$, then we have $\RHom_{A\op}(\RHom_A(M, A),A) \cong M$ as 
$(A_0, C)$-bimodules.  
\end{enumerate}
\end{lemma}

\begin{proof}
(1)  Since $M$ is bounded below, we have $\pdim(M) = \grpdim(M) \leq d$ by 
Proposition~\ref{prop:global dimension}.  This now follows directly from Lemma~\ref{lem:extderived}(3) and its right-sided analog.

(2)  This is an easy consequence of the definitions of $\X$ and $\Y$ and the long exact sequence in Ext.

(3) Let $M \in \X$.    Since $M$ is perfect of projective dimension at most $d$ by part (1), $M$ is quasi-isomorphic to a complex of graded $A$-modules $P^{\bullet}$ of the form $0 \to P^{-d} \to \dots \to P^{-1} \to P^0 \to 0$, where each $P^i$ is a finitely generated graded projective 
$A$-module.  Then $N = \RHom_A(M, A)[d]$ is identified with a complex $Q^\bullet$ of the form 
$0 \to Q^0 \to Q^1 \to \dots \to Q^d \to 0$, where $Q^i = \Hom_A(P^i, A)$.  Since the $P^i$ are finitely generated, it readily follows that 
$\RHom_{A\op}(\RHom_A(P^{\bullet},A), A)) \cong P^{\bullet}$, which means that 
$\RHom_{A\op}(N[-d], A) \cong M$, or equivalently  $\RHom_{A\op}(N, A)[d] \cong M$. 
Thus we have $N \in \Y$.  

The whole argument can be repeated starting with $N \in \Y$ and $Q^\bullet$ 
to get that $M = \RHom_{A\op}(N, A)[d]$ is in $\X$, with $\RHom_A(M, A)[d] \cong N$.
It is now easy to see that the functors $\RHom_A(-, A)[d]$ and $\RHom_{A\op}(-, A)[d]$
yield a contravariant equivalence between $\X$ and $\Y$ as claimed.  

(4) Given $c \in C$, one way to calculate the left action of $c$ on $\Ext^d_A(M, A)$ is as follows.  Right multiplication by $c$ on $M$ gives a morphism $r_c: M \to M$ in $A \lGr$.  Let $P^{\bullet}$ be the graded perfect complex quasi-isomorphic to $M$ as in part (3).  Then $r_c$ lifts to a morphism of complexes $\widehat{r_c}: P^{\bullet} \to P^{\bullet}$.  
Applying $\Hom_A( - , A)$ to the morphism of complexes, we get a morphism of complexes $\widehat{r_c}^*: Q^{\bullet} \to Q^{\bullet}$ where $Q^{\bullet} = \Hom_A(P^{\bullet}, A)$; then taking cohomology induces a morphism of right $A$-modules which is $l_c: N \to N$, the left action of $c$ on the Ext group 
$N = \Ext^d_A(M, A)$.  Of course, the analogous process on the other side shows how to calculate the right action of $c$ on $\Ext^d_{A\op}(N, A)$.  Since applying $\Hom_{A\op}(\Hom_A(-, A), A)$ to the morphism of complexes $\widehat{r_c}: P^{\bullet} \to P^{\bullet}$ gives the same morphism of complexes back, we see 
that the right action of $c$ on $\RHom_{A\op}(\RHom_A(M, A), A)$ is the same as the original right action on $M$, as required.
\end{proof}

We now show that a number of conditions that are natural possible generalizations of the AS~regular condition to the non-connected graded case are in fact equivalent.
\begin{theorem}
\label{thm:reg char}
Let $A$ be a locally finite graded $k$-algebra with $\grgldim A = d$ and let $S = A/J$ where $J = J(A)$ is the graded 
Jacobson radical.   
The following conditions on $A$ are equivalent: 
\begin{enumerate}[label=\textnormal{(\alph*)}]
\item $\RHom_A(-, A)[d]$ gives a bijection from the set of graded simple left $A$-modules up to isomorphism to 
the set of graded simple right $A$-modules up to isomorphism;
\item $\RHom_A(-,A)[d]$ gives a contravariant equivalence from $A_0 \lgr$ to $\rgr A_0$; 
\item $\RHom_A(-,A)[d]$ gives a contravariant equivalence from $S \lgr$ to $\rgr S$; 
\item $\RHom_A(S,A)[d] \cong V$ as right $S$-modules, for some invertible graded $(S,S)$-bimodule $V$;
\item[\textnormal{(d$'$)}] $\RHom_A(S,A)[d] \cong V$ as $(S, S)$-bimodules, for some invertible graded $(S,S)$-bimodule $V$;
\item $\RHom_A(A_0, A)[d] \cong (A_0^* \otimes_{A_0} W)$ as right $A_0$-modules, for some invertible graded $(A_0, A_0)$-bimodule $W$;
\item[\textnormal{(e$'$)}]  $\RHom_A(A_0, A)[d] \cong (A_0^* \otimes_{A_0} W)$ as $(A_0, A_0)$-bimodules, for some invertible graded $(A_0, A_0)$-bimodule $W$.
\end{enumerate}
\end{theorem}
\begin{proof}
(a) $\implies$ (b):  This follows from Lemma~\ref{lem:contraequiv}(3) if we can show that in the notation of that lemma, 
$\X = A_0 \lgr$ and $\Y = \rgr A_0$.  Condition (a) gives that $M \in \X$ for each graded simple left module $M$.  
Now $\X$ is closed under extensions as noted in Lemma~\ref{lem:contraequiv}(2).  Since $A_0$ is 
artinian, the objects in $A_0 \lgr$ have finite length, so it follows that $\X = A_0 \lgr$.  By hypothesis every 
graded simple right module $N$ is of the form $\RHom_A(M, A)[d]$ for some simple graded left module $M$, and 
as saw in the proof of Lemma~\ref{lem:contraequiv}, this implies that $\RHom_{A\op}(N, A) \cong M$ and so $N \in \Y$.  
Since $\Y$ is also closed under extensions, similarly we get $\Y = \rgr A_0$ as required.

(b) $\implies$ (a): this is obvious since a contravariant equivalence of Abelian categories preserves simple modules.
 
(b) $\implies$ (c):  This is immediate since $S \lgr$ is the full subcategory of semisimple objects of $A_0 \lgr$, $\rgr S$ is the full subcategory of semisimple objects of $\rgr A_0$, and a contravariant equivalence preserves semisimple objects.

(b) $\implies$ (e):  As is well known, since $A_0$ is a finite-dimensional $k$-algebra, the functor $G = \Hom_k(-, k) = (-)^*$ gives a contravariant equivalence from $A_0 \lgr$ to $\rgr A_0$, with inverse $G^{-1} = \Hom_k(-, k): \rgr A_0 \to A_0 \lgr$.
By (b), the functor $F = \RHom_A(-,A)[d]$ also gives such a contravariant equivalence.  Thus $F \circ G^{-1}: \rgr A_0 \to \rgr A_0$ is 
a (covariant) equivalence of categories.  As such, it must be of the form $- \otimes_{A_0} W$ for some graded invertible $(A_0, A_0)$-bimodule $W$, by Morita theory.  Applying this to the object $A_0^*$, (e) follows.

(c) $\implies$ (d):  This is virtually the same as the proof of (b) $\implies$ (e), working over the ring $S$ instead.  In this case we obtain $\RHom_A(S, A)[d] \cong S^* \otimes_S V$ for some invertible $(S, S)$-bimodule $V$.  However, since $S$ is 
semisimple, we have $S^* \cong S$ by Lemma~\ref{lem:Sstar},
and so $S^* \otimes_S V \cong V$ in $\rgr S$.

(d) $\implies$ (b):  Suppose that $\RHom_A(S,A)[d] \cong V$ as right $S$-modules, where $V$ is an invertible graded $(S,S)$-bimodule.  Let $1 = e_1+ \dots + e_n$ be a decomposition of $1$ as a sum of primitive orthogonal idempotents $e_i \in A_0$.   Then 
$S = \bigoplus_{i=1}^n Se_i$ decomposes $S$ as a direct sum of simple graded left modules.  Consider the subcategories 
$\X \subseteq A_0 \lgr$ and $\Y \subseteq \rgr A_0$ of Lemma~\ref{lem:contraequiv}.   By hypothesis, $S \in \X$.
Since $\X$ is closed under direct summands by Lemma~\ref{lem:contraequiv}(2), we get that all simple left $A_0$-modules are in $\X$.  
As in the proof of Lemma~\ref{lem:contraequiv}, we also get $V \in \Y$.  Since $V$ is invertible, it must be a right generator over $S$, 
which forces it to include each indecomposable projective right $S$-module as a summand. This is equivalent to
saying that it must contain every simple right $A_0$-module as a direct summand.
Since $\Y$ is also closed under summands, every simple right $A_0$-module is in $\Y$.  Now as in the argument for 
(a) $\implies$ (b), since $\X$ and $\Y$ are closed under extensions we get $\X = A_0 \lgr$ and $\Y = \rgr A_0$ and 
(b) follows.

(e) $\implies$ (b):  In this case we have $\RHom_A(A_0, A)[d] \cong A_0^* \otimes_{A_0} W$, as right $A_0$-modules, for some 
invertible graded $(A_0, A_0)$-bimodule $W$. Again we consider the subcategories $\X \subseteq A_0 \lgr$ and 
$\Y \subseteq \rgr A_0$ of Lemma~\ref{lem:contraequiv}.   In this case the hypothesis implies that $A_0 \in \X$.  
Since $\grgldim A = d$, we have $\gldim A_0 \leq d$ by Lemma~\ref{lem:degree zero dimension}.  We claim 
now that every $M \in A_0 \lgr$ is in $\X$.  We prove the claim by induction on projective 
dimension over $A_0$.    Since every finitely generated graded projective is a direct summand of a finite rank graded free module, 
and $\X$ is closed under graded shifts, direct sums and direct summands, from $A_0 \in \X$ we get $P \in \X$ for each graded projective $P \in A_0 \lgr$.  If all $M \in A_0 \lgr$ of projective dimension $\leq e$ are in $\X$, with $e < \gldim A_0$, suppose that  $M' \in A_0 \lgr$ has $\pdim M' = e +1$.  Consider the short exact sequence of graded modules 
$0 \to K \to P \to M' \to 0$, where $P \to M'$ is a projective cover of $M'$ in $A_0 \lgr$.  Then $\pdim K \leq e$ and so 
$K \in \X$; since $P \in \X$ also, now the long exact sequence in Ext easily implies that $M' \in \X$, completing the induction step.  Thus $\X = A_0 \lgr$ as claimed.   

Now Lemma~\ref{lem:contraequiv} gives an equivalence of categories 
$F: A_0 \lgr \to \Y$, where $F = \RHom_A(-, A)[d]$.  In particular, $\Y$ must be an Abelian category.  
For any object $N \in \Y$, since 
it has finite length as an $A_0$-module, it must have finite length in the category $\Y$, and clearly $\on{length}_\Y(N) \leq \on{length}_{A_0\op}(N)$.
Now note that setting $N = F(A_0) = A_0^* \otimes_{A_0} W$, we have 
\[
\on{length}_{A_0}(A_0) = \on{length}_{A_0\op}(A_0^*) = \on{length}_{A_0\op}(N) \geq \on{length}_\Y(N),
\]
using that taking duals preserves length, as does the autoequivalence $- \otimes_{A_0} W$.
On the other hand, since $F$ is a contravariant equivalence, it also preserves length and so $\on{length}_{A_0}(A_0) = \on{length}_\Y(N)$.  Thus all terms in the displayed equation are equal, and in particular $\on{length}_{A_0\op}(N) = \on{length}_\Y(N)$.  
This means that $N$ has the same composition series over $A_0\op$ as it does in $\Y$, so all of the simple $\Y$-objects occurring 
as composition factors of $N$ are also simple over $A_0\op$.  
Finally, since each of the $n$ simple $A_0$-modules up to isomorphism is a composition factor of $A_0$, 
each of the $n$ simple $\Y$-objects is a composition factor of $F(A_0) = N$.  Thus every simple $\Y$-object 
is also simple over $A_0\op$.
Since $A_0\op$ also has $n$ simple objects, we conclude that $\Y = \rgr A_0$, and (b) follows.

(e$'$) $\implies$ (e) and (d$'$) $\implies$ (d) are obvious.  

(e) $\implies$ (e$'$):  We have $\RHom_A(A_0, A) \cong A_0^* \otimes_{A_0} W[-d]$ as right modules, for some invertible $(A_0, A_0)$-bimodule $W$.  Let $U =  \Ext^d_A(A_0, A)$, which is an $(A, A)$-bimodule isomorphic on the right to $A_0^* \otimes_{A_0} W$.  We know that (b) holds since (e) $\implies$ (b).  By the proof of Lemma~\ref{lem:contraequiv}, the inverse of the contravariant 
equivalence $\RHom_A(-, A)[d]: A_0 \lgr \to \rgr A_0$ is $\RHom_{A\op}(-, A)[d]: \rgr A_0 \to A_0 \lgr$.  
Moreover, taking $C = A_0$ in Lemma~\ref{lem:contraequiv}(4), we obtain $\Ext^d_{A\op}(U, A) \cong A_0$ 
as $(A_0, A_0)$-bimodules.  
Suppose that $x \in A_0$ satisfies $xU = 0$.  Inspecting the manner in which the induced right action of $x$ 
on $\Ext^d_{A\op}(U, A)$ is obtained in the proof of Lemma~\ref{lem:contraequiv}(4), we find that 
right multiplication by $x$ on $\Ext^d_{A\op}(U, A) \cong A_0$ is also $0$.
It follows that $x = 1x = 0$.  So $U$ is a torsionfree left $A_0$-module.  

Thus $U$ is an $(A_0, A_0)$-bimodule that is isomorphic as a right $A_0$-module to
$A_0^* \otimes_{A_0} W$, and which is torsionfree on the left. We may view the left $A_0$-module structure 
on $U$ as being given by an algebra homomorphism 
\[
A_0 \to \End_{A\op}(U_A) \cong \End_{A\op}(A_0^* \otimes_{A_0} W).  
\]
Since $W$ is invertible, clearly 
$\End_{A\op}(A_0^* \otimes_{A_0} W) \cong \End_{A\op}(A_0^*)$.  Since $( - )^*$ is a contravariant 
equivalence $A_0 \lgr \to \rgr A_0$, we have 
$\End_{A\op}(A_0^*) \cong (\End_{A_0}(A_0))\op \cong ((A_0)\op)\op \cong A_0$.  Thus 
the left structure of $U$ is given by an algebra homomorphism $\sigma : A_0 \to A_0$, and the fact that $U$ is torsionfree on the left implies that $\ker \sigma = 0$.  Since $A_0$ is a finite-dimensional algebra, $\sigma$ is an isomorphism.  We conclude from this that 
$U \cong {}^{\sigma} (A_0^* \otimes_{A_0} W)$ as $(A_0, A_0)$-bimodules.  This is the same as $^{\sigma}(A_0^*)^1 \otimes_{A_0} W$.
Now it is easy to check that $^{\sigma}(A_0^*)^1 \cong {}^1(A_0^*)^{\sigma^{-1}}$, so our bimodule is isomorphic 
to $A_0^* \otimes_{A_0} {}^{\sigma^{-1}}(W)^1$.  Letting $W' = {}^{\sigma^{-1}}(W)^1$, we have that 
$U \cong A_0^* \otimes_{A_0} W'$ as $(A_0, A_0)$-bimodules, where $W'$ is graded invertible.  Thus (e$'$) holds.

(d) $\implies$ (d$'$): This is analogous to the proof of (e) $\implies$ (e$'$), but a bit easier since $S^* \cong S$; we leave it to the reader.
\end{proof}

\begin{definition}
\label{def:AS regular}
Let $A$ be a locally finite graded $k$-algebra.  If $A$ satisfies the equivalent conditions in 
Theorem~\ref{thm:reg char}, we say that $A$ is a \emph{generalized AS~regular algebra of dimension~$d$}, or 
sometimes just \emph{AS~regular}.  If $A \otimes K$ is a (generalized) AS~regular $K$-algebra of dimension~$d$ 
for all field extensions $K$ of $k$, we call $A$ \emph{geometrically AS~regular of dimension~$d$}. 
\end{definition}

\begin{corollary}
\label{cor:otherside}
Let $A$ be a locally finite graded $k$-algebra.  Then $A$ is (geometrically) AS~regular if and only if 
$A\op$ is (geometrically) AS~regular.  In particular, all of the opposite-sided versions of the properties (a)--(e$'$) in 
Theorem~\ref{thm:reg char} are also equivalent characterizations of AS~regularity.
\end{corollary}

\begin{proof}
Assume that condition (b) of Theorem~\ref{thm:reg char} holds, so $F = \RHom_A(-, A)[d]$ gives a contravariant 
equivalence from $A_0 \lgr$ to $\rgr A_0$.   Thus in Lemma~\ref{lem:contraequiv}, we must have 
$\X = A_0 \lgr$ and $\Y = \rgr A_0$.   From the proof of the lemma, it is clear that the quasi-inverse of $F$ 
is given by $G = \RHom_{A\op}(-, A)[d]$, so that $G$ gives a contravariant equivalence from $\rgr A_0 \to A_0 \lgr$; 
that is, condition~(2) also holds for $A\op$.  This shows that if $A$ is generalized AS~regular, then so is $A\op$, and the converse is immediate.  Then for every field extension $k \subseteq K$, $A \otimes_k K$ is generalized AS~regular if and only if 
$(A \otimes_k K)\op \cong A\op \otimes_k K$ is.  Thus $A$ is geometrically AS~regular if and only if $A\op$ is 
geometrically AS~regular.
\end{proof}

\begin{remark}\label{rem:reg affine}
If a locally finite graded algebra $A$ is generalized AS~regular, then it is finitely generated as a $k$-algebra.
Indeed, as discussed in the proof of Theorem~\ref{thm:reg char}, the semisimple left $A$-module $S = A/J(A)$
is perfect. It follows from Lemma~\ref{lem:affine} that $A$ is a finitely generated algebra.
\end{remark}

There are a few other existing generalized notions of AS~regular algebras in the literature, which
we now compare to Definition~\ref{def:AS regular}.

\begin{remark}\label{rem:MV def}
A notion of ``generalized AS~regular algebra" is given by Minamoto and Mori in~\cite[Definition~3.15]{MM} as follows: 
a locally finite graded algebra $A$ is generalized AS~regular of dimension $d$ if $\gldim A = d$, for any simple graded left $A$-module $M$ we have 
$\Ext^i_A(M, A) =0$ for $i \neq d$, and the functors $\Ext^d_A(-, A)$ and $\Ext^d_{A\op}(-, A)$ give inverse bijections 
between the set of simple graded left $A$-modules and simple graded right $A$-modules.  
This property originated in work of Martinez-Villa for graded quotient algebras of path algebras 
in~\cite{MV} (see also his work with Solberg \cite{MVS}).   Using Lemma~\ref{lem:contraequiv}(1), 
it is easy to see that this definition is equivalent to condition~(a) of Theorem~\ref{thm:reg char}, except 
that we assume the potentially weaker condition $\grgldim A = d$ rather than $\gldim A = d$.  There may in fact 
be no examples where these numbers are different, and as we saw in Proposition~\ref{prop:global dimension}(3),
they are the same if $S = A/J$ is separable. We prefer to assume the weaker condition.
Morally, this shows that our definition of generalized AS~regular and the one in \cite[Definition~3.15]{MM} 
are essentially the same.
\end{remark}

\begin{remark}
\label{rem:MM def}
In \cite{MM}, Minamoto and Mori also define a locally finite graded algebra $A$ to be ``AS~regular 
over $A_0$" if $\gldim A  = d$, $\gldim A_0 < \infty$, and one has $\RHom_{A_0}(A_0, A)[d] \cong {}^{\mu} (A_0^*)(\ell)$ 
as complexes of $(A, A)$-bimodules, for some $\ell \in \mb{Z}$ and automorphism $\mu$ of $A$.
See~\cite[Definition 3.1]{MM} and the comments following the definition.  Note that Lemma~\ref{lem:degree zero dimension} shows that the hypothesis $\gldim A_0 < \infty$ is a consequence of $\grgldim A < \infty$.
Thus for $A$ to be AS regular over $A_0$ of dimension~$d$ in the sense of Minamoto and Mori
is equivalent to $A$ being generalized AS~regular in our sense, together with the potentially stronger condition $\gldim A = d$ (rather than our assumption $\grgldim A = d$)
and the condition that the invertible bimodule in Theorem~\ref{thm:reg char}(e$'$) is of the particular form $W = {}^{\mu} A_0(\ell)$.  As we discussed in the 
previous remark, the difference between assuming $\gldim A = d$ or $\grgldim A = d$ is minor and in most cases of interest irrelevant.  
However, it is a further restriction to assume that $W$ has the particular form  ${}^{\mu} A_0(\ell)$.   In Example~\ref{ex:skew group}
below, we provide an instance where the bimodule $W$ that actually occurs for an algebra satisfying Theorem~\ref{thm:reg char}(e$'$) 
is not of this form. Thus our notion of generalized AS~regular is less restrictive than AS~regularity over $A_0$.  

Condition~(b) of Theorem~\ref{thm:reg char} was also shown by Minamoto-Mori to be a consequence of their 
definition of AS~regularity over $A_0$ in~\cite[Proposition 3.5]{MM}; our theorem gives another proof of this.
We also note that there is a notion of \emph{ASF-regular algebra} in~\cite{MM},  defined in terms of graded local cohomology, which Ueyama recently showed~\cite[Corollary~2.11]{Ueyama} to be equivalent for noetherian algebras $A$ to the condition of being AS~regular 
over $A_0$. It seems possible that a suitable ``invertible bimodule twist'' of this property could be equivalent
to the generalized AS regular property, at least for noetherian algebras, but we do not pursue that
possibility here.
\end{remark}

We have introduced conditions~(c), (d), and~(d$'$) of Theorem~\ref{thm:reg char} because we have found
that certain formalisms are easier to handle over the semisimple algebra $S = A/J$ rather than over 
$A_0$ as in~(b), (e), (e$'$).  Thus it is useful to know that the analogous conditions defined relative to $S$ 
still produce equivalent notions.  For this reason we also took condition (d$'$) as the ``official" definition of 
generalized AS~regular in the introduction.

\begin{remark}\label{rem:AS reg Morita}
The generalized AS~regular property is preserved by $k$-linear graded Morita equivalence~\cite[Section~1]{Sierra}.
The proof of this fact is similar to that of Proposition~\ref{prop:CY Morita};
for instance, one can show that either condition~(b) or~(c)  from Theorem~\ref{thm:reg char}
is preserved under a graded $k$-linear equivalence between graded module categories.
However, we do not include the proof here. 
On the other hand, Example~\ref{ex:skew group} below illustrates that the property
of AS~regularity over $A_0$ is not preserved under such an equivalence.
\end{remark}

\subsection{Characterizing twisted Calabi-Yau algebras}

The major goal of this section is to relate the generalized AS regular property to the twisted Calabi-Yau
property for locally finite graded algebras. This will be achieved in Theorem~\ref{thm:twisted CY equivalence} below, 
where we show that the twisted Calabi-Yau property is equivalent to the geometrically AS~regular property.

We require several technical lemmas in preparation for that theorem.

\begin{lemma}
\label{lem:technical}
Let $A$ be a homologically smooth (graded) algebra, and let $X$ be a finite-dimensional $(A, A)$-bimodule.  
\begin{enumerate}
\item $\RHom_{A^e}(A, A^e) \otimes^L_A X \cong \RHom_{A\op}(X^*, A)$ as complexes of (graded) $(A, A)$-bimodules.
\item $X \otimes^L_A  \RHom_{A^e}(A, A^e) \cong \RHom_A(X^*, A)$ as complexes of (graded) $(A, A)$-bimodules.
\end{enumerate}
\end{lemma}

\begin{proof}
We prove part (1), the proof of part (2) being symmetric.   
Consider the bimodule $A \otimes A$ as a left $A^e$-module via the ``outer'' structure $(c \otimes d\op)(a_1 \otimes a_2)
= ca_1 \otimes a_2 d$ and as a right $A^e$-module via the ``inner'' structure $(a_1 \otimes a_2)(c \otimes d\op)
= a_1 c \otimes d a_2$. In this way, we have an isomorphism $A \otimes A \cong A^e$ of $(A^e,A^e)$-bimodules.
Restricting scalars along the natural homomorphism $\id_A \otimes 1 \colon A \to A \otimes A\op = A^e$, this makes
the right $A^e$-module $A \otimes A$ into a right $A$-module under the action 
$(a_1 \otimes a_2) \cdot c = a_1 c \otimes a_2$. Then we have the quasi-isomorphism 
\begin{equation}
\label{eq:tech1}
A^e \otimes^L_A X \cong (A \otimes A) \otimes^L_A X = (A \otimes^L_A X) \otimes A \cong X \otimes A
\end{equation}
as complexes of $(A^e, A^e)$-bimodules.  

Next, we consider the $(A^e,A^e)$-bimodule $X^* \otimes A$ with similar ``outer'' left $A^e$-action and
``inner'' right $A^e$-action. We have the following quasi-isomorphisms of $(A^e, A^e)$-bimodules, where  
the first is an adjoint isomorphism:
\begin{align}
\label{eq:tech2}
\RHom_{A\op}(X^* \otimes A,  A) &\cong \RHom_k( X^*, \RHom_{A\op}(A, A))  \nonumber \\
&\cong \RHom_k( X^*, A) \\
&\cong  X^{**} \otimes A \cong  X \otimes  A. \nonumber
\end{align}
From Lemma~\ref{lem:derived Tor and Ext}(1) we further obtain quasi-isomorphisms
\begin{equation}
\label{eq:tech3}
(X^* \otimes A) \otimes^L_{A^e} A \cong X^* \otimes^L_A A \otimes^L_A A \cong X^*
\end{equation}
as complexes of $(A, A)$-bimodules.  

Finally, since $A$ is homologically smooth, it is a perfect object in the derived category of graded left $A^e$-modules.  
Thus using Lemma~\ref{lem:moving tensor}(2) along with the preceding observations, we have 
isomorphisms of $(A, A)$-bimodules as follows:
\begin{align*}
\RHom_{A^e}(A, A^e) \otimes^L_A X &\cong \RHom_{A^e}(A,  A^e \otimes^L_A X) & \text{by Lemma}\ \ref{lem:moving tensor}(2) \\
& \cong \RHom_{A^e}(A,  X \otimes  A)  & \text{by}\ \eqref{eq:tech1} \\
&\cong \RHom_{A^e}(A, \RHom_{A\op}(X^* \otimes A, A)) & \text{by}\ \eqref{eq:tech2}  \\
&\cong \RHom_{A\op}((X^* \otimes A) \otimes^L_{A^e} A, A)   & \text{by adjointness} \\
&\cong \RHom_{A\op}(X^*, A).  & \text{by}\ \eqref{eq:tech3}
\end{align*} 
\noindent Note that in the first line, Lemma~\ref{lem:moving tensor}(2) as stated only gives an isomorphism of right $A$-modules.  
The left $A$-module structure comes from the right $A\op$ structure of $A^e$, and it is easy to see that it is also preserved 
by the given isomorphism.

In the graded case, since $A$ is perfect as an $A^e$-module and $X^*$ is perfect as a right or left $A$-module by 
Lemma~\ref{lem:smooth resolution}, there is no difference between $\RHom$ and $\RgrHom$ in the statement, and it is routine to see that all of the canonical isomorphisms used above preserve grading.
\end{proof}

\begin{lemma}\label{lem:invertible}
Let $A$ be a graded algebra with $A_0$ finite-dimensional, and denote $J = J(A)$ and $S = A/J$. Let $U$ be a graded 
$k$-central $(A,A)$-bimodule such that both ${}_A U$ and $U_A$ are projective.  Suppose 
that $(S \otimes_A U)J = 0$ and $J(U \otimes_A S) = 0$. 
If $V = U \otimes_A S$ is an invertible $(S, S)$-bimodule, then $U$ is an invertible $(A,A)$-bimodule.
\end{lemma}

\begin{proof}
As in the proof of Lemma~\ref{lem:invertible commutes}(1), from the hypothesis that
$(S \otimes_A U)J = 0 = J(U \otimes_A S)$ we may deduce that 
$UJ = JU$ and that $S \otimes_A U \cong %S \otimes_A U \otimes_A S \cong 
U \otimes_A S = V$
as $(A, A)$-bimodules.   
Since $V$ is invertible, it is a finitely generated module over $S$; thus $U$ is a finitely generated 
$A$-module by Nakayama's Lemma.  

Note that, as complexes of $(A, A)$-bimodules, we have
\begin{align*}
S \otimes_A^L  \RHom_{A\op}(U, U)   
&\cong \RHom_{A\op}(U,S \otimes^L_A U ) \\
&\cong  \Hom_{A\op}(U, S \otimes_A U) \\
&\cong \Hom_{A\op}(S \otimes_A U, S \otimes_A U),
\end{align*}
where the first isomorphism follows from Lemma~\ref{lem:moving tensor} (applied to the opposite side)
since $U_A$ is (finitely generated projective and hence) perfect, the second isomorphism follows from 
projectivity of $U_A$, and the last isomorphism holds since $JU$ is in the kernel of any homomorphism 
$U \to S \otimes_A U \cong U/JU$.  Taking $0$th cohomology we obtain an isomorphism 
\[
S \otimes_A \Hom_{A\op}(U, U)  \cong \Hom_{A\op}(S \otimes_A U, S \otimes_A U) = \Hom_{S\op}(V, V)
\]
which is given by the natural map.

Now that we know $S \otimes_A \End_{A\op}(U) \cong \End_{S\op}(V)$ in the natural way, we obtain a 
commutative diagram
\[
\xymatrix{ A \ar[r]^-{\phi} \ar[d]^{\pi} & \End_{A\op}(U) \ar[d]^{\rho} \\
S \ar[r]_-{\overline{\phi}} & \End_{S\op}(V)}
\] 
where $\phi(a)$ is left multiplication by $a$, the bottom row is formed from applying $S \otimes_A -$ to the top row and using the isomorphism above, and the vertical arrows are the natural quotient maps.

We claim that the projective right module $U_A$ is a generator, or equivalently, that each indecomposable graded projective of $A$ 
occurs as a summand of $U_A$.   Recall that the indecomposable projective graded right $A$-modules are the modules $e_i A$, where 
$1 = e_1 + \dots + e_n$ is a decomposition of $1$ as a sum of primitive orthogonal idempotents in $A_0$.  Each lies over 
a simple $S$-module $e_i S = e_i A \otimes_A S$, and $e_i A \cong e_j A$ if and only if $e_i S \cong e_j S$.   
Since $U \otimes_A S = V$ is a generator over $S$, each of the distinct simple right modules up to isomorphism 
occurs as a summand.   Thus each of the distinct indecomposable graded projective right $A$-modules occurs as 
a summand of $U$.  So $U$ is a progenerator on the right.
Since $V = S \otimes_A U$ as well, by symmetry, $U$ must also be a progenerator on the left.

Now by standard Morita theory, to see that $U$ is invertible it suffices to show that the map $\phi$ in the diagram above is an isomorphism. 
Since $U$ is a left progenerator, $_A U$ is torsionfree and so the map $\phi$ is injective.   The map $\overline{\phi}$ is an isomorphism since $V$ is an invertible $(S, S)$-bimodule; in particular, $\overline{\phi}$ is surjective.  Then $\phi$ is surjective by Nakayama's Lemma.  Hence $\phi$ is an isomorphism as required.
\end{proof}

We need the following fact concerning minimal complexes of projectives.  Similar results are well-known for complexes over 
certain kinds of rings but we are unaware of a reference that works in the generality we need here.
If $P^{\bullet}$ is a complex of graded projective left modules over a locally finite graded algebra $A$ with graded Jacobson radical $J = J(A)$, we say $P$ is \emph{minimal} if $\im d_n \subseteq J P^{n-1}$ for all $n$.  This is equivalent to the complex $S \otimes_A P^{\bullet}$ having all of its differentials equal to $0$, where $S= A/J$.  Note that if $P^{\bullet} \to M$ is a graded projective resolution of the module $M$, then 
it is a minimal graded projective resolution if and only if $P^{\bullet}$ is a minimal complex.

\begin{lemma}\label{lem:Cartan-Eilenberg}
Let $A$ be a locally finite graded algebra, with graded Jacobson radical $J = J(A)$.  Let $S= A/J$ and assume that $S$ is separable.
Let $P^\bullet$  be a bounded complex of graded projective left $A$-modules, where each $P^i$ is left bounded and locally finite, but 
not necessarily finitely generated as an $A$-module.  Then  $P^{\bullet}$ is quasi-isomorphic to 
a bounded minimal complex of projectives $Q^{\bullet}$.
\end{lemma}

This result is well known in the connected graded case. We omit the proof, but the reader is referred to~\cite[Proposition~13.2.6]{Ye} for a proof of a dual version regarding minimal injective complexes.

We are now ready to relate the twisted Calabi-Yau condition to the generalized AS~regular condition 
for locally finite graded algebras. 
Our arguments follow the precedent set in ~\cite{YZ} and~\cite[Lemma~1.2]{RRZ1}.  

\begin{theorem}
\label{thm:twisted CY equivalence}
Let $A$ be a locally finite graded $k$-algebra and set $S = A/J(A)$. 
Then the following are equivalent:
\begin{enumerate}[label=\textnormal{(\alph*)}]
\item $A$ is graded twisted Calabi-Yau of dimension~$d$;
\item $A$ is generalized AS~regular of dimension~$d$ and $S$ is a separable $k$-algebra;
\item $A$ is geometrically AS~regular of dimension~$d$.
\end{enumerate}
Such an algebra has left and right global and graded global dimensions equal to~$d$,
and is a finitely generated $k$-algebra. 
\end{theorem}
\begin{proof}
(a) $\implies$ (c): If $A$ is graded twisted Calabi-Yau of dimension~$d$, then any base field 
extension $A \otimes K$ is twisted Calabi-Yau over $K$ by Proposition~\ref{prop:extension of scalars}. 
Thus to establish~(c), it suffices to assume that $A$ is a locally finite twisted Calabi-Yau algebra $A$ of dimension~$d$
and to show that $A$ is generalized AS~regular of dimension~$d$. 
Let $U = \Ext^d_{A^e}(A, A^e)$, which is 
an invertible $(A, A)$-bimodule by assumption. Recall from Lemma~\ref{lem:Sstar} that
$S \cong S^*$. 
Using Proposition~\ref{prop:duality}, we compute
\[
\Ext^i_A(S,A) \cong \Tor_{d-i}^A(S^* \otimes_A U, A) \cong \Tor_{d-i}(S \otimes_A U, A).
\]
This is equal to zero for $i \neq d$ because $A$ is flat, while for $i = d$ we have the isomorphism of
graded $(A,A)$-bimodules
\[
\Ext^d_A(S,A) \cong \Tor_{0}^A(S \otimes_A U, A) = S \otimes_A U.
\]
By Lemma~\ref{lem:extderived}, we obtain $\RHom_A(S, A) \cong  S \otimes_A U[-d]$ as 
complexes of graded $(A,A)$-bimodules.  
By Lemma~\ref{lem:invertible commutes}(2), $W = S \otimes_A U$ is a graded-invertible
$(S,S)$-bimodule.
Thus $A$ satisfies condition (d$'$) of Theorem~\ref{thm:reg char}, so $A$ is generalized 
AS~regular of dimension $d$. 

(c)$\implies$(b): If $A$ is a geometrically AS~regular $k$-algebra, then clearly it is generalized AS~regular.
For any field extension $K/k$, because the AS~regular algebra $A^K = A \otimes K$ has graded global 
dimension~$d$, it follows from Lemma~\ref{lem:degree zero dimension} that the finite-dimensional 
$K$-algebra $(A^K)_0 = (A_0)^K$ has finite global dimension (at most $d$). Lemma~\ref{lem:fd smooth} implies that 
$A_0$ is homologically smooth over $k$, so that Rickard's Theorem~\ref{thm:Rickard} implies that 
$A_0/J(A_0) = A/J(A) = S$ is separable.

(b)$\implies$(a):
Let $A$ be generalized AS~regular of dimension~$d$, and assume that $S$ is separable over $k$.
By condition~(d$'$) of Theorem~\ref{thm:reg char}, there is an  invertible $(S, S)$-bimodule $V$ such that
$\RHom_A(S, A) \cong V[-d]$ as complexes of $(S, S)$-bimodules.   This implies that $S$ is perfect as a left 
$A$-module by Lemma~\ref{lem:extderived}.  Then by Theorem~\ref{thm:homologically smooth}, $A$ is graded homologically smooth.
Similarly, since we know that the opposite-sided version of condition~(d$'$) of Theorem~\ref{thm:reg char} also holds 
by Corollary~\ref{cor:otherside}, we also have $\RHom_{A\op}(S, A) \cong W[-d]$ as complexes of $(S, S)$-bimodules, for 
some invertible $(S, S)$-bimodule $W$.  

Now applying Lemma~\ref{lem:technical}(1) with $X = S$, recalling that $S \cong S^*$ by Lemma~\ref{lem:Sstar},
we obtain quasi-isomorphisms
\[
\RHom_{A^e}(A, A^e) \otimes_A^L S \cong \RHom_{A\op}(S^*, A) \cong \RHom_{A\op}(S, A) \cong  W[-d]
\]
as complexes of graded $(A, A)$-bimodules.  
Similarly, applying Lemma~\ref{lem:technical}(2) with $X = S$ yields 
$S \otimes_A^L \RHom_{A^e}(A, A^e) \cong  V[-d]$ as complexes of graded $(A, A)$-bimodules. 

Now since $A$ is homologically smooth, $P^{\bullet} = \RHom_{A^e}(A, A^e)$ is a bounded complex, consisting of finitely generated graded 
projective right $A^e$-modules.  In particular, each $P^d$ is left bounded and locally finite.  
Now by a right-sided version of Lemma~\ref{lem:Cartan-Eilenberg}, as a complex of graded right $A$-modules 
$P^{\bullet}$ is quasi-isomorphic to a minimal complex of graded projective right $A$-modules $Q^{\bullet}$.  
In particular, the complex $Q^{\bullet} \otimes_A S$ has zero differentials. 
Since this complex is quasi-isomorphic to $W[-d]$ as complexes of graded right $A$-modules, 
comparing cohomology gives $Q^d \otimes_A S \cong W$ and $Q^i \otimes_A S  = 0$ 
for $i \neq d$, so that $Q^i = 0$ for $i \neq d$ by Nakayama's Lemma.  In particular, $Q^{\bullet}$ 
and consequently $\RHom_{A^e}(A, A^e)$ have cohomology only in degree $d$.  
By Lemma~\ref{lem:extderived}, $\RHom_{A^e}(A, A^e) \cong U[-d]$ for the $(A, A)$-bimodule $U = H^d(P^{\bullet}) = \Ext^d_{A^e}(A, A^e)$.  
The quasi-isomorphisms above show that $U \otimes_A S \cong W$ as $(A, A)$-bimodules.

A symmetric argument shows that $S \otimes_A U \cong V$ as $(A, A)$-bimodules; 
in particular, $J(U \otimes_A S) = 0$ and $(S \otimes_A U)J= 0$.
We know that $U$ is projective as a right $A$-module, since $U \cong Q^d$ as right modules, in the notation of the previous paragraph; by symmetry, it is also projective as a left $A$-module.  Thus the hypotheses of Lemma~\ref{lem:invertible} hold, and we conclude that $U$ is an invertible $(A, A)$-bimodule (and, of course, that $V \cong W$).   This establishes~(a).

Finally, when~(a) holds, then since $A$ is finitely generated as an algebra by Lemma~\ref{lem:smooth resolution}(3), 
its graded and ungraded global dimensions are equal by Proposition~\ref{prop:global dimension},
and these dimensions are in fact equal to~$d$ by Corollary~\ref{cor:twisted CY global dimension}.
\end{proof}

\begin{remark}
If $A$ is graded and locally finite, then the twisted Calabi-Yau property is certainly stricter than the generalized
AS regular property in general.
For if $A = S$ is a finite-dimensional semisimple $k$-algebra that is not separable, 
considered as a graded algebra concentrated in degree zero, then $A$ is easily seen to be generalized 
AS~regular; but Corollary~\ref{cor:graded CY 0} shows that $A$ is not graded twisted Calabi-Yau
of dimension~0, for it fails to be homologically smooth over $k$, as shown in
Example~\ref{ex:inseparable}.  On the other hand, if the field $k$ is perfect, then every finite-dimensional semisimple $k$-algebra is
separable, which makes the twisted Calabi-Yau and generalized AS regular properties equivalent in
Theorem~\ref{thm:twisted CY equivalence}.
\end{remark}

In~\cite{Bo}, Bocklandt takes the approach of defining Calabi-Yau algebras in terms of derived 
categories of finite-dimensional modules being Calabi-Yau triangulated categories~\cite{Keller}.
One hopes that the statement of Theorem~\ref{thm:twisted CY equivalence} could be extended to
include a further equivalent condition that the derived category of finite-dimensional graded left 
$A$-modules forms a ``twisted Calabi-Yau triangulated category.'' An attempt to define such 
categories was made in~\cite{RRZ2}, but further work to develop the theory of such categories must
be done before such an equivalence can be carried out. This is partly due to the fact that~\cite{RRZ2}
worked only in the setting of twisted Calabi-Yau algebras that possess a Nakayama automorphism.
But there is still a more fundamental question of whether that definition could or should be modified
to ensure uniqueness of the Nakayama autoequivalence; see~\cite[Section~2.5]{RRZ2}.

%\separate

\subsection{Operations preserving generalized Artin-Schelter regularity}

We showed earlier that the tensor product of two twisted Calabi-Yau algebras is again twisted Calabi-Yau 
in Proposition~\ref{prop:twisted CY tensor}.  The analogous statement for generalized AS~regular algebras 
is of course false; if $E$ is a finite-dimensional non-separable extension of $k$, then $E$ is generalized AS~regular 
of dimension $0$, while $E \otimes_k E$ has infinite global dimension and so certainly cannot be generalized AS~regular.
However, by analogy with the fact that the tensor product of a separable algebra with a semisimple algebra is semisimple,
one might suspect that the tensor product a locally finite graded twisted CY algebra with a generalized AS regular 
algebra is again generalized AS regular. 
This is indeed the case, and can be established using a suitable adaptation of the method of proof used in
Proposition~\ref{prop:twisted CY tensor}.

\begin{proposition}
Let $A_1$ and $A_2$ be graded algebras. If $A_1$ is generalized AS~regular of dimension~$d_1$
and $A_2$ is twisted Calabi-Yau of dimension~$d_2$, then $A_1 \otimes A_2$ is generalized
AS~regular of dimension $d_1 + d_2$.
\end{proposition}

\begin{proof}
Denote $J_i = J(A_i)$ and $S_i = A_i/J_i$ for $i = 1,2$, as well as $J = J(A_1 \otimes A_2)$
and $S = (A_1 \otimes A_2)/J$.
By Theorem~\ref{thm:twisted CY equivalence}, $A_2$ is generalized AS~regular and $S_2$ is separable. 
It follows from Lemma~\ref{lem:enveloping radical} that $J = J_1 \otimes A_2 + A_1 \otimes J_2$ and $S \cong S_1 \otimes S_2$.

From Theorem~\ref{thm:reg char}(d) we have invertible graded $(S_i,S_i)$-bimodules
$V_i$ such that $\RHom_{A_i}(S_i,A)[d_i] \cong V_i$. Note that each $S_i$ is perfect
as a right $A_i$-module by Lemma~\ref{lem:contraequiv}(1). Thus we may apply
Lemma~\ref{lem:tensor} in the following sequence of quasi-isomorphisms:
\begin{align*}
\RHom_{A_1 \otimes A_2}(S, A_1 \otimes A_2) &\cong \RHom_{A_1 \otimes A_2}(S_1 \otimes S_2, A_1 \otimes A_2) \\
&\cong \RHom_{A_1}(S_1, A_1) \otimes \RHom_{A_2}(S_2, A_2) \\
&\cong V_1[-d_1] \otimes V_2[-d_2] = (V_1 \otimes V_2)[-(d_1 + d_2)].
\end{align*}
Since $V_1 \otimes V_2$ has inverse bimodule $V_1^{-1} \otimes V_2^{-1}$ over
$A_1 \otimes A_2$, we find that $A_1 \otimes A_2$ is generalized AS~regular
of dimension $d_1 + d_2$, as desired.
\end{proof}

This directly applies to show, as one might expect, that generalized AS~regularity is preserved
when passing to polynomial rings.

\begin{corollary}
If a graded algebra $A$ is generalized AS~regular of dimension~$d$, then $A[t]$ is generalized
AS~regular of dimension $d+1$.
\end{corollary}

\begin{proof}
As it is well-known that $k[t]$ is Calabi-Yau of dimension~1 (see also Theorem~\ref{thm:CY1}
with $S = k$ and $V = k(-1)$), it follows from the previous proposition that $A[t] \cong A \otimes k[t]$
is generalized AS~regular of dimension $d+1$.
\end{proof}

As mentioned in Section~\ref{sec:intro} and as discussed in~\cite{MM}, regularity properties in the context 
of non-connected locally finite graded algebras provide an avenue to connect noncommutative algebraic geometry
with the study of finite-dimensional algebras. To encourage further investigation into this interesting connection, we pose
the following question that is of a fundamental nature, but which remains largely open.

\begin{question}
For which finite-dimensional algebras $B$ is there a locally finite generalized AS~regular algebra $A$ of 
dimension~$d$ such that $B \cong A_0$?
\end{question}

While the answer for $d = 0$ is obviously the class of semisimple algebras $B$, we do not even
know the answer in case $d = 1$.

Note that the analogous question for twisted Calabi-Yau algebras of dimension~$d$ would follow from
an answer to the above by imposing the extra condition that $B/J(B)$ is separable, thanks to 
Theorem~\ref{thm:twisted CY equivalence}.
In the twisted Calabi-Yau case, the answer for $d = 0$ is given in Theorem~\ref{thm:CY 0} and for
$d = 1$ is given in Theorem~\ref{thm:CY1} below.   See Example~\ref{ex:bigA0} below for a class of examples $A$ 
which are Calabi-Yau of dimension $2$ and have a more interesting $A_0$.

\section{Applications to algebras of dimension at most~2}
\label{sec:applications}

In this section, we will apply the tools developed in preceding sections to the the study of
twisted Calabi-Yau algebras of dimension~$1$ and~$2$. (Recall that the twisted Calabi-Yau algebras of
dimension~0 are fully characterized as separable algebras by Theorem~\ref{thm:CY 0}.)
We will show that a locally finite graded twisted Calabi-Yau algebra $A$ of 
dimension~$d$ is noetherian in case either $d = 1$, or $d = 2$ and $A$ has finite 
GK dimension. 
By establishing this result before attempting any classification of low-dimensional twisted 
Calabi-Yau algebras, we emphasize that this result is of a structural nature, rather than by 
exhaustively listing all possible algebras.
We then conclude with a characterization of graded twisted Calabi-Yau algebras of 
dimension~1 as certain tensor algebras.

%\separate

\subsection{The noetherian property in dimension $d \leq 2$}

Our approach to showing that a graded algebra is noetherian is to \emph{show that
every graded noetherian module is finitely presented.} This is inspired by Cohen-type
arguments, using the key idea from~\cite[Theorem~4.5]{R}.

\begin{lemma}
\label{lem:noetherian test}
A graded $k$-algebra $A$ is left noetherian if and only if every graded noetherian left $A$-module
is finitely presented.
\end{lemma}

\begin{proof}
Suppose that every graded noetherian left $A$-module is finitely presented, and assume 
for contradiction that $A$ is not left noetherian. Then there exists a graded left ideal 
$I \subseteq A$ that is not finitely generated; see~\cite[II.3]{NV1} or~\cite[Theorem~5.4.7]{NV2}. 
Using Zorn's Lemma, we may pass to a 
maximal such $I$. Then $A/I$ is a graded noetherian left $A$-module and thus is finitely 
presented. But then $I$ is finitely generated (by Schanuel's Lemma, for instance),
a contradiction. The converse is clear.
\end{proof}

Now assume that $A$ is a graded algebra of finite graded global dimension $d$ with $\dim_k A_0 < \infty$, and let $M$
be a noetherian graded left module. Consider the minimal graded projective resolution of $M$,
\begin{equation}
\label{projective resolution}
0 \to P^{-d} \to \cdots \to P^{-1} \to P^0 \to M.
\end{equation}
If $A$ is to be left noetherian, then one expects all of the $P^i$ to be finitely generated.
Thus a valid approach to proving that $A$ is left noetherian would be to prove by
descending induction that the $P^i$ are finitely generated.
The next result illustrates how the twisted Calabi-Yau property takes care of the ``base case''
of this proposed inductive argument. 

\begin{proposition}
\label{prop:base case}
Let $A$ be a locally finite graded twisted Calabi-Yau algebra of dimension~$d$, and let $_A M$ be a graded left $A$-module. 
If $\soc(M)$ is a finitely generated module (for example, if $M$ is noetherian), then the term 
$P^{-d}$ in the minimal projective resolution~\eqref{projective resolution} of $M$ is finitely 
generated.
\end{proposition}

\begin{proof}
Set $S = A/J(A)$ and let $U$ be the Nakayama bimodule of $A$.
From Proposition~\ref{prop:socle} we have
\[
\Tor^A_d(S,M) \cong U^{-1} \otimes_A \soc(M).
\]
Because $U^{-1}$ is invertible, $\soc(M)$ is finitely generated as a module if and only if 
$U^{-1} \otimes_A \soc(M)$ is.  Also, the semisimple left $A$-module $U^{-1} \otimes_A \soc(M)$ is a direct sum of simple modules 
over the finite-dimensional algebra $A/J(A) = S$, each of which is finite-dimensional.
Thus $U^{-1} \otimes_A \soc(M)$ is finitely generated if and only if it is finite-dimensional, so that
the claim follows from Lemma~\ref{lem:fg term}.
\end{proof}

This allows us to show that graded twisted Calabi-Yau algebras of dimension~1 are noetherian.

\begin{corollary}\label{cor:noetherian dim 1}
Let $A$ be a locally finite graded twisted Calabi-Yau algebra of dimension~1.
Then $A$ is noetherian.
\end{corollary}

\begin{proof}

Let $_A M$ be a graded noetherian left module with minimal projective resolution
\[
0 \to P^{-1}\to P^0 \to M.
\]
Certainly $P^0$ is finitely generated because $M$ is. Proposition~\ref{prop:base case} implies that
$P^{-1}$ is finitely generated. Hence $M$ is finitely presented.  It follows from Lemma~\ref{lem:noetherian test} that $A$ is left noetherian. By symmetry, $A$ is also right noetherian.
\end{proof}

The dimension~2 case requires us to know a bit more information about the GK~dimension
of graded projective modules.
In the companion paper~\cite{RR2}, we study in detail the basic properties of the Gelfand-Kirillov dimension 
of graded twisted Calabi-Yau algebras. There we establish the following fact.

\begin{lemma} 
\label{lem:same GK}
(See \cite[Proposition 7.4]{RR2}.)
Let $A$ be a graded, locally finite, twisted Calabi-Yau $k$-algebra of dimension~2.  Suppose that $\GK(A) < \infty$.  Let 
\[
P_1 \overset{d_1}{\longrightarrow} P_2 \overset{d_2}{\longrightarrow} P_3
\]
 be an exact sequence of projectives in $\rGr A$. If $P_1$ and
$P_3$ are finitely generated, then so is $P_2$.
\end{lemma}

We now proceed to show that graded twisted Calabi-Yau algebras of global dimension~2 with finite GK
dimension are noetherian. 

\begin{theorem}\label{thm:noetherian dim 2}
Let $A$ be a locally finite graded twisted Calabi-Yau algebra of dimension~2. 
Then $A$ is noetherian if and only if it has finite GK dimension. 
\end{theorem}

\begin{proof}
Assume $\GK(A) < \infty$.
Let $M$ be a noetherian graded left $A$-module with minimal projective resolution
\[
0 \to P^{-2} \to P^{-1} \to P^0 \to M.
\]
Because $M$ is finitely generated, so is $P^0$.   By Proposition~\ref{prop:base case}, we see that
$P^{-2}$ is also finitely generated.   Now by Lemma~\ref{lem:same GK}, $P^{-1}$ is
finitely generated, so that $M$ is finitely presented. Thus $A$ is left noetherian by
Lemma~\ref{lem:noetherian test}; by symmetry, $A$ is also right noetherian.

Conversely, if $A$ is noetherian, then $\GK(A) < \infty$ by~\cite[Proposition 4.4]{RR2}.
\end{proof}

\begin{remark}
A direct consequence of Theorem~\ref{thm:noetherian dim 2} and
Proposition~\ref{prop:extension of scalars} is that if $A$ is a locally finite graded twisted Calabi-Yau
$k$-algebra of dimension $d \leq 2$ having finite GK dimension, then for every field extension $K/k$
the algebra $A \otimes K$ is noetherian.  Following our earlier terminology, this 
property could be called ``geometrically noetherian", but the term ``stably noetherian" has also been used 
in the literature~\cite{Bell}.  This makes it tempting to posit that such algebras are in fact strongly noetherian, that is, 
that $A \otimes R$ is noetherian for all commutative noetherian $k$-algebras $R$.   
To this end, it would also be interesting to study graded twisted Calabi-Yau $R$-algebras over a 
general commutative ring $R$, under the assumption that $A = \bigoplus_{n=0}^\infty A_n$
is graded with each $A_n$ a finitely generated projective $R$-module. It seems likely that a 
number of results proved in the preceding sections could generalize (with suitable 
modification) to this general setting. 
\end{remark}

We also wonder whether the previous theorem can be extended to the ungraded case.

\begin{question}
Let $A$ be a \emph{(not necessarily graded)} twisted Calabi-Yau algebra of dimension $d \leq 2$. 
If $A$ has finite GK dimension, must $A$ be noetherian?  
\end{question}

%\separate

\subsection{Twisted CY-1 algebras are tensor algebras}

Next, we give a structural characterization of locally finite graded twisted Calabi-Yau 
algebras of dimension $d=1$ as certain tensor algebras.
Given a semisimple algebra $S$ and a finite-dimensional $(S, S)$-bimodule $V$, the tensor algebra is
$T_S(V) = \bigoplus_{n \geq 0} V^{\otimes n}$, where
$V^{\otimes n} = V \otimes_S V \otimes_S \dots \otimes_S V$ is the $n$-fold tensor power over $S$
with the usual convention $V^{\otimes 0} = S$, and the multiplication is induced by the tensor product 
over $S$.   
Note that if $V$ is positively graded, then $A = T_S(V)$ is locally finite; conversely, if $A = T_S(V)$ is 
($\mb{N}$)-graded and locally finite, then $V$ must at least be nonnegatively graded.  
In fact, it is a straightforward exercise to see that $T_S(V)$ is nonnegatively graded if and only if
$V^{\otimes N}$ is positively graded for some $N \geq 1$. 
Any graded $A$-module is also a graded $S$-module via the inclusion 
$S = V^{\otimes 0} \subseteq A$.  Furthermore, we claim that
\[
J(A) = \bigoplus_{n \geq 1} V^{\otimes n} = V \otimes_S A = A \otimes_S V.
\]
Indeed, because $S$ is semisimple and $A/(V \otimes_S A) \cong S$, we must have 
$J(A) \subseteq V \otimes_S A$. On the other hand, from $V^{\otimes N} \subseteq A_{\geq 1}$
we have $(V \otimes_S A)^N = V^{\otimes N} \otimes_S A \subseteq A_{\geq 1} \subseteq J(A)$.
As $J(A)$ is semiprime, it follows that $V \otimes_S A \subseteq J(A)$, proving the desired containment.

We have the following technical lemma concerning the structure of $\Ext^i_A(S, A)$ for an algebra $A = T_S(V)$.

\begin{lemma}\label{lem:tensor algebra ext}
Let $S$ be a finite-dimensional semisimple $k$-algebra, let $V$ be a nonnegatively graded finite-dimensional $(S,S)$-bimodule,
and denote $A = T_S(V)$.  Assume that $A$ is locally finite.  Let $\Vhat = \Hom_S(V,S)$ denote the $S$-dual of $V$ as a left module, and
let the image of the natural right-multiplication map $\rho: S \to \End({}_S V)$ be denoted by $S'$.
Then we have the following isomorphisms of $(S,S)$-bimodules:
\begin{enumerate}
\item $\Ext_A^0(S,A) = \Hom_A(S,A) \cong \ann_r(J(A))$. 
\item $\Ext_A^1(S,A) \cong \Vhat \oplus ((\End_S(V)/S') \otimes_S A)$.
\end{enumerate}
\end{lemma}

\begin{proof}
For~(1), it is clear that $\Hom_A(S,A) = \Hom_A(A/J(A),A)$ is naturally identified with the right annihilator
of $J(A)$ in $A$. 

We now establish~(2).   As a graded left $A$-module, we have $J(A) \cong A \otimes_S V$, 
which is projective because ${}_S V$ is projective. Thus ${}_A S$ has a graded projective resolution 
given by 
\[
0 \to J(A) \to A \to S \to 0.
\]
In order to compute $\Ext^1_A(S,A)$, we apply the functor $\Hom_A(-,A)$ to the deleted 
resolution, so that we must examine
\[
\Hom_A(A,A) \overset{\phi}{\longrightarrow} \Hom_A(J(A),A) \to 0, 
\]
where $\phi$ is given by restriction of the right-multiplication morphisms in $A \cong \Hom_A(A,A)$ to
$J(A)$. Then $\Ext^1_A(S,A)$ will be the cokernel of $\phi$.

We now have 
\begin{align*}
\Hom_A(J(A),A) & \cong \Hom_A(A \otimes_S V, A) \\
&\cong \Hom_S(V, \Hom_A(A ,A))  \ \ \text{(by adjointness)} \\
&\cong \Hom_S(V,A) \\
&\cong \Hom_S(V, S \oplus (V \otimes_S A) \\
&\cong \Hom_S(V, S) \oplus \Hom_S(V, V \otimes_S A) \\
&\cong \Hom_S(V, S) \oplus (\Hom_S(V, V) \otimes_S A) \ \ \text{(by Lemma~\ref{lem:moving tensor})} \\
&= \Vhat \oplus (\End_S(V) \otimes_S A),
\end{align*}
as $(S, S)$-bimodules.

One may verify that under this isomorphism, 
the morphism 
\[
\phi \colon A = \Hom_A(A,A) \to \Hom_A(J(A),A)
\]
 corresponds to the map
$A \cong S \otimes_S A \overset{\rho \otimes 1}{\to} \End_S(V) \otimes_S A$ composed with the coordinate
inclusion into $\Vhat \oplus (\End_S(V) \otimes_S A)$.
Taking the cokernel of this map yields the desired $(S,S)$-bimodule isomorphism
$\Ext^1_A(S,A) \cong \Vhat \oplus (\End_S(V)/S') \otimes_S A$.
\end{proof}

This allows us to characterize those tensor algebras $T_S(V)$ as above that are generalized 
AS~regular.

\begin{theorem}\label{thm:tensor algebra}
Let $S$ be a finite-dimensional semisimple $k$-algebra, and let $0 \neq V$ be a nonnegatively graded
finite-dimensional $(S,S)$-bimodule such that $A = T_S(V)$ is locally finite.  Then $A$ is generalized 
AS~regular of dimension $d$ if and only if $V$ is an invertible bimodule, in which case $d = 1$.
\end{theorem}
\begin{proof} 
Suppose that $A$ is generalized AS regular of dimension $d$.  Since $V$ is nonzero,
we see that $\widehat{V} = \Hom_S(V, S) \neq 0$ and so $\Ext^1_A(S, A) \neq 0$ by Lemma~\ref{lem:tensor algebra ext}(2).
Using condition~(d) of Theorem~\ref{thm:reg char}, we see that necessarily $d = 1$; in addition, $\Hom_{A}(S, A) = 0$ and $\Ext_{A}^1(S, A)$ is an invertible $(S, S)$-bimodule.  By Corollary~\ref{cor:otherside}, 
the right-sided versions of these conditions also hold.  In particular, $\Hom_{A\op}(S, A) = 0$ and 
so $A$ has no right socle by Lemma~\ref{lem:tensor algebra ext}(1).  
Then $A$ has no nonzero finite-dimensional right ideals.  In particular, every indecomposable 
graded projective left $A$ module $e_i A$ satisfies $\dim_k e_i A = \infty$, where $1 = e_1 + \dots + e_n$ is a decomposition 
of $1$ into primitive orthogonal idempotents $e_i \in S$.  Any nonzero right $S$-module $M$ is a direct sum of simple modules $e_i S$, 
and $e_i S \otimes_S A = e_i A$.  Thus $\dim_k M \otimes_S A = \infty$ as well.
On the other hand, $W = \Ext_A^1(S, A)$ is an invertible $(S, S)$-bimodule and hence since $\dim_k S < \infty$, 
we have $\dim_k W < \infty$.   Now Lemma~\ref{lem:tensor algebra ext}(2) forces $\End_S(V)/S' = 0$, 
so that $W = \Vhat = \Hom_S(V, S)$ is an invertible $(S, S)$-bimodule.   But then 
$W^{-1} = \Hom_S(W, S) \cong V$ is also invertible.  

Conversely, suppose that $V$ is an invertible $(S, S)$-bimodule.   Then the natural right multiplication map $\rho: S \to \End_S(V)$ is 
an isomorphism.  In particular, its image $S'$ equals $\End_S(V)$ and so $\End_S(V)/S' = 0$.  By Lemma~\ref{lem:tensor algebra ext}(2), 
we have $\Ext_A^1(S, A) \cong \Vhat = \Hom_S(V, S)$, which is an invertible $(S, S)$-bimodule isomorphic to $V^{-1}$.

Next, we have the short exact sequence 
\[
0 \to J(A) \to A \to S \to 0, 
\]
which is a graded minimal projective resolution of $S$, since $J(A) \cong A \otimes_S V$ is projective as noted in the proof of Lemma~\ref{lem:tensor algebra ext}.  Then by Proposition~\ref{prop:global dimension}, 
we have $\grgldim(A) = \pdim({}_A S) = 1$.
Applying $\Hom_A(-, A)$ to the short exact sequence we obtain 
\[
0 \to \Hom_A(S, A) \to \Hom_A(A, A) \to \Hom_A(J(A), A) \to \cdots  
\]
and we claim that the map $A \cong \Hom_A(A,A) \to \Hom_A(J(A),A)$ is injective.  Indeed,
if the restriction of a right multiplication map $\rho_x \colon A \to A$ to $J(A)$ is zero,
then since $J(A) = V \otimes_S A$ we have $V \otimes_S Ax = 0$.  Applying $V^{-1} \otimes_S - $ we get $Ax = 0$ and 
hence $x = 0$.  This proves the claim, and thus $\Hom_A(S, A) = 0$.  Now since $\grgldim(A) = 1$, we 
also have $\Ext^i_A(S, A) = 0$ if $i > 1$.  Using Lemma~\ref{lem:extderived}(3), we find
that $\RHom_A(S, A)[1] \cong V^{-1}$ is an invertible bimodule. Thus by condition~(d) of 
Theorem~\ref{thm:reg char}, $A$ is generalized AS~regular of dimension $1$. 
\end{proof}

We can now characterize locally finite graded twisted Calabi-Yau algebras of dimension $1$.

\begin{theorem}
\label{thm:CY1}
Let $A$ be a graded $k$-algebra with $S = A/J(A)$.  Then the following are equivalent:
\begin{enumerate}
\item $A$ is locally finite twisted Calabi-Yau of dimension $1$.
\item $A \cong T_S(V)$ for a separable $k$-algebra $S$ and an invertible, nonnegatively graded $(S, S)$-bimodule $V$
such that $V^{\otimes N}$ is positively graded for some integer $N \geq 1$.
\end{enumerate}
\end{theorem}
\begin{proof}
$(1) \implies (2)$:  Let $J = J(A)$. Since $A$ is twisted Calabi-Yau of dimension~$1$, 
Theorem~\ref{thm:twisted CY equivalence} implies that $A$ is generalized AS~regular of dimension $1$ and 
that $S$ is separable. Then $S^e$ is also semisimple by Lemma~\ref{lem:sepdef}(2). Thus there exists a
graded left $S^e$-submodule $V$ of $J$ such that $J = V \bigoplus J^2$ as graded $S^e$-modules, that is, as 
graded $(S, S)$-bimodules.  
Similarly, there is a copy of $S \subseteq A$ such that $S \oplus J \cong A$ as graded $(S, S)$-bimodules, 
and we use this to identify $S$ with a graded subalgebra of $A$.  Note that $V \cong J/J^2$ and so $\dim_k V < \infty$ by 
Lemma~\ref{lem:affine} and Lemma~\ref{lem:smooth resolution}(3).  
Clearly $V$ is nonnegatively graded.  

By the universal property of the tensor algebra, there is a unique algebra homomorphism 
$\phi: T_S(V) \to A$ that maps $S$ isomorphically to the given fixed copy of $S$ in $A$ and maps $V$ 
isomorphically to the fixed complement of $J^2$ in $J$. By the choice of $V$ and $S$ as graded
submodules, this is a graded algebra homomorphism. 
By the proof of Lemma~\ref{lem:affine}, we see that $\phi$ is surjective.  

We claim that $\phi$ is an isomorphism. Indeed, denoting $\phi_n = \phi \vert_{V^{\otimes n}}: V^{\otimes n} \to A$,
because $\phi$ is a graded homomorphism it suffices to show by induction that each $\phi_n$ is 
injective. This is easily verified for $n \leq 1$ by the definition of $\phi$. Now let $n \geq 2$ and
suppose that $\phi_{n-1}$ is injective.
As we saw in Theorem~\ref{thm:tensor algebra}, 
the minimal projective graded left resolution of $S$ looks like   
\[
0 \to A \otimes_S V \overset{d_1}{\to}  A \to S \to 0,
\]
where $d_1$ is the natural multiplication map (after identifying $V$ with a subset of $A$ as above).  
  Now $\phi_n$ decomposes as the 
composite
\[
V^{\otimes n} = V^{\otimes (n-1)} \otimes_S V \overset{\phi_{n-1} \otimes 1}{\to} A \otimes_S V \overset{d_1}{\to} A.
\]
Now $\phi_{n-1}$ is injective by the inductive hypothesis. Because $V$ is a flat $S$-module (as $S$
is semisimple), it follows that $\phi_{n-1} \otimes 1$ is injective.  Thus $\phi_n$ is injective as desired.
So we find that $\phi$ is an isomorphism. 

Thus $T_S(V) \cong A$, so that that this tensor algebra is locally finite; as mentioned above, this implies
that $V^{\otimes N}$ is positively graded for some $N \geq 1$.
Now since $A$ is twisted Calabi-Yau of dimension $1$, Theorem~\ref{thm:twisted CY equivalence} implies
that it is also generalized AS~regular of dimension $1$ and that $S$ is separable.  
Finally, $V$ must be invertible by Theorem~\ref{thm:tensor algebra}.

$(2) \implies (1)$:  Because $V^{\otimes N}$ is positively graded, the algebra $A \cong T_S(V)$ is
locally finite. The hypothesis now implies that $A$ is generalized AS~regular of dimension $1$, by Theorem~\ref{thm:tensor algebra}.
Since we also assume that $S$ is separable, $A$ is twisted Calabi-Yau of dimension $1$ by Theorem~\ref{thm:twisted CY equivalence}.
\end{proof}

The proof of $(1) \implies (2)$ above uses separability of $S$ in an essential way.  We don't know the 
answer to the following.
\begin{question}
If $A$ is locally finite graded and generalized AS~regular of dimension $1$, must $A$ be isomorphic to a tensor algebra $T_S(V)$ for 
some graded finite-dimensional $(S, S)$-bimodule $V$?
\end{question}

\section{Examples}

As mentioned earlier, a graded invertible bimodule over a locally finite graded algebra $A$ need 
not have the form  ${}^1 A^\sigma$.  Here is a simple example.

\begin{example}
\label{ex:weirdbimod}
Let $k$ be a field and let $A = k \oplus \M_2(k)$, considered as a graded algebra with $A = A_0$.  
Decompose $1 = e_1 + e_2 + e_3$ as a sum of primitive orthogonal idempotents, where $e_1 \in k$ and $e_2, e_3 \in \M_2(k)$.
Then $A e_1$ is a simple module of $k$-dimension $1$, and $A e_2 \cong A e_3$ are 
simple modules of $k$-dimension $2$.  

Let $U = A e_1 \oplus A e_1 \oplus A e_2$.  Then $U$ is a projective left $A$-module which is 
obviously also a generator, so that  $B = \End_A(U)$ is Morita equivalent to $A$. Since
$\Hom_A(A e_1, A e_2) = 0 = \Hom_A(A e_2, A e_1)$ we in fact have
\[
B = \End_A(U) \cong \End_A((Ae_1)^{\oplus 2}) \oplus \End_A( A e_2) \cong M_2(k) \oplus k \cong A.
\]  
Thus $U$ carries the structure of a graded invertible $(A, A)$-bimodule.
On the other hand, $U$ and $A$ are not isomorphic as left $A$-modules because they have different 
composition factors (even different dimensions over $k$).  Thus $U$ cannot be of the form 
$^1 A^{\sigma}$ for any automorphism $\sigma$.   
\end{example} 

We can use a similar idea as in the previous example to construct an example of a 
twisted Calabi-Yau algebra $A$ whose Nakayama bimodule $U$ is not of the form ${}^1 A^\sigma$
for any automorphism $\sigma$ of $A$.

\begin{example}\label{ex:skew group}
Recall that if $G$ is a finite group acting via automorphisms on an algebra $A$, 
then we can construct the skew group algebra $A \rtimes G$:  as a vector space this is $A \otimes kG$, 
but the multiplication is defined by $(a \otimes g)(b \otimes h) = ag(b) \otimes gh$ for $a, b \in A$ and $g, h \in G$.
When $A$ is locally finite graded and $G$ acts by graded automorphisms, then $A \rtimes G$ is again locally finite graded, 
where the elements of $G$ have degree $0$.

Assume that $k$ has characteristic not equal to $2$.  Let $A = k[x,y] \rtimes \mb{Z}_2$, where the 
non-identity element of $\mb{Z}_2$ acts via the  automorphism $\alpha$ with $\alpha(x) = x, \alpha(y) = -y$.  
The structure of $A$ can be seen to be a quiver algebra modulo certain relations using the McKay
correspondence; we refer readers to~\cite[Corollary 4.1]{BSW} for details.
It is known that $A = kQ/I$, where $Q$ is the McKay quiver of the action; in this case $Q$ has two vertices, one 
loop $x_i$ at each vertex $i$, an arrow $y_1$ from vertex 1 to vertex 2, and an arrow $y_2$ from 
vertex 2 to vertex 1.  The ideal $I$ is generated by $x_1 y_1 - y_2 x_1$ and $x_2 y_2 - y_1 x_2$, where 
we compose arrows from left to right.  We also have $A_0 = ke_1 + ke_2$ where $e_1$ and $e_2$ are the 
trivial paths.  %See \cite[Corollary 4.1]{BSW}.  
By \cite[Theorem 4.1]{RRZ1}, $A$ is twisted 
Calabi-Yau of dimension $2$ (the term skew Calabi-Yau is used there), and its Nakayama bimodule is $U = {}^1 A^{\mu}(2)$, 
where the Nakayama automorphism $\mu$ is also calculated by that result.  Explicitly, $\mu$ acts trivially 
on $x$ and $y$, while $\mu(g) = \hdet(g) g$ for $g \in \mb{Z}_2$.  Here, $\hdet(g)$ is simply the determinant 
of the action of $g$ on $kx + ky$.  Thus if $\mb{Z}_2 = \{1, a\}$ then $\mu(a) = -a$.  Since $e_1 = (1 + a)/2$ and 
$e_2 = (1-a)/2$, we see that $\mu$ switches the two idempotents in $A_0$.

Now let $M_2(A)$ be the 2 by 2 matrix ring over $A$.  By Proposition~\ref{prop:CY Morita} and its proof, we have 
that $M_2(A)$ is twisted Calabi-Yau with Nakayama bimodule $M_2(U)$.  Write $1 = f_1 + f_2 + f_ 3 + f_4$ 
as a sum of primitive orthogonal idempotents in $M_2(A)$, where 
\[
f_1 = \begin{pmatrix} e_1 & 0 \\ 0 & 0 \end{pmatrix}, f_2 = \begin{pmatrix} e_2 & 0 \\ 0 & 0 \end{pmatrix}, 
f_3 = \begin{pmatrix} 0 & 0 \\ 0 & e_1 \end{pmatrix}, f_4 = \begin{pmatrix} 0 & 0 \\ 0 & e_2 \end{pmatrix}.
\]
Then $f_1 M_2(A) \cong f_3 M_2(A) $ and $f_2 M_2(A)  \cong f_4 M_2(A)$ as right $M_2(A)$-modules.   
Thus the idempotent $g = f_1 + f_2 + f_3$ is full, so that the algebra $B = g M_2(A) g$ is Morita equivalent
to $A$. By Proposition~\ref{prop:CY Morita} and its proof, $B$ is also twisted Calabi-Yau of dimension~$2$ 
with Nakayama bimodule $V = g M_2(U) g$.

Consider the right $B$-module structure of $V$.  Note that the Morita equivalence $M \mapsto Mg$ from graded right 
$M_2(A)$-modules to graded right $B$-modules sends $gM_2(U)$ to $V$.  We have 
$gM_2(U) = f_1 M_2(U) \oplus f_2 M_2(U) \oplus f_3 M_2(U)$, where 
$f_1 M_2(U) \cong (e_1 U, e_1 U) \cong f_3 M_2(U)$, and $f_2 M_2(U) \cong (e_2U, e_2 U)$, as modules
given by row vectors.  Moreover, as (ungraded) right $A$-modules, 
we have $e_1 U = (e_1 A)^{\mu}$; since $\mu$ switches the idempotents it is easy to see that in fact 
$e_1 U \cong e_2 A$ as right modules.  Similarly, $e_2 U \cong e_1 A$. Thus in fact the right $M_2(A)$-module
$g M_2(U)$ is, up to isomorphism, a direct sum of 2 copies of $(e_2 A, e_2 A)$ and one copy of $(e_1 A, e_1 A)$.  
On the other hand, $gM_2(A)$ is, up to isomorphism, a direct sum of 1 copy of $(e_2 A, e_2 A)$ 
and 2 copies of $(e_1 A, e_1 A)$.  These same properties pass via the Morita equivalence to the ring $B$; 
and so $B$ has two indecomposable projective right modules $P$ and $Q$ up to isomorphism, such that 
$B \cong P^{\oplus 2} \oplus Q$ and $V \cong P \oplus Q^{\oplus 2}$.  So $V$ is not a free right $B$-module 
and hence it cannot be of the form ${}^1 B^{\sigma}$ for an automorphism $\sigma$.
\end{example}

\noindent The previous example also shows that there really is a difference between the twisted Calabi-Yau property 
for locally finite graded algebras $A$ and Artin-Schelter regularity over $A_0$ as defined by Minamoto and Mori 
(see Remark~\ref{rem:MM def}).  For if $B$ is the algebra constructed in the example, then since $B$ is twisted Calabi-Yau 
it is also generalized AS~regular.  Considering condition (e$'$) in Theorem~\ref{thm:reg char}, 
we have $\RHom_B(B_0, B)[2] \cong (B_0^* \otimes_{B_0} W)$ as $(B_0, B_0)$-bimodules, for some invertible graded $(B_0, B_0)$-bimodule $W$.  By the proof of Theorem~\ref{thm:twisted CY equivalence}, if $V$ is the Nakayama bimodule of $B$ then we 
have $W = B/B_{\geq 1} \otimes V$.  Since $V$ is not a free $B$-module on the right, it is easy to see that 
$W$ is not a free $B_0$-module on the right, and hence $W$ is not of the form ${}^1 B_0^{\sigma}$.
 As mentioned in Remark~\ref{rem:AS reg Morita}, this also shows that the 
property of Artin-Schelter regularity over $A_0$ is not preserved under Morita equivalence, since $B$ is 
Morita equivalent to the algebra $A = k[x,y] \rtimes \Z_2$ above and $A$ is AS~regular over $A_0$ 
with $W = A/A_{\geq 1} \otimes_A U \cong {}^1 A_0^\mu(2)$.

Throughout the paper we have allowed arbitrary locally finite graded algebras $A$, without any 
assumption on the structure of $A_0$.  The most common examples of (twisted) Calabi-Yau algebras 
occur as $A = kQ/I$ for a finite quiver $Q$, where the relations generating $I$ come from taking derivatives 
of a superpotential; see, for example, \cite{BSW}. In these examples, $A$ has a natural grading where 
the arrows in $Q$ have degree $1$, and hence $A_0 \cong k^{\oplus n}$ is semisimple, where $n$ is the 
number of vertices in the quiver.  In many cases, however, it is possible to grade $Q$ in a different manner 
so that the arrows have degrees possibly different from~$1$, in a way that is compatible with the relations.
Choosing some arrows to have degree $0$, one obtains examples with more interesting $A_0$. 
The following is one special case.

\begin{example}
\label{ex:bigA0}
Let $Q$ be a finite connected quiver.  Let $\overline{Q}$ be the double of $Q$, obtained by 
adding an arrow $\alpha^*$ in the opposite direction for each arrow $\alpha$ in $Q$.  The corresponding \emph{preprojective algebra} is $A = k\overline{Q}/(r)$, where $r = \sum_{\alpha} \alpha \alpha^* - \alpha^* \alpha$, the sum over all arrows $\alpha$.  
As long as $Q$ is not a Dynkin quiver, it is known that $A$ is Calabi-Yau of dimension~$2$.  This follows 
because the matrix Hilbert series of $A$ is as expected \cite[Theorem 3.4.1]{EE}, \cite[Lemma 7.6]{RR2}.

If one regrades $A$ by choosing all non-starred arrows to have degree $0$, then the relation 
$r$ is still homogeneous (now of degree $1$), so $A$ obtains a different grading where $A_0 \cong kQ$.  In case $Q$ has 
no oriented cycles, this regraded $A$ is still locally finite, and consequently it is still (graded) Calabi-Yau since this property
does not depend on the grading thanks to Theorem~\ref{thm:graded versus ungraded}.
%
%As long as $Q$ is not a Dynkin quiver, it is known that $A$ is Calabi-Yau of dimension~$2$.  This follows 
%because the matrix Hilbert series of $A$ is as expected \cite[Theorem 3.4.1]{EE}, \cite[Lemma 7.6]{RR2}.  
Thus for connected $Q$ which are not Dynkin and have no oriented cycles, the preprojective 
algebra $A$ regraded with non-starred arrows having degree $0$ gives a locally finite graded Calabi-Yau 
algebra with $A_0 = kQ$. 
\end{example}

  %%%%%%%%%%%%%%%%%%%%%%%

\bibliographystyle{amsplain}
\bibliography{twistedcyareasregular-arxiv-v2}

\providecommand{\bysame}{\leavevmode\hbox to3em{\hrulefill}\thinspace}
\providecommand{\MR}{\relax\ifhmode\unskip\space\fi MR }
\providecommand{\MRhref}[2]{%
  \href{http://www.ams.org/mathscinet-getitem?mr=#1}{#2}
}
\providecommand{\href}[2]{#2}
\begin{thebibliography}{10}

\bibitem{AmiotOppermann}
Claire Amiot and Steffen Oppermann, \emph{Higher preprojective algebras and
  stably {C}alabi-{Y}au properties}, Math. Res. Lett. \textbf{21} (2014),
  no.~4, 617--647.

\bibitem{AS}
Michael Artin and William~F. Schelter, \emph{Graded algebras of global
  dimension {$3$}}, Adv. in Math. \textbf{66} (1987), no.~2, 171--216.

\bibitem{Bell}
Jason~P. Bell, \emph{Noetherian algebras over algebraically closed fields}, J.
  Algebra \textbf{310} (2007), no.~1, 148--155.

\bibitem{Be}
Roland Berger, \emph{Dimension de {H}ochschild des alg\`ebres gradu\'ees}, C.
  R. Math. Acad. Sci. Paris \textbf{341} (2005), no.~10, 597--600.

\bibitem{Bo}
Raf Bocklandt, \emph{Graded {C}alabi {Y}au algebras of dimension 3}, J. Pure
  Appl. Algebra \textbf{212} (2008), no.~1, 14--32.

\bibitem{BSW}
Raf Bocklandt, Travis Schedler, and Michael Wemyss, \emph{Superpotentials and
  higher order derivations}, J. Pure Appl. Algebra \textbf{214} (2010), no.~9,
  1501--1522.

\bibitem{BourbakiVIII}
N.~Bourbaki, \emph{\'{E}l\'ements de math\'ematique. {A}lg\`ebre. {C}hapitre 8.
  {M}odules et anneaux semi-simples}, Springer, Berlin, 2012, Second revised
  edition of the 1958 edition.

\bibitem{Broomhead}
Nathan Broomhead, \emph{Dimer models and {C}alabi-{Y}au algebras}, Mem. Amer.
  Math. Soc. \textbf{215} (2012), no.~1011.

\bibitem{BG}
Ken~A. Brown and Ken~R. Goodearl, \emph{Lectures on algebraic quantum groups},
  Advanced Courses in Mathematics. CRM Barcelona, Birkh\"{a}user Verlag, Basel,
  2002.

\bibitem{BGMS}
Ragnar-Olaf Buchweitz, Edward~L. Green, Dag Madsen, and {\O}yvind Solberg,
  \emph{Finite {H}ochschild cohomology without finite global dimension}, Math.
  Res. Lett. \textbf{12} (2005), no.~5-6, 805--816.

\bibitem{CE}
Henri Cartan and Samuel Eilenberg, \emph{Homological {A}lgebra}, Princeton
  University Press, Princeton, N. J., 1956.

\bibitem{CR}
Charles~W. Curtis and Irving Reiner, \emph{Representation {T}heory of {F}inite
  {G}roups and {A}ssociative {A}lgebras}, Pure and Applied Mathematics, Vol.
  XI, Interscience Publishers, a division of John Wiley \& Sons, New
  York-London, 1962.

\bibitem{DI}
Frank DeMeyer and Edward Ingraham, \emph{Separable {A}lgebras {O}ver
  {C}ommutative {R}ings}, Lecture Notes in Mathematics, Vol. 181,
  Springer-Verlag, Berlin-New York, 1971.

\bibitem{Eilenberg}
Samuel Eilenberg, \emph{Homological dimension and syzygies}, Ann. of Math. (2)
  \textbf{64} (1956), 328--336.

\bibitem{EE}
Pavel Etingof and Ching-Hwa Eu, \emph{Koszulity and the {H}ilbert series of
  preprojective algebras}, Math. Res. Lett. \textbf{14} (2007), no.~4,
  589--596.

\bibitem{FD}
Benson Farb and R.~Keith Dennis, \emph{Noncommutative algebra}, Graduate Texts
  in Mathematics, vol. 144, Springer-Verlag, New York, 1993.

\bibitem{G}
Victor Ginzburg, \emph{Calabi-{Y}au algebras},
  \href{http://arxiv.org/abs/math/0612139}{arXiv:math/0612139 [math.AG]}, 2006
  preprint.

\bibitem{Hartshorne}
Robin Hartshorne, \emph{Algebraic {G}eometry}, Springer-Verlag, New
  York-Heidelberg, 1977, Graduate Texts in Mathematics, No. 52.

\bibitem{HIO}
Martin Herschend, Osamu Iyama, and Steffen Oppermann,
  \emph{{$n$}-representation infinite algebras}, Adv. Math. \textbf{252}
  (2014), 292--342.

\bibitem{Igusa}
Kiyoshi Igusa, \emph{Notes on the no loops conjecture}, J. Pure Appl. Algebra
  \textbf{69} (1990), no.~2, 161--176.

\bibitem{Keller}
Bernhard Keller, \emph{Calabi-{Y}au triangulated categories}, Trends in
  representation theory of algebras and related topics, EMS Ser. Congr. Rep.,
  Eur. Math. Soc., Z\"urich, 2008, pp.~467--489.

\bibitem{Keller:deformed}
\bysame, \emph{Deformed {C}alabi-{Y}au completions}, J. Reine Angew. Math.
  \textbf{654} (2011), 125--180, With an appendix by Michel Van den Bergh.

\bibitem{Kra}
U.~Kr{\"a}hmer, \emph{Poincar\'e duality in {H}ochschild (co)homology}, New
  techniques in {H}opf algebras and graded ring theory, K. Vlaam. Acad. Belgie
  Wet. Kunsten (KVAB), Brussels, 2007, pp.~117--125.

\bibitem{LMR}
T.~Y. Lam, \emph{Lectures on {M}odules and {R}ings}, Graduate Texts in
  Mathematics, vol. 189, Springer-Verlag, New York, 1999.

\bibitem{FC}
\bysame, \emph{A {F}irst {C}ourse in {N}oncommutative {R}ings}, second ed.,
  Graduate Texts in Mathematics, vol. 131, Springer-Verlag, New York, 2001.

\bibitem{EMR}
\bysame, \emph{Exercises in {M}odules and {R}ings}, Problem Books in
  Mathematics, Springer, New York, 2007.

\bibitem{Levasseur}
Thierry Levasseur, \emph{Some properties of noncommutative regular graded
  rings}, Glasgow Math. J. \textbf{34} (1992), no.~3, 277--300.

\bibitem{Li}
Huishi Li, \emph{Global dimension of graded local rings}, Comm. Algebra
  \textbf{24} (1996), no.~7, 2399--2405.

\bibitem{Lu}
Valery~A. Lunts, \emph{Categorical resolution of singularities}, J. Algebra
  \textbf{323} (2010), no.~10, 2977--3003.

\bibitem{MV}
Roberto Mart{\'{\i}}nez-Villa, \emph{Serre duality for generalized {A}uslander
  regular algebras}, Trends in the representation theory of finite-dimensional
  algebras ({S}eattle, {WA}, 1997), Contemp. Math., vol. 229, Amer. Math. Soc.,
  Providence, RI, 1998, pp.~237--263.

\bibitem{MVS}
Roberto Martin{\'e}z-Villa and {\O}yvind Solberg, \emph{Artin-{S}chelter
  regular algebras and categories}, J. Pure Appl. Algebra \textbf{215} (2011),
  no.~4, 546--565.

\bibitem{MR}
J.~C. McConnell and J.~C. Robson, \emph{Noncommutative {N}oetherian {R}ings},
  revised ed., Graduate Studies in Mathematics, vol.~30, American Mathematical
  Society, Providence, RI, 2001, With the cooperation of L. W. Small.

\bibitem{MM}
Hiroyuki Minamoto and Izuru Mori, \emph{The structure of {AS}-{G}orenstein
  algebras}, Adv. Math. \textbf{226} (2011), no.~5, 4061--4095.

\bibitem{NV1}
C.~N{\u{a}}st{\u{a}}sescu and F.~van Oystaeyen, \emph{Graded ring theory},
  North-Holland Mathematical Library, vol.~28, North-Holland Publishing Co.,
  Amsterdam-New York, 1982.

\bibitem{NV2}
Constantin N{\u{a}}st{\u{a}}sescu and Freddy Van~Oystaeyen, \emph{Methods of
  {G}raded {R}ings}, Lecture Notes in Mathematics, vol. 1836, Springer-Verlag,
  Berlin, 2004.

\bibitem{P}
Richard~S. Pierce, \emph{Associative {A}lgebras}, Graduate Texts in
  Mathematics, vol.~88, Springer-Verlag, New York-Berlin, 1982, Studies in the
  History of Modern Science, 9.

\bibitem{RRZ1}
Manuel Reyes, Daniel Rogalski, and James~J. Zhang, \emph{Skew {C}alabi-{Y}au
  algebras and homological identities}, Adv. Math. \textbf{264} (2014),
  308--354.

\bibitem{RRZ2}
\bysame, \emph{Skew {C}alabi-{Y}au triangulated categories and {F}robenius
  {E}xt-algebras}, Trans. Amer. Math. Soc. \textbf{369} (2017), no.~1,
  309--340.

\bibitem{R}
Manuel~L. Reyes, \emph{Noncommutative generalizations of theorems of {C}ohen
  and {K}aplansky}, Algebr. Represent. Theory \textbf{15} (2012), no.~5,
  933--975.

\bibitem{RR2}
Manuel~L. Reyes and Daniel Rogalski, \emph{Growth of graded twisted
  {C}alabi-{Y}au algebras}, J. Algebra \textbf{539} (2019), 201--259.

\bibitem{Rogalski}
Daniel Rogalski, \emph{Noncommutative projective geometry}, Noncommutative
  {A}lgebraic {G}eometry, Math. Sci. Res. Inst. Publ., vol.~64, Cambridge Univ.
  Press, New York, 2016, pp.~13--70.

\bibitem{Rotman}
Joseph~J. Rotman, \emph{An {I}ntroduction to {H}omological {A}lgebra}, second
  ed., Universitext, Springer, New York, 2009.

\bibitem{Schedler}
Travis Schedler, \emph{Deformations of algebras in noncommutative geometry},
  Noncommutative {A}lgebraic {G}eometry, Math. Sci. Res. Inst. Publ., vol.~64,
  Cambridge Univ. Press, New York, 2016, pp.~71--165.

\bibitem{Sierra}
Susan~J. Sierra, \emph{Rings graded equivalent to the {W}eyl algebra}, J.
  Algebra \textbf{321} (2009), no.~2, 495--531.

\bibitem{StacksProject}
The {Stacks Project Authors}, \emph{\itshape {S}tacks {P}roject},
  \url{http://stacks.math.columbia.edu}, 2018.

\bibitem{Ueyama}
Kenta Ueyama, \emph{Cluster tilting modules and noncommutative projective
  schemes}, Pacific J. Math. \textbf{289} (2017), no.~2, 449--468.

\bibitem{V}
Michel van~den Bergh, \emph{A relation between {H}ochschild homology and
  cohomology for {G}orenstein rings}, Proc. Amer. Math. Soc. \textbf{126}
  (1998), no.~5, 1345--1348.

\bibitem{V:erratum}
\bysame, \emph{Erratum to: ``{A} relation between {H}ochschild homology and
  cohomology for {G}orenstein rings'' [{P}roc. {A}mer. {M}ath. {S}oc. {\bf 126}
  (1998), no. 5, 1345--1348; {MR}1443171 (99m:16013)]}, Proc. Amer. Math. Soc.
  \textbf{130} (2002), no.~9, 2809--2810 (electronic).

\bibitem{We}
Charles~A. Weibel, \emph{An {I}ntroduction to {H}omological {A}lgebra},
  Cambridge Studies in Advanced Mathematics, vol.~38, Cambridge University
  Press, Cambridge, 1994.

\bibitem{Ye}
Amnon Yekutieli, \emph{Derived {C}ategories}, Cambridge Studies in Advanced
  Mathematics, vol. 183, Cambridge University Press, 2019.

\bibitem{YZ}
Amnon Yekutieli and James~J. Zhang, \emph{Homological transcendence degree},
  Proc. London Math. Soc. (3) \textbf{93} (2006), no.~1, 105--137.

\end{thebibliography}

%\end{comment}
\end{document}